\documentclass[12pt,a4paper,reqno]{amsart} %Fr o m detta
\usepackage{footnote}
\usepackage{amssymb}
\usepackage{latexsym}
\usepackage{amsmath}
\usepackage{mathrsfs}
\usepackage[X2,T1]{fontenc}
\usepackage[applemac]{inputenc}
\usepackage{cite}
\usepackage{calc}                   %Detta beh?vs f?r att
\newcommand{\scal}[2]{\langle #1,#2\rangle}
\newcommand{\rr}[1]{\mathbf R^{#1}}

\newcommand{\mabfr}{{\boldsymbol r}}

\newcommand{\cc}[1]{\mathbf C^{#1}}
\newcommand{\nm}[2]{\Vert #1\Vert _{#2}}

\newcommand{\sets}[2]{\{ \, #1\, ;\, #2\, \} }
\newcommand{\Sets}[2]{\left \{ \, #1\, ;\, #2\, \right \} }

\newcommand{\ep}{\varepsilon}
\newcommand{\fy}{\varphi}
\newcommand{\cdo}{\, \cdot \, }
\newcommand{\supp}{\operatorname{supp}}

\newcommand{\eabs}[1]{\langle #1\rangle}     %%%%%   for <x>

\newcommand{\vrum}{\vspace{0.1cm}}

\newcommand{\intrd}{\int _{\rd }}

\newcommand{\rd}{\mathbf{R} ^{d}}

\newcommand{\im}{i}

\newcommand{\nn}[1]{{\mathbf N}^{#1}}
\newcommand{\maclA}{\mathcal A}

\newcommand{\maclH}{\mathcal H}

\newcommand{\maclS}{\mathcal S}
\newcommand{\mascB}{\mathscr B}

\newcommand{\mascE}{\mathscr E}
\newcommand{\mascF}{\mathscr F}
\newcommand{\mascP}{\mathscr P}
\newcommand{\mascS}{\mathscr S}
\newcommand{\bsySig}{\boldsymbol \Sigma}
\newcommand{\bsycalA}{\boldsymbol {\mathcal A}}
\newcommand{\bsycalS}{\boldsymbol {\mathcal S}}

\numberwithin{equation}{section}          %Detta g?r att man f?r
\newtheorem{thm}{Theorem}
\numberwithin{thm}{section}

\newcommand{\rubrik}{}
\newtheorem{prop}[thm]{Proposition}
\newtheorem{cor}[thm]{Corollary}
\newtheorem{lemma}[thm]{Lemma}

\theoremstyle{definition}

\newtheorem{defn}[thm]{Definition}
\newtheorem{example}[thm]{Example}

\theoremstyle{remark}
\newtheorem{rem}[thm]{Remark}

\author{Joachim Toft}

\address{Department of Mathematics,
Linn{\ae}us University, V{\"a}xj{\"o}, Sweden}

\email{joachim.toft@lnu.se}

\title{Images of function and distribution spaces
under the Bargmann transform}



\begin{document}

%\begin{savenotes}

\begin{abstract}
We consider a broad family of test function spaces and their dual (distribution) space.
The family includes Gelfand-Shilov spaces, a family of test function spaces
introduced by S. Pilipovi{\'c}. We deduce different characterizations of such spaces,
especially under the Bargmann transform and the Short-time Fourier transform.
The family also include a test function space, whose dual space is mapped
by the Bargmann transform bijectively to the set of entire functions.
\end{abstract}

\maketitle

\par

%%%%%%%%%%%%%%%%%%%%%%%
\section{Introduction}\label{sec0}
%%%%%%%%%%%%%%%%%%%%%%%

\par

In the paper we consider a family of test function and distribution spaces,
including Gelfand-Shilov spaces. The family also include a family of functions and
distribution spaces,
named Pilipovi{\'c} test function and distribution spaces, respectively,
essentially introduced by S. Pilipovic in \cite{Pil2}. We show that those
spaces which are invariant under the Fourier transform may in convenient ways
be characterized by Hermite function expansions with simple conditions on the
coefficients. Furthermore we deduce that the Bargmann
transform maps these spaces into topological spaces of entire functions
or power series expansions, obeying
convenient estimates. At the same time, characterizations of these
function and distribution spaces in terms of suitable estimates of their short-time Fourier
transforms with Gaussian windows are deduced.

\par

We also
introduce a test function space, whose dual, or distribution space,
is mapped bijectively to the set of entire functions, by the Bargmann transform. 
Finally we introduce a broad family of modulation spaces as subsets
of the latter distribution space, and show that these modulation spaces
are quasi-Banach spaces. Here it is also shown that Gelfand-Shilov and
suitable Pilipovi{\'c}
spaces of functions and distributions are obtained by
suitable unions or intersections of this broad family of modulation spaces.

\par

One of the motivation for performing such investigations concerns questions on
global and detailed regularity properties for solutions to partial differential equations.
An other motivation is to find suitable test fundtion spaces which are dense and
smaller than the Schwartz space, $\mascS (\rr d)$, and which possess suitable
mapping properties under the Fourier transform. The corresponding distribution
spaces then become larger than the set of tempered distributions, $\mascS '(\rr
d)$, and the mapping properties for the Fourier transform carry over to these
distribution spaces. For these reasons, the function and distribution
spaces above might put convenient frames.

\par

In such contexts, the Gelfand-Shilov spaces, and spaces similar to them, are often feasible.
An important type of such spaces concern $\maclS _s(\rr d)$ ($\Sigma _s(\rr d)$), $s\ge 0$,
and consists of all smooth functions $f$ on $\rr d$ such that
\begin{equation}\label{EqGSDef}
\sup |x^\beta D^\alpha f(x)| \lesssim h^{|\alpha +\beta |}(\alpha !\beta !)^s
\end{equation}
for some (for every) $h>0$ (cf. \cite{GS}). (See \cite{Ho1} and Section \ref{sec1}
for notations.)

\par

Evidently, $\Sigma _s \subseteq \maclS _s$
and $\maclS _{s_0}\subseteq \mascS$ when $s_0\ge 0$. If in addition $s>\frac 12$ and
$s_0\ge \frac 12$, then these embeddings are dense, and therefore,
$\maclS _s'\subseteq \Sigma _s'$ and $\mascS '\subseteq \maclS _{s_0}'$ hold true for
corresponding distribution spaces. On the other hand, if $s\le \frac 12$ and
$s_0<\frac 12$, then $\Sigma _s$ and $\maclS _{s_0}$ are trivial, i.{\,}e.
they are equal to $\{ 0 \}$ (cf. \cite{GS}).

\par

There are also several other characterizations for Gelfand-Shilov spaces. For example
we have the following.

\par

\begin{prop}\label{Introprop}
Let $s\ge \frac 12$ ($s>\frac 12$), $f\in \mascS '(\rr d)$ and let
$c_{\alpha}(f)$ be the Hermite coefficient of order $\alpha$ of $f$. Then the following
conditions are equivalent:
\begin{enumerate}
\item $f\in \maclS _s(\rr d)$ ($f\in \Sigma _s(\rr d)$);

\vrum

\item $|f(x)|\lesssim e^{-r|x|^{\frac 1s}}$ and $|\widehat f(\xi )|\lesssim
e^{-r|\xi |^{\frac 1s}}$ for some (for every) $r>0$;

\vrum

\item $|c_\alpha (f)|\lesssim e^{-r|\alpha |^{\frac 1{2s}}}$ for some (for every)
$r>0$.
\end{enumerate}
\end{prop}

\par

Characterizations of the form given in Proposition \ref{Introprop} might be suitable
in applications and can be carried over to other situations. For example, in \cite{DaRu},
Dasgupta and Ruzhansky establish similar equivalences for functions
defined on compact Lie groups, after the Gelfand-Shilov spaces and the conditions
on $s$ are replaced by suitable Gevrey classes and $s>0$, respectively.

\par

The equivalence between (1) and (2) in the previous proposition were deduced
by Eijndhoven in \cite{Eij}, and later on extended by Chung, Chung and Kim in
\cite{ChuChuKim}. The equivalence between (1) and (3) can in an extended
form be found in e.{\,}g. \cite{LozPer}.

\medspace

The Pilipovi{\'c} space, $\bsycalS _{\! s}(\rr d)$ ($\bsySig _s(\rr d)$)\footnote{Throughout
the paper it is important to distinguish between the boldface 
characters, $\bsySig _*$, $\bsycalS _{\! *}$ etc. which
denote Pilipovi{\'c} spaces, and non-boldface characters,
$\Sigma _*$, $\maclS _*$ etc. which denote Gelfand-Shilov
spaces.}
of order $s\ge 0$, introduced in \cite{Pil2} in the case $s\ge \frac 12$, is
the set of all smooth functions $f$ on $\rr d$ such that
\begin{equation}\label{EqPilSpaceDef}
\nm {H^Nf}{L^\infty} \lesssim h^N(N!)^{2s},
\end{equation}
for some (for every) $h>0$, where $H=|x|^2-\Delta$ is the harmonic oscillator. That is,
when defining $\bsycalS _{\! s}(\rr d)$ and $\bsySig _s(\rr d)$, each $x_j$ and
$D_j$ in \eqref{EqGSDef}, in the definition of $\maclS _s(\rr d)$ and
$\Sigma _s(\rr d)$ should be replaced by $H^{\frac 12}$ in \eqref{EqPilSpaceDef}.

\par

It might not be surprising that $\bsycalS _{\! s}$ and $\bsySig _{s}$
are similar to $\maclS _s$ and $\Sigma _s$, respectively, when taking
into account the similarities in the definitions of these spaces.

\par

In fact, we have
\begin{alignat}{3}
\bsycalS _{\! s}(\rr d) &=\maclS _s(\rr d), &
\quad &\text{when}&\quad s &\ge \frac 12 ,\label{InclSpaces01}
\\[1ex] 
\bsySig _{s}(\rr d) &=\Sigma _s (\rr d), &
\quad &\text{when}&\quad  s &> \frac 12 \ \text{or}\ s=0,\label{InclSpaces02}
\\[1ex]
\bsycalS _{\! s}(\rr d) &\neq \{ 0 \} , \ \maclS _s(\rr d) = \{ 0 \} ,&
\quad &\text{when} & \quad s&<\frac 12, \label{InclSpaces03}
\intertext{and}
\bsySig _{s}(\rr d) &\neq \{ 0 \} ,\  \Sigma _s(\rr d) = \{ 0 \} , &
\quad &\text{when} &\quad s&\le \frac 12 ,\ s\neq 0.  \label{InclSpaces04}
\end{alignat}
Here \eqref{InclSpaces04} in the case $s=\frac 12$, \eqref{InclSpaces01}
and \eqref{InclSpaces02} was proved already in  \cite{Pil2}.
%it is proved that
%$\bsycalS _{\! s} = \maclS _s$ when $s\ge \frac 12$ and $\bsySig _{s}=\Sigma _s$
%when $s>\frac 12$, while
%$$
%\bsySig _{\frac 12}\neq \{ 0 \}
%\quad \text{and}\quad
%\Sigma _{\frac 12} = \{ 0 \} .
%$$
Furthermore, the arguments in \cite{Pil2} also show that \eqref{InclSpaces03}
and \eqref{InclSpaces04} hold for any admissible $s$,
since $\bsycalS _{\! s}(\rr d)$ and $\bsySig _{s}(\rr d)$ contain
all Hermite functions on $\rr d$ for such choices of $s$.

\par

The set $\bsySig (\rr d) =
\bsySig _{\frac 12}(\rr d)$ is of
peculiar interests, because it is the largest test function space among the sets
$\bsySig _{s}(\rr d)$ and $\bsycalS _{\! s}(\rr d)$ which at the same time is
\emph{not} a Gelfand-Shilov space.

\par

In Section \ref{sec2} we introduce a family of test function spaces,
denoted by $\maclH _s (\rr d)$ and $\maclH _{0,s}(\rr d)$, $s\ge 0$,
and contains all smooth functions $f$ on $\rr d$ such that
%
%which are defined by imposing suitable conditions of the involved
%coefficients in the Hermite series expansions of the involved functions.
%More precisely, it makes sense
%to let $\maclH _s (\rr d)$ ($\maclH _{0,s}(\rr d)$) be the set of all
%$f$ such that
its Hermite coefficients satisfy %the condition
$$
|c_\alpha (f)|\lesssim e^{-r|\alpha |^{\frac 1{2s}}}
$$
for some (for every) $r>0$. By Proposition \ref{Introprop}
it follows that $\maclH _s = \maclS _s$ when $s\ge \frac 12$
and $\maclH _{0,s} = \Sigma _s$ when $s>\frac 12$.
In Section \ref{sec4} we extend these relations into $\maclH _s =\bsycalS _{\! s}$
and $\maclH _{0,s}=\bsySig _s$ for every $s>0$.

\par

In Section \ref{sec2} we also consider test function spaces,
denoted by $\maclH _{\flat _\sigma} (\rr d)$ and $\maclH _{0,\flat _\sigma}
(\rr d)$, $\sigma >0$, where the assumptions on the Hermite coefficients above
are replaced by
$$
|c_\alpha (f)|\lesssim r^{|\alpha |}(\alpha !)^{-\frac 1{2\sigma}}.
$$
It follows that these spaces contains any $\bsycalS _{\! s}$ when $s<\frac 12$,
but are smaller and dense in $\bsySig (\rr d)$. Hence, their duals
$\maclH _{\flat _\sigma} '(\rr d)$ and $\maclH _{0,\flat _\sigma}'(\rr d)$,
respectively, contain $\bsySig '(\rr d)$.

\par

A motivations for considering these spaces is the convenient
mapping properties of the Bargmann transform on these
spaces. More precisely, in Section \ref{sec3} we deduce that any set of
entire functions
\begin{equation}\label{EntireFuncExpEst}
\begin{aligned}
\sets {F\in A(\cc d)}{|F(z)|\lesssim e^{r|z|^{\tau}} \ \text{for every}\ r >0},
\\[1ex]
\sets {F\in A(\cc d)}{|F(z)|\lesssim e^{r|z|^{\tau}} \ \text{for some}\ r >0},
\end{aligned}
\end{equation}
when $\tau >0$, as well as $A(\cc d)$, and the set of all functions which
are defined and analytic near origin in $\cc d$, are counter images of
the spaces
\begin{alignat*}{2}
&\maclH _{0,\flat _\sigma}(\rr d),&\quad &\maclH _{\flat _\sigma}(\rr d),
\quad \maclH _{0,\frac 12}(\rr d)
\\[1ex]
&\maclH _{0,\frac 12}'(\rr d),&\quad
&\maclH _{\flat _\sigma}'(\rr d)\quad
\text{and}\quad
\maclH _{0,\flat _\sigma}' (\rr d),
\end{alignat*}
for suitable choices of $\sigma$. (See Theorem \ref{flatSpacesChar} or the
tables in Section \ref{sec6}.)

\par

%As indicated above, an essential part of our analysis concerns of characterizing the spaces
%in \eqref{InclSpaces}--\eqref{InclSpaces2} by their images under
%the Bargmann transform $\mathfrak V_d$.
More precisely, beside
\begin{alignat}{5}
&\maclH _{0,s}(\rr d),& \quad &\maclH _s(\rr d),& \quad &\maclH _s'(\rr d), &
\quad
&\maclH _{0,s}'(\rr d),&\qquad s&\ge 0\ \text{or}\ s=\flat _\sigma ,
\label{Eq:HsSpaces}
\intertext{we also introduce}
&\bsycalA _{0,s}(\cc d),&\quad &\bsycalA _s(\cc d),&\quad &\bsycalA _s'(\cc d),&
\quad
&\bsycalA _{0,s}'(\cc d),&\qquad s&\ge 0\ \text{or}\ s=\flat _\sigma ,
\label{Eq:AsSpaces}
\end{alignat}
of suitable topological spaces of power series on $\cc d$, in Section \ref{sec2}.
From the definitions it follows that the Bargmann transform is homeomorphic
from the former spaces to the latter ones. Hence, the duals of
$\bsycalA _{0,s}(\cc d)$ and $\bsycalA _s(\cc d)$ can be identified with
$\bsycalA _{0,s}'(\cc d)$ and $\bsycalA _s'(\cc d)$, respectively.
In Sections \ref{sec3} and \ref{sec4} we show that the spaces 
in \eqref{Eq:AsSpaces} are convenient spaces of entire functions,
e.{\,}g. of the form \eqref{EntireFuncExpEst}.
A separate case appears when $s<\frac 12$, and in this situation the
estimates on $F$ in \eqref{EntireFuncExpEst} should be replaced by
$$
|F(z)|\lesssim e^{r(\log \eabs z ^{\frac 1{1-2s}}}
$$
(cf. \cite[Theorem 10]{FeGaTo2}).
%images of $\bsySig _s(\rr d)$ and $\bsycalS _{\! s}(\rr d)$ when $s<\frac 12$
%under the Bargmann transform are deduced.
We refer to the tables in the end of Section \ref{sec6} for a more comprehensive
explanation of these image properties.

\par

In the second part of Section \ref{sec3} we show that $\maclH _{\flat _1}'$
might be feasible when defining broad families of modulation spaces, with minimal
restrictions on the involved weight functions. In fact, it is here proved that
if the modulation spaces are defined in the background of $\maclH _{\flat _1}'$,
then these spaces are complete. In Section \ref{sec3} we
also deduce that the test function space $\maclS _C$, introduced by Gr{\"o}chenig
in \cite{Gc2}, agrees with $\maclH _{\flat _1}$.

%\par
%
%One of the motivations for considering these spaces is the convenient
%mapping properties of the Bargmann transform $\mathfrak V_d$ on these
%spaces. More precisely, let $A_a\{ 0\}$ be the set of all functions which are defined
%and analytic near origin in $\cc d$. In Section \ref{sec3} it is then deduced that
%%%
%\begin{align*}
%\mathfrak V_d(\maclH _{0,\flat _\sigma} (\rr d)) &= \sets {F\in A(\cc d)}{|F(z)|\lesssim
%e^{r|z|^{\frac {2\sigma }{\sigma +1}}} \ \text{for every}\ r >0},
%\\[1ex]
%\mathfrak V_d(\maclH _{\flat _\sigma} (\rr d)) &= \sets {F\in A(\cc d)}{|F(z)|\lesssim
%e^{r|z|^{\frac {2\sigma }{\sigma +1}}} \ \text{for some}\ r >0},
%\\[1ex]
%\mathfrak V_d(\maclH _{\flat _1} '(\rr d)) &= A(\cc d)
%\quad \text{and}\quad
%\mathfrak V_d(\maclH _{0,\flat _1}' (\rr d)) = A_d\{ 0\} .
%\intertext{If in addition $t>1$, then}
%\mathfrak V_d(\maclH _{\flat _\sigma}' (\rr d)) &= \sets {F\in A(\cc d)}{|F(z)|\lesssim
%e^{r|z|^{\frac {2\sigma }{\sigma -1}}} \ \text{for every}\ r >0},
%\intertext{and}
%\mathfrak V_d(\maclH _{0,\flat _\sigma}' (\rr d)) &= \sets {F\in A(\cc d)}{|F(z)|\lesssim
%e^{r|z|^{\frac {2\sigma }{\sigma -1}}} \ \text{for some}\ r >0},
%\end{align*}
%%%

\medspace

In the same manner we consider the Gelfand-Shilov spaces
$\maclS _{t_0}^{s_0}(\rr d)$ and $\Sigma _t^s(\rr d)$, and their distribution
spaces $(\maclS _{t_0}^{s_0})'(\rr d)$ and $(\Sigma _t^s)'(\rr d)$. These
spaces and the Pilipovi{\'c} spaces, are related to each others by the relations
\eqref{InclSpaces01}--\eqref{InclSpaces04}, and
\begin{multline}\label{InclSpaces}
\bsySig (\rr d) \subseteq \maclS _{t_0}^{s_0}(\rr d)
\subseteq \Sigma _t^s(\rr d) \subseteq \maclS _t^s(\rr d)
\subseteq \mascS (\rr d)
\\[1ex]
\subseteq \mascS '(\rr d) \subseteq
(\maclS _t^s)'(\rr d) \subseteq (\Sigma _t^s)'(\rr d)
\subseteq (\maclS _{t_0}^{s_0})'(\rr d)
\subseteq \bsySig '(\rr d),
\\[1ex]
\text{when}\quad
\frac 12 \le s_0<s \quad \text{and}\quad \frac 12 \le t_0<t,
\end{multline}
and
\begin{equation}\label{InclSpaces2}
\bsySig _{s_1}(\rr d) \subseteq \bsycalS _{\! s_1}(\rr d)\subseteq 
\bsySig _{s_2}(\rr d) 
\quad \text{when}  \quad s_1<s_2.
\end{equation}
%%
%
%%%
%\begin{alignat}{3}
%\bsySig _{s_1}(\rr d) &\subseteq \bsycalS _{\! s_1}(\rr d)\subseteq 
%\bsySig _{s_2}(\rr d) &
%\quad &\text{when} & \quad s_1&<s_2\label{InclSpaces2}
%\intertext{and}
%\bsycalS _s(\rr d) &\subseteq \maclH _{0,\flat}(\rr d)\subseteq \maclH
%_\flat (\rr d)\subseteq \bsySig (\rr d) &\quad &\text{when}
%& \quad s&<\frac 12 .\label{InclSpaces3}
%\end{alignat}
%%
%(See Section \ref{sec1} for definitions.)

\par

%In Sections \ref{sec3}--\ref{sec5} we also show that the spaces
%of power series can be identified with convenient spaces
%of entire functions on $\cc d$. In Section \ref{sec3} we identify $\bsycalA _{0,\flat}$
%$\bsycalA _\flat$ and their duals in such ways, and i
In Section \ref{sec5} we describe the images of $\Sigma _t^s$,
$\maclS _t^s$ and their duals under the Bargmann transform, for
suitable $s$ and $t$. (See also \cite{CapRodToft} for sub results.)
For example, since $\maclH _s(\rr d) =\maclS _s(\rr d)$ when
$s\ge \frac 12$, it follows from the image properties of the
Bargmann properties on Pilipovi{\'c} spaces that the mappings
\begin{equation}\label{EquivalencesBargGS}
\begin{alignedat}{4}
\mathfrak V_d \, &: &\, &\maclS _{s}(\rr d) & &\to & \,
&\bsycalA _{s}(\cc d)
\\[1ex]
\mathfrak V_d \, &: &\, &\maclS _{s}'(\rr d) & &\to & \,
&\bsycalA _{s}'(\cc d)
\end{alignedat}
\end{equation}
are bijective, which in terms of the the characterization results in Sections
\ref{sec3} and \ref{sec4} can be described as.
\begin{align}
\bsycalA _{s}(\cc d) &= \sets {F\in A(\cc d)}
{|F(z)|\lesssim e^{\frac {|z|^2}2-r|z|^{\frac 1s}},\ 
\text{for some $r>0$}}\label{EqAsIdenti}
\intertext{and}
\bsycalA _{s}'(\cc d) &= \sets {F\in A(\cc d)}
{|F(z)|\lesssim e^{\frac {|z|^2}2+r|z|^{\frac 1s}},\ 
\text{for every $r>0$}},\label{EqAsIdenti2}
\end{align}
%%
%when $s\ge \frac 12$.
for such choices of $s$.
In Section \ref{sec5} we, more generally, characterize all the
Gelfand-Shilov spaces of functions and distributions in \eqref{InclSpaces}
in such ways (cf. Theorem \ref{BargGSMapProp}).
%
%
%extend these relations
%Since $\maclH _s = \bsycalS _{\! s} = \maclS _s$ for $s\ge \frac 12$ and
%$\maclH _{0,s} = \bsySig _{s} = \Sigma _s$ for $s> \frac 12$, we obtain at the same time
%the images of Pilipovi{\'c} spaces with such choices of $s$.
%In fact, in Section \ref{sec5} we more generally
%explain the image of $\maclS _t^s(\rr d)$ for $s,t\ge \frac 12$ and
%$\Sigma _t^s(\rr d)$ under the Bargmann transform.
%
%%%
%\begin{equation}\label{EquivalencesBargGS}
%\begin{alignedat}{4}
%f &\in \maclS _{s_0}(\rr d) &
%\quad &\Longleftrightarrow & \quad
%|\mathfrak V_df(z)|&\lesssim e^{|z|^2/2-r|z|^{\frac 1s_0}}, &
%\quad &\text{for some $r>0$,}
%\\[1ex]
%f &\in \maclS _{s_0}'(\rr d)&
%\quad &\Longleftrightarrow & \quad
%|\mathfrak V_df(z)|&\lesssim e^{|z|^2/2+r|z|^{\frac 1s_0}}, &
%\quad &\text{for every $r>0$.}
%%
%%\mathfrak V_df &\in (\bsycalA _{t_0}^{s_0})'(\cc d).
%\end{alignedat}
%\end{equation}
%%%
%Moreover, we prove that the Bargmann transform
%is bijective from $\maclS _{t_0}^{s_0}(\rr d)$ to
%$\bsycalA _{t_0}^{s_0}(\cc d)$, and from
%$(\maclS _{t_0}^{s_0})'(\rr d)$ to
%$(\bsycalA _{t_0}^{s_0})'(\cc d)$.

\par

\par

Since short-time Fourier transforms with the standard Gaussian as window functions
carry over to corresponding Bargmann transforms by simple manipulations, the mapping
results for the Bargmann transform also lead to characterizations of Gelfand-Shilov
spaces, Pilipovi{\'c}
spaces and their duals in terms of estimates on the short-time Fourier transforms (see Remark
\ref{Rem:LinkSTFT} and Proposition \ref{STFTPilSpChar}). For example, let $V_\phi f$
be the short-time Fourier transform of $f$ with respect to the standard Gaussian function
$\phi$ as the window function. Then it follows from
\eqref{EqAsIdenti}, \eqref{EqAsIdenti2} and corresponding mapping properties for
$\bsySig _s(\rr d)$ and $\bsySig _s'(\rr d)$ that if $s=\frac 12$, then $f\in \maclS _s(\rr d)$
($f\in \bsySig _s(\rr d)$), if and only if 
$$
|V_\phi f(x,\xi )|\lesssim e^{-\frac {|x|^2+|\xi |^2}{4(1+r)}},
$$
for some (for every) $r>0$, and $f\in \maclS _s'(\rr d)$
($f\in \bsySig _s'(\rr d)$), if and only if
$$
|V_\phi f(x,\xi )|\lesssim e^{\frac {r(|x|^2+|\xi |^2)}{4}},
$$
for every (for some) $r>0$ (see Proposition \ref{STFTPilSpChar}).

\medspace

In Section \ref{sec6} we present some consequences and further remarks.
It is here proved that every Pilipovi{\'c} space are invariant under fractional
Fourier transforms. Here it is also shown that Pilipovi{\'c} spaces which
are \emph{not} Gelfand-Shilov spaces are neither invariant under dilations
nor algebras. Evidently, such properties are important when discussing problems
for partial differential equations, and it is expected that the absence of such
properties have bad impact on the applicability of such spaces. On the other hand,
in Section \ref{sec6} we also deduce that the Pilipovi{\'c} spaces are invariant
under actions of linear partial differential operators with polynomial coefficients.
Hence, it makes sense to discuss this class of operators in the framework of 
Pilipovi{\'c} spaces. Note that most of operators which appear in physics and
engineering belong to this class.

\par

In Section \ref{sec6} we also give examples on extension of the notion of Pilipovi{\'c}
spaces in terms of counter images of suitable spaces of entire functions under
the Bargmann transform. For example, we give a definition of spaces
$\bsySig _t^s(\rr d)$ and $\bsycalS _t^s(\rr d)$ when $s\ge 0$ or
$s=\flat _\sigma$ and $t\ge 0$ or $t=\flat _\sigma$ such that
$\bsycalS _t^s=\maclS _t^s$ when $s,t\ge \frac 12$, $\bsySig _t^s
=\Sigma _t^s$ when $s,t>\frac 12$, and $\bsySig _s^s=\bsySig _s$,
$\bsycalS _{\! s}^s = \bsycalS _{\! s}$ for general $s$.

\par

\section*{Acknowledgement}

I am very grateful to professors Stevan Pilipovi{\'c} and
Nenad Teofanov for careful reading of the paper, leading
to improvements of the content and the style. I am also grateful
to Yuanyuan Chen and Patrik Wahlberg for fruitful and
valuable discussions.

\par

%%%%%%%%%%%%%%%%%%%%%%%
\section{Preliminaries}\label{sec1}
%%%%%%%%%%%%%%%%%%%%%%%

\par

In this section we recall some basic facts. We start by discussing
Gelfand-Shilov spaces and their properties. Thereafter we recall
the definition of Pilipovi{\'c} spaces and some of their properties.
Then we recall some facts on modulation spaces. Finally
we recall the Bargmann transform and some of its
mapping properties, and introduce suitable classes of entire functions on
$\cc d$.

\par

\subsection{Gelfand-Shilov spaces}\label{subsec1.1}
We start by recalling some facts on Gelfand-Shilov spaces.
Let $0<h,s,t\in \mathbf R$ be fixed. Then $\mathcal S_{s,h}(\rr d)$
consists of all $f\in C^\infty (\rr d)$ such that
\begin{equation}\label{gfseminorm}
\nm f{\mathcal S_{t,h}^s}\equiv \sup \frac {|x^\beta \partial ^\alpha
f(x)|}{h^{|\alpha  + \beta |}\alpha !^s\, \beta !^t}
\end{equation}
is finite. Here the supremum should be taken over all $\alpha ,\beta \in
\mathbf N^d$ and $x\in \rr d$.

\par

Obviously $\mathcal S_{s,h}^t$ is a Banach space, contained in $\mascS$,
and which increases with $h$, $s$ and $t$ and 
$\mathcal S_{s,h}^t\hookrightarrow \mathscr S$. Here and
in what follows we use the notation $A\hookrightarrow B$ when the topological
spaces $A$ and $B$ satisfy $A\subseteq B$ with continuous embeddings.
Furthermore, if $s,t>\frac 12$, or $s =t=\frac 12$ and $h$
is sufficiently large, then $\maclS _{t,h}^s$ contains all finite linear
combinations of Hermite functions. Since such linear combinations
are dense in $\mathscr S$ and in $\maclS _{t,h}^s$, it follows
that the dual $(\mathcal S_{t,h}^s)'(\rr d)$ of $\mathcal S_{t,h}^s(\rr d)$ is
a Banach space which contains $\mathscr S'(\rr d)$, for such choices
of $s$ and $t$.

\par

The \emph{Gelfand-Shilov spaces} $\mathcal S_{t}^s(\rr d)$ and
$\Sigma _{t}^s(\rr d)$ are defined as the inductive and projective 
limits respectively of $\mathcal S_{t,h}^s(\rr d)$. This implies that
\begin{equation}\label{GSspacecond1}
\mathcal S_t^{s}(\rr d) = \bigcup _{h>0}\mathcal S_{t,h}^s(\rr d)
\quad \text{and}\quad \Sigma _t^{s}(\rr d) =\bigcap _{h>0}
\mathcal S_{t,h}^s(\rr d),
\end{equation}
and that the topology for $\mathcal S_t^{s}(\rr d)$ is the strongest
possible one such that the inclusion map from $\mathcal S_{t,h}^s
(\rr d)$ to $\mathcal S_t^{s}(\rr d)$ is continuous, for every choice 
of $h>0$. The space $\Sigma _t^s(\rr d)$ is a Fr{\'e}chet space
with seminorms $\nm \cdo{\mathcal S_{t,h}^s}$, $h>0$. Moreover,
$\Sigma _t^s(\rr d)\neq \{ 0\}$, if and only if $s+t\ge 1$ and
$(s,t)\neq (\frac 12,\frac 12)$, and $\maclS _t^s(\rr d)\neq \{ 0\}$, if and only
if $s+t\ge 1$.
%From now on we
%assume that the Gelfand-Shilov parameter pair $(s,t)$ are
%\emph{admissible}, or \emph{$\GS$-admissible},
%that is, $s+t\ge 1$ and $(s,t)\neq (\frac 12,\frac 12)$ when considering
%$\Sigma _t^s(\rr d)$, and $s+t\ge 1$ when considering
%$\maclS _t^s(\rr d)$

\medspace

The \emph{Gelfand-Shilov distribution spaces} $(\mathcal S_t^{s})'(\rr d)$
and $(\Sigma _t^s)'(\rr d)$ are the projective and inductive limit
respectively of $(\mathcal S_{t,h}^s)'(\rr d)$.  This means that
\begin{equation}\tag*{(\ref{GSspacecond1})$'$}
(\mathcal S_t^s)'(\rr d) = \bigcap _{h>0}(\mathcal S_{t,h}^s)'(\rr d)\quad
\text{and}\quad (\Sigma _t^s)'(\rr d) =\bigcup _{h>0}(\mathcal S_{t,h}^s)'(\rr d).
\end{equation}
We remark that in \cite{GS} it is proved that $(\mathcal S_t^s)'(\rr d)$
is the dual of $\mathcal S_t^s(\rr d)$, and $(\Sigma _t^s)'(\rr d)$
is the dual of $\Sigma _t^s(\rr d)$ (also in topological sense). For conveniency we
set
$$
\maclS _s=\maclS _s^s,\quad \maclS _s'=(\maclS _s^s)',\quad
\Sigma _s=\Sigma _s^s
\quad \text{and}\quad
\Sigma _s'=(\Sigma _s^s)'.
$$

\par

For every admissible $s,t>0$ and $\ep >0$ we have
\begin{equation}\label{GSembeddings}
\begin{alignedat}{2}
\Sigma _t^s (\rr d) &\hookrightarrow &
\maclS _t^s(\rr d) &\hookrightarrow  \Sigma _{t+\ep}^{s+\ep}(\rr d)
\\[1ex]
\quad \text{and}\quad
(\Sigma _{t+\ep}^{s+\ep})' (\rr d) &\hookrightarrow & (\maclS _t^s)'(\rr d)
&\hookrightarrow  (\Sigma _t^s)'(\rr d).
\end{alignedat}
\end{equation}

\par

A convenient basis for Gelfand-Shilov spaces on $\rr d$ concerns the set of
Hermite functions, $\{ h_\alpha \} _{\alpha \in \nn d}$ on $\rr d$, where
the Hermite function $h_\alpha$ of order $\alpha \in \nn d$ is defined by
$$
h_\alpha (x) = \pi ^{-\frac d4}(-1)^{|\alpha |}
(2^{|\alpha |}\alpha !)^{-\frac 12}e^{\frac {|x|^2}2}
(\partial ^\alpha e^{-|x|^2}).
$$
It follows that
$$
h_{\alpha}(x)=   ( (2\pi )^{\frac d2} \alpha ! )^{-1}
e^{-\frac {|x|^2}2}p_{\alpha}(x),
$$
for some polynomial $p_\alpha$ on $\rr d$, which is
called the Hermite polynomial of order $\alpha$.

\par

The set $\{ h_\alpha \} _{\alpha \in \nn d}$ is an
orthonormal basis for $L^2(\rr d)$. It is also a
basis for the Schwartz space and its distribution space,
and for any $\Sigma _t^s$ when $s,t>\frac 12$,
$\maclS _t^s$ when $s,t\ge \frac 12$ and their distribution
spaces. They are also eigenfunctions to the Harmonic
oscillator $H\equiv |x|^2-\Delta$. More precisely, we have
$$
Hh_\alpha = (2|\alpha |+d)h_\alpha ,\qquad H\equiv |x|^2-\Delta .
$$

\par

From now on we let $\mathscr F$ be the Fourier transform which
takes the form
$$
(\mathscr Ff)(\xi )= \widehat f(\xi ) \equiv (2\pi )^{-\frac d2}\int _{\rr
{d}} f(x)e^{-i\scal  x\xi }\, dx
$$
when $f\in L^1(\rr d)$. Here $\scal \cdo \cdo$ denotes the usual
scalar product on $\rr d$. The map $\mathscr F$ extends 
uniquely to homeomorphisms on $\mathscr S'(\rr d)$,
from $(\mathcal S_t^s)'(\rr d)$ to $(\mathcal S_s^t)'(\rr d)$ and
from $(\Sigma _t^s)'(\rr d)$ to $(\Sigma _s^t)'(\rr d)$. Furthermore,
$\mascF$ restricts to
homeomorphisms on $\mathscr S(\rr d)$, from
$\mathcal S_t^s(\rr d)$ to $\mathcal S_s^t(\rr d)$ and
from $\Sigma _t^s(\rr d)$ to $\Sigma _s^t(\rr d)$,
and to a unitary operator on $L^2(\rr d)$. Similar facts hold true
when $s=t$ and the Fourier transform is replaced by a partial
Fourier transform.

\par

Gelfand-Shilov spaces and their distribution spaces can in convenient
ways be characterized by means of estimates of short-time Fourier
transforms, (see e.{\,}g. \cite{GZ,To11}).
We here recall the details and start by giving the definition of
the short-time Fourier transform.

\par

Let $\phi \in \maclS _s '(\rr d)$ be fixed. Then the \emph{short-time
Fourier transform} $V_\phi f$ of $f\in \maclS _s '
(\rr d)$ with respect to the \emph{window function} $\phi$ is
the Gelfand-Shilov distribution on $\rr {2d}$, defined by
$$
V_\phi f(x,\xi ) \equiv  (\mascF _2 (U(f\otimes \phi )))(x,\xi ) =
\mascF (f \, \overline {\phi (\cdo -x)})(\xi
),
$$
where $(UF)(x,y)=F(y,y-x)$. Here $\mascF _2F$ is the partial Fourier transform
of $F(x,y)\in \maclS _s'(\rr {2d})$ with respect to the $y$ variable.
If $f ,\phi \in \maclS _s (\rr d)$, then it follows that
$$
V_\phi f(x,\xi ) = (2\pi )^{-\frac d2}\int f(y)\overline {\phi
(y-x)}e^{-i\scal y\xi}\, dy .
$$

\par

In the sequel, $A\lesssim B$ means that $A\le cB$ for a suitable
constant $c>0$. We also set $A\asymp B$ when $A\lesssim B$
and $B\lesssim A$.

\par

\begin{prop}\label{stftGelfand2}
Let $s,t,s_0,t_0>0$ be such that $s_0+t_0\ge 1$, 
$s_0\le s$ and $t_0\le t$. Also let
$\phi \in \mathcal S_{t_0}^{s_0}(\rr d)\setminus 0$ and
$f\in (\mathcal S_{t_0}^{s_0})'(\rr d)$.
Then the following is true:
\begin{enumerate}
\item $f\in  \mathcal S_{t}^s(\rr d)$, if and only if
\begin{equation}\label{stftexpest2}
|V_\phi f(x,\xi )| \lesssim  e^{-r (|x|^{\frac 1t}+|\xi |^{\frac 1s})},
\end{equation}
holds for some $r > 0$;

\vrum

\item if in addition $(s_0,t_0)\neq (\frac 12,\frac 12)$ and $\phi \in
\Sigma _{t_0}^{s_0}(\rr d)$, then
$f\in  \Sigma _{t}^{s}(\rr d)$, if and only if \eqref{stftexpest2}
holds for every $r > 0$.
\end{enumerate}
\end{prop}

\par

A proof of Theorem \ref{stftGelfand2} can be found in
e.{\,}g. \cite{GZ} (cf. \cite[Theorem 2.7]{GZ}). The
corresponding result for Gelfand-Shilov distributions
is the following improvement of \cite[Theorem 2.5]{To11}.

\par

\begin{prop}\label{stftGelfand2dist}
Let $s,t,s_0,t_0>0$ be such that $s_0+t_0\ge 1$, 
$s_0\le s$ and $t_0\le t$. Also let
$\phi \in \mathcal S_{t}^{s}(\rr d)\setminus 0$ and
$f\in (\mathcal S_{t_0}^{s_0})'(\rr d)$.
Then the following is true:
\begin{enumerate}
\item $f\in  (\mathcal S_{t}^s)'(\rr d)$, if and only if
\begin{equation}\label{stftexpest2}
|V_\phi f(x,\xi )| \lesssim  e^{r(|x|^{\frac 1t}+|\xi |^{\frac 1s})},
\end{equation}
holds for every $r > 0$;

\vrum

\item if in addition $(s_0,t_0)\neq (\frac 12,\frac 12)$ and $\phi \in
\Sigma _{t_0}^{s_0}(\rr d)$, then
$f\in  (\Sigma _{t}^{s})'(\rr d)$, if and only if \eqref{stftexpest2}
holds for some $r > 0$.
\end{enumerate}
\end{prop}

\par

We note that in (2) in \cite[Theorem 2.5]{To11} it should stay
$(\Sigma _{t}^s)'(\rr d)$ instead of $\Sigma _{t}^s(\rr d)$. In Section
\ref{sec5} we show that Proposition \ref{stftGelfand2dist} (1) remains
true also in the case $s=t=\frac 12$ and $\phi$ is a Gaussian.

\par

\begin{proof}
The assertion (2) is the same as (2) in \cite[Theorem 2.5]{To11}.

\par

Let $p\in [1,\infty ]$ and let $Q^p_r(\rr d)$ be the set of all $\phi \in L^p_{loc} (\rr d)$
such that
$$
\nm \phi {Q^p_r} \equiv \nm {\phi e^{r|\cdo |^{\frac 1t}}}{L^p}
+
\nm {\widehat \phi e^{r|\cdo |^{\frac 1s}}}{L^p}
$$
is finite. By \cite{Eij,ChuChuKim} it follows that $\maclS _t^s(\rr d)$ is the inductive
limit of $Q^p_r(\rr d)$ with respect to $r>0$, also in topological sense. Hence, if
$f\in (\maclS _t^s)'(\rr d)$, then
$$
|(f,\phi )| \lesssim \nm \phi {Q^p_r},\qquad \phi \in \maclS _t^s(\rr d),
$$
for every $r>0$.

\par

Now assume that $\phi \in \maclS _t^s(\rr d)$ is fixed. Then there is an $r_0>0$ such that
$\phi \in Q^p_\ep (\rr d)$ for every $\ep \in (0,r_0]$. Hence, for such $\ep$ we have
\begin{multline*}
|V_\phi f(x,\xi )| = |(f,\phi (\cdo -x)e^{i\scal \cdo \xi})|
\\[1ex]
\nm {\phi (\cdo -x)e^{\ep |\cdo |^{\frac 1t}}}{L^p}
+
\nm {\widehat \phi (\cdo -\xi )e^{\ep |\cdo |^{\frac 1s}}}{L^p}
\\[1ex]
\lesssim e^{\ep c|x|^{\frac 1t}} +e^{\ep c|\xi |^{\frac 1s}}
\lesssim e^{\ep c(|x|^{\frac 1t} +|\xi |^{\frac 1s})},
\end{multline*}
which implies that \eqref{stftexpest2} holds for every $r>0$.

\par

On the other hand, assume that \eqref{stftexpest2} holds for every $r>0$,
and let $\psi \in \maclS _{t_0}^{s_0}(\rr d)$. By Moyal's identity
$$
(f,\psi ) = \nm \phi {L^2}^{-2} (V_\phi f,V_\phi \psi ),
$$
Proposition \ref{stftGelfand2} and \eqref{stftexpest2}, and the fact that
$\maclS _{t_0}^{s_0}$ is dense in $\maclS _{t}^{s}$, it follows that the
$(f,\psi )$ extends uniquely to any $\psi \in \maclS _{t}^{s}(\rr d)$,
and that
$$
|(f,\psi )| \lesssim \nm \psi {Q^\infty _r},
$$
for every $r>0$. Hence, $f\in (\maclS _t^s)'(\rr d)$, which gives the result.
\end{proof}

\par

\begin{rem}\label{SchwFunctionSTFT}
The short-time Fourier transform can also be used to identify
elements in $\mascS (\rr d)$ and in $\mascS '(\rr d)$. In
fact, if $\phi \in \mascS (\rr d)\setminus 0$ and $f\in (\maclS
_{\frac 12})'(\rr d)$, then the following is true:
\begin{enumerate}
\item $f\in \mascS (\rr d)$, if and only if for every $N\ge 0$,
it holds
$$
|V_\phi f(x,\xi )| \lesssim \eabs {x,\xi }^{-N} \text ;
$$

\vrum

\item $f\in \mascS '(\rr d)$, if and only if for some $N\ge 0$,
it holds
$$
|V_\phi f(x,\xi )| \lesssim \eabs {x,\xi }^{N} \text ;
$$
\end{enumerate}
(Cf. \cite[Chapter 12]{Gc2}.)
\end{rem}

\par

%%% %
\subsection{The Pilipovi{\'c} spaces}\label{subsec1.2}
%%% %
Next we consider spaces which are obtained by suitable
estimates of Gelfand-Shilov or Gevrey type, after the operator $x^\beta \partial ^\alpha $
in \eqref{gfseminorm} is replaced by powers of the harmonic oscillator $H=|x|^2-\Delta$.
If $s_0\ge \frac 12$ and $s>\frac 12$, then it turns out that this is merely an alternative
approach for obtaining the Gelfand-Shilov spaces $\maclS _{s_0}(\rr d)$ and
$\Sigma _s(\rr d)$, respectively. (Cf. \cite{GrPiRo1,Pil1,Pil2}.)

\par

Let $h>0$, $s\ge 0$ and let $\bsycalS _{\! h,s}(\rr d)$ be the set of all
$f\in C^\infty (\rr d)$ such that
\begin{equation}\label{GFHarmCond}
\nm f{\bsycalS _{\! h,s}}\equiv \sup _{N\ge 0}
\frac {\nm{H^Nf}{L^\infty}}{h^N(N!)^{2s}}<\infty .
\end{equation}
Then $\bsycalS _{\! h,s}(\rr d)$ is a Banach space. If $s>0$, then
$\bsycalS _{\! h,s}(\rr d)$ contains all Hermite functions. Furthermore,
if $s=0$, and $\alpha \in \nn d$ satisfies $2|\alpha |+d\le h$,
then $h_\alpha \in \bsycalS _{\! h,s}(\rr d)$.

\par

We let
$$
\bsySig _s(\rr d) \equiv \bigcap _{h>0}\bsycalS _{\! h,s}(\rr d)
\quad \text{and}\quad
\bsycalS _{\! s}(\rr d) \equiv \bigcup _{h>0}\bsycalS _{\! h,s}(\rr d),
$$
and equip these spaces by projective and inductive limit topologies,
respectively, of $\bsycalS _{\! h,s}(\rr d)$, $h>0$.

\par

In \cite{Pil1,Pil2}, Pilipovi{\'c} proved
\begin{equation}\label{Eq:PilGSidentities}
\maclS _s(\rr d)=\bsycalS _{\! s}(\rr d),\quad s\ge \frac 12,
\quad \text{and}\quad
\Sigma _s(\rr d)=\bsySig _s(\rr d), \quad s> \frac 12.
\end{equation}
On the other hand, $\Sigma _s(\rr d)$ is trivially equal to $\{ 0 \}$ when
$s\le \frac 12$, while any Hermite function $h_\alpha$ fulfills
\eqref{GFHarmCond} for every $h>0$, when $0<s\le \frac 12$.

\par

Hence,
\begin{align*}
\bsySig _s(\rr d) &= \Sigma _s(\rr d), \qquad s > \frac 12
\intertext{and}
\bsySig _s(\rr d) &\neq \Sigma _s(\rr d)=\emptyset \qquad 0<s \le \frac 12.
\end{align*}
We call $\bsySig _s(\rr d)$ the \emph{Pilipovi{\'c} space (of Beurling type) of
order $s\ge 0$ on $\rr d$}. For conveniency we set $\bsySig (\rr d)\equiv
\bsySig _{\frac 12}(\rr d)$, and call this space the \emph{Pilipovi{\'c} space on $\rr d$}.

\par

Similarly, $\bsycalS _{\! s}(\rr d)$ is called the \emph{Pilipovi{\'c} space
(of Roumieu type) of order $s\ge 0$ on $\rr d$}. It follows that
$\bsycalS _{\! s}(\rr d)=\maclS _s(\rr d)$ when $s\ge \frac 12$, but 
$\emptyset =\maclS _s(\rr d) \neq \boldsymbol {\maclS} _{\! s}(\rr d)$ when
$0\le s<\frac 12$.

\par

%The right-hand side of \eqref{GFHarmCond} defines a semi-norm on
%$\mascS (\rr d)$ and we let $\bsySig _s(\rr d)$ and $\bsycalS _{\! s}(\rr d)$
%be equipped by the projective limit topology and inductive limit topology,
%respectively, through these semi-norms.
%
%\par

\begin{rem}\label{PilSpaceRem}
At first glance, the spaces $\bsySig _s(\rr d)$ for $s\le \frac 12$
and $\bsycalS _{\! s}(\rr d)$ for $s<\frac 12$ seem to
fit well in the family of Gelfand-Shilov spaces, and it is tempting
to put $\Sigma _s(\rr d)= \bsySig _s(\rr d)$ and $\maclS _s(\rr d)
=\bsycalS _{\! s}(\rr d)$ for such choices of $s$.
%equal to these spaces, instead
%of letting $s=t=\frac 12$ in \eqref{GSspacecond1} for the definition
%of $\Sigma _{\frac 12}(\rr d)$, leading to that $\Sigma
%_{\frac 12}(\rr d)$ is trivial.
Such approach is also justified by \eqref{GSembeddings} and the
embeddings
%
%Moreover, by straight-forward
%computations it follows that
$$
\bsycalS _{\! 0}(\rr d)\subseteq \bsySig _s (\rr d) \subseteq \bsycalS _{\! s}(\rr d)
\subseteq \bsySig _{s+\ep}(\rr d),\quad s,\ep >0,
$$
which follow by straight-forward computations.
%holds for every $s,\ep >0$, which fits well with respect to the
%embeddings \eqref{GSembeddings}.

\par

On the other hand, for such choices of $s$, it is for several
reasons appropriate to distinguish between $\bsySig _s (\rr d)$
and $\bsycalS _{\! s}(\rr d)$, and the trivial spaces $\Sigma _s(\rr d)$
and $\maclS _s(\rr d)$. In fact, the following holds true:
\begin{enumerate}
\item for every $s,t>0$, the spaces $\Sigma _t^s(\rr d)$ and
$\maclS _t^s(\rr d)$, and their duals, are invariant under the
dilation map $f\mapsto f(\lambda \cdo )$, for every $\lambda \in
\mathbf R \setminus 0$. On the other hand, $\bsySig _s (\rr d)$
when $s\le \frac 12$ and $\bsycalS _{\! s}(\rr d)$ when $s<\frac 12$, are
\emph{not} invariant under such dilations
in view of Corollary \ref{CorLackInvariance} in Section \ref{sec4}.

\vrum

\item Additionally to \eqref{GSembeddings} we also have
\begin{multline}\label{GSembeddings2}
\Sigma _{t_1}^{s_1}(\rr d)\subseteq \Sigma _{t_2}^{s_2}(\rr d)
\quad \text{and}\quad
\maclS _{t_1}^{s_1}(\rr d)\subseteq \maclS _{t_2}^{s_2}(\rr d)
\\[1ex]
\quad \text{when}\quad
0<s_1\le s_2,\ 0<t_1\le t_2.
\end{multline}

\par

Now let $\phi (x)=h_0(x)= \pi ^{-\frac d4}e^{-\frac {|x|^2}2}$. Then $\phi \in
\bsySig _{\frac 12} (\rr d)$. On the other hand, we have that
$f\in \Sigma _t^s(\rr d)$, if and only if
$$
|f(x)|\lesssim e^{-r |x|^{\frac 1t}}\quad \text{and}\quad |\widehat f (\xi )|
\lesssim e^{-r |\xi |^{\frac 1s}} ,
$$
for every $r>0$. (Cf. e.{\,}g. \cite{ChuChuKim, NiRo}.) This implies that
$\Sigma _t^s(\rr d)$ contains no Gauss functions when $s+t>1$, and
$s=\frac 12$ or $t=\frac 12$. In particular, $\phi \notin \Sigma _t^s(\rr d)$
for such choices of $s$ and $t$. Consequently, if we set $\Sigma _{\frac 12}(\rr d)$
equals to $\bsySig _{\frac 12} (\rr d)$, then \eqref{GSembeddings2} fails to
hold when $s_1=t_1=\frac 12$, $s_2+t_2>1$, and $s_2=s_1$ or $t_2=t_1$.
\end{enumerate}
%
%\par
%
%In the same way it is appropriate to distinguish between
%$\boldsymbol {\maclS} _{\! s}(\rr d)$ and $\maclS _s(\rr d)$,
%when $s<\frac 12$.
\end{rem}

\par

The dual spaces of $\bsycalS _{\! h,s}(\rr d)$, $\bsySig _s(\rr d)$
and $\bsycalS _{\! s}(\rr d)$ are denoted by $\bsycalS _{\! h,s}'(\rr d)$,
$\bsySig _s'(\rr d)$ and $\bsycalS _{\! s}'(\rr d)$, respectively.
In Section \ref{sec4} it is proved that the supremum norm in
\eqref{GFHarmCond} can be replaced by any $L^p$ norm when
defining $\bsySig _s(\rr d)$ and $\bsycalS _{\! s}(\rr d)$, and their
topologies (cf. Proposition \ref{NormEquiv}). As a
consequence,
$$
\bsySig _s'(\rr d) = \bigcup _{h>0} \bsycalS _{\! h,s}'(\rr d)
$$
when $s>0$ and
$$
\bsycalS _{\! s}'(\rr d) = \bigcap _{h>0} \bsycalS _{\! h,s}'(\rr d)
$$
when $s\ge 0$, with inductive respective projective limit topologies
of $\bsycalS _{\! h,s}'(\rr d)$, $h>0$.

\par

Since Gauss kernels are in background of the most essential parts of our
analysis, we restrict our considerations to involve the spaces
$$
\bsySig _{s_0+\ep }(\rr d),\quad \bsycalS _{\! s_0}(\rr d),
\quad \Sigma _{t+\ep}^{s+\ep}(\rr d),\quad \maclS _t^s(\rr d), \quad \mascS (\rr d),
$$
when $s_0\ge 0$, $s,t\ge \frac 12$ and $\ep >0$, as well as corresponding distribution spaces.
For the other situations, other types of analysis are needed in view of (2)
in Remark \ref{PilSpaceRem}, and is postponed for the future.

\par

\begin{rem}
We note that in \cite{GrPiRo1,GrPiRo2} it is proved that other
elliptic operators can be used to define $\Sigma _s$ when $s>\frac 12$,
and $\maclS _s$ when $s\ge \frac 12$, instead of the harmonic
oscillator. Furthermore, in \cite{GrPiRo2} it is also proved that
the spaces which correspond to $\bsySig$ are again non-trivial,
but in general different from $\bsySig$. The latter fact is also
a consequence of Corollary \ref{CorLackInvariance} in Section
\ref{sec6}. In fact, let $f$
be smooth and let $c>0$ be such that $c\neq 1$. Then $f$ fulfills
\eqref{GFHarmCond} for every $h>0$ after $H$ is replaced by
$c|x|^2-\Delta$, if and only if $x\mapsto f(c^{1/4}x)$
belongs to $\bsySig$. On the other hand, $\bsySig$ is not
invariant under dilations in view of Corollary
\ref{CorLackInvariance}.
\end{rem}

\par

%%% %
\subsection{Modulation spaces}\label{subsec1.3}
%%% %

\par

We start by discussing general properties on the involved weight
functions. A \emph{weight} on a Borel set $\Omega \subseteq \rr d$
is a positive function $\omega$ such that $\omega$ and $1/\omega$
belong to $L^\infty _{loc}(\Omega)$.

\par

%We especially consider the families of weights, given in \cite{To11}. We remark
%that some of these families are broad in the sense that they strictly contain
%all moderate weights. Here
We recall that if $\omega$ and $v$
are weights on $\rr d$, then $\omega$ is called
\emph{$v$-moderate}, or \emph{moderate}, if
\begin{equation}\label{moderate}
\omega (x+y)\lesssim \omega (x)v(y),\qquad x,y\in \rr d.
\end{equation}
(Cf. \cite{Gc2}.)
The set of moderate weights on $\rr d$ is denoted by $\mascP _E(\rr d)$. If
$v$ here above can be chosen as a polynomial, then $\omega$ is called
\emph{polynomially moderate}. The set of polynomially moderate weights
on $\rr d$ is denoted by $\mascP (\rr d)$.

\par

We note that $v$ above put limits of the size on $\omega$, since \eqref{moderate}
implies
$$
v(-x)^{-1}\lesssim \omega (x)\lesssim v(x).
$$

\par

For the larger families of weights in \cite{To11}, the condition \eqref{moderate}
is relaxed into
\begin{equation}\label{modrelax}
\omega (x)\lesssim \omega (x+y)\lesssim \omega (x) \quad
\text{when}\quad Rc \le |x|\le \frac c{|y|} ,\quad R\ge 2,
\end{equation}
or
\begin{multline}\tag*{(\ref{modrelax})$'$}
\omega (x)^2\lesssim \omega (x+y)\omega (x-y)\lesssim \omega (x)^2
\\[1ex]
\text{when}\quad
Rc \le |x|\le \frac c{|y|} ,\quad R\ge 2,
\end{multline}
Evidently, if \eqref{moderate} holds then \eqref{modrelax}
is true, and if \eqref{modrelax} holds then \eqref{modrelax}$'$ is true.

\par

We recall that is \eqref{modrelax}  fulfilled for the weights 
$$
\omega _1(x) = \eabs x^s\quad \text{and}\quad
\omega _2(x) = e^{r|x|^{\frac 1t}},
$$
when $r,s\in \mathbf R$ and $t\ge \frac 12$. On the other hand,
$\omega _2\notin \mascP _E(\rr d)$ when $t<1$.
We also recall that
\begin{equation}\label{Gaussest}
C^{-1}e^{-c|x|^2}\le \omega (x)\le Ce^{c|x|^2},
\end{equation}
holds for some positive constants $c$ and $C$ when $\omega$
satisfies \eqref{modrelax} or \eqref{modrelax}$'$, in view of
\cite[Proposition 2.13]{To11}.

\par

\begin{defn}\label{defweights}
Let $\omega$ be a weight on $\rr d$. 
\begin{enumerate}
\item $\mascP  _{Q}(\rr d)$ is the set of all
weights $\omega$ on $\rr d$ such that \eqref{modrelax}$'$
holds for some positive constants $c$ and $C$;

\vrum

\item $\mascP  _{Q}^0(\rr d)$ is the set of all
weights $\omega \in \mascP  _{Q}(\rr d)$ such that for every
$c>0$, there is a constant $C>0$ such that \eqref{Gaussest}
holds.

\vrum

\item The set $\Omega \subseteq \mascP _Q(\rr d)$
is called an \emph{admissible family of weights}, if
there is a rotation invariant function $0<\omega _0(x)\in
L^\infty _{loc}(\rr d)\cap L^1(\rr d)$ which
decreases with $|x|$ and such that
$$
\omega \cdot \omega _0\in \Omega \quad \text{and}\quad
\omega / \omega _0\in \Omega \quad \text{when}\ \omega \in \Omega .
$$
\end{enumerate}
\end{defn}

\par

\begin{example}\label{exadmweights}
Every family in Definition \ref{defweights} are admissible. Moreover,
if $\omega _0\in \mascP _{Q}(\rr d)$ and $\Omega$ is a family of
admissible weights, then
\begin{enumerate}
\item $\sets {\eabs \cdo ^N}{N\in \mathbf Z}$ is admissible;

\vrum

\item $\omega _0\cdot \Omega \equiv \sets {\omega _0\omega}
{\omega \in \Omega}$ is admissible.
\end{enumerate}
\end{example}

\par

%We shall now discuss modulation spaces and recall some basic properties.
In what follows 
we let $\mascB $ be a \emph{mixed quasi-norm space} on $\rr d$. This means that for some 
$p_1,\dots ,p_n\in (0,\infty ]$ and vector spaces
\begin{equation}\label{Vdirsum}
V_1,\dots ,V_n\subseteq \rr d\quad \text{such that}\quad
V_1\oplus \cdots \oplus V_n =\rr d,
\end{equation}
then $\mascB =\mascB _n$, where $\mascB _j$, $j=1,\dots ,n$ is inductively defined 
by
\begin{equation}\label{mixnormspacenorm}
\mascB _j =
\begin{cases}
L^{p_1}(V_1), & \ j=1
\\[1ex]
L^{p_j}(V_j; \mascB _{j-1}), &\ j=2,\dots , n.
\end{cases}
\end{equation}
The minimal exponent $\min (p_1,\dots ,p_n)$ is 
denoted by $\nu _1(\mascB )$, and the quasi-norm of 
$\mascB $ is given by $\nm f{\mascB } \equiv
\nm {F_{n-1}}{L^{p_n}(V_n)}$, where $F_0=f$ and
$$
F_j(x_{j+1},\dots ,x_n) =
\nm {F_{j-1}(\cdo ,x_{j+1},\dots ,x_n)}{L^{p_j}(V_j)},\quad j=1,\dots , n-1.
$$

\par

Let $p,q\in (0,\infty ]$, and $L^{p,q}(\rr {2d})$ % and its twisted space $L^{p,q}_{\tw}(\rr {2d})$
be the quasi-Banach spaces, which consist of all measurable $F$ on $\rr {2d}$ such that
$$
\nm F{L^{p,q}} \equiv \Big (\intrd \Big (\intrd |F(x,\xi
 )|^p\, dx\Big )^{q/p}\, d\xi \Big )^{1/q}<\infty \, .
$$
(with obvious modifications when $p=\infty$ or $q=\infty$).
It follows that $L^{p,q}(\rr {2d})$ % and $L^{p,q}_{\tw}(\rr {2d})$
is a mixed quasi-norm spaces.

\par

The definition of modulation spaces is given in the following.

\par

\begin{defn}\label{bfspaces2}
Let $\mascB $ be a mixed quasi-norm space on
$\rr {2d}$, $\omega \in \mascP _{Q}(\rr {2d})$, and let
\begin{equation}\label{phidef}
\phi (x)=\pi ^{-\frac d4}e^{-\frac {|x|^2}2}.
\end{equation}
Then the
\emph{modulation space}
$M(\omega ,\mascB )$ consists of all $f\in
\bsySig '(\rr d)$ such that
\begin{equation}\label{modnorm2}
\nm f{M(\omega ,\mascB )}
\equiv \nm {V_\phi f\, \omega }{\mascB }<\infty .
\end{equation}
\end{defn}

\par

We also set  $M^{p,q}_{(\omega )}(\rr d) = M(\omega ,L^{p,q}(\rr {2d}))$, and
$M^{p}_{(\omega )}=M^{p,p}_{(\omega )}$. Furthermore, if in addition $\omega =1$,
then we set $M^{p,q}=M^{p,q}_{(\omega )}$ and $M^p=M^{p}_{(\omega )}$.

\par

In \cite{To11}, it is assumed that $f$ in Definition \ref{bfspaces2} should belong
to $\maclS _{\frac 12}'(\rr d)$ instead of the larger class $\bsySig '(\rr d)$. This implies
that the modulation spaces might be larger in Definition \ref{bfspaces2} compared
to \cite{To11}. On the other hand, if $\omega$ belongs to 
$\mascP _Q^0(\rr {2d})$ and $f\in M(\omega ,\mascB )$, then $f\in
\maclS _{\frac 12}'(\rr d)$ by Theorem \ref{BargGSMapProp} in Section \ref{sec5}.

\par

If $\omega$ here above is a moderate weight, then it can be proved
that the condition \eqref{modnorm2} is independent of the choice of
$\phi \in \Sigma _1(\rr d)\setminus 0$, and that different choices of $\phi$
gives rise to equivalent norms (see e.{\,}g. Proposition 1.5 in \cite{Toft12}).
On the other hand, for general weights in
$\mascP _Q(\rr {2d})$, such invariance property seems to not yield.

\par

%%% %
\subsection{Spaces of entire functions and
the Bargmann transform}\label{subsec1.4}
%%% %
We shall now consider the Bargmann transform which is defined by the
formula
\begin{equation*}%\label{bargtransf}
(\mathfrak V_df)(z) =\pi ^{-\frac d4}\int _{\rr d}\exp \Big ( -\frac 12(\scal
z z+|y|^2)+2^{\frac 12}\scal zy \Big )f(y)\, dy,
\end{equation*}
when $f\in L^2(\rr d)$ (cf. \cite{B1}). We note that if $f\in
L^2(\rr d)$, then the Bargmann transform
$\mathfrak V_df$ of $f$ is the entire function on $\cc d$, given by
$$
(\mathfrak V_df)(z) =\int _{\rr d}\mathfrak A_d(z,y)f(y)\, dy,
$$
or
\begin{equation}\label{bargdistrform}
(\mathfrak V_df)(z) =\scal f{\mathfrak A_d(z,\cdo )},
\end{equation}
where the Bargmann kernel $\mathfrak A_d$ is given by
$$
\mathfrak A_d(z,y)=\pi ^{-\frac d4} \exp \Big ( -\frac 12(\scal
zz+|y|^2)+2^{\frac 12}\scal zy\Big ).
$$
Here
\begin{gather*}
\scal zw = \sum _{j=1}^dz_jw_j,\quad \text{when} \quad
z=(z_1,\dots ,z_d) \in \cc d
\\[1ex]
\text{and} \quad w=(w_1,\dots ,w_d)\in \cc d,
\end{gather*}
and otherwise $\scal \cdo \cdo $ denotes the duality between test
function spaces and their corresponding duals.
We note that the right-hand side in \eqref{bargdistrform} makes sense
when $f\in \maclS _{\frac 12}'(\rr d)$ and defines an element in $A(\cc d)$,
since $y\mapsto \mathfrak A_d(z,y)$ can be interpreted as an element
in $\maclS _{\frac 12} (\rr d)$ with values in $A(\cc d)$. Here and in what follows,
$A(\Omega )$ denotes the set of analytic functions on the open set
$\Omega \subseteq \cc d$.

\par

It was proved in \cite{B1} that $f\mapsto \mathfrak V_df$ is a bijective
and isometric map  from $L^2(\rr d)$ to the Hilbert space $A^2(\cc d)
\equiv B^2(\cc d)\cap A(\cc d)$, where $B^2(\cc d)$ consists of all
measurable functions $F$ on $\cc  d$ such that
\begin{equation}\label{A2norm}
\nm F{B^2}\equiv \Big ( \int _{\cc d}|F(z)|^2d\mu (z)  \Big )^{\frac 12}<\infty .
\end{equation}
Here $d\mu (z)=\pi ^{-d} e^{-|z|^2}\, d\lambda (z)$, where $d\lambda (z)$ is the
Lebesgue measure on $\cc d$. We recall that $A^2(\cc d)$ and $B^2(\cc d)$
are Hilbert spaces, where the scalar product are given by
\begin{equation}\label{A2scalar}
(F,G)_{B^2}\equiv  \int _{\cc d} F(z)\overline {G(z)}\, d\mu (z),
\quad F,G\in B^2(\cc d).
\end{equation}
If $F,G\in A^2(\cc d)$, then we set $\nm F{A^2}=\nm F{B^2}$
and $(F,G)_{A^2}=(F,G)_{B^2}$.

\par

Furthermore, Bargmann proved that there is a convenient reproducing. formula on
$A^2(\cc d)$. More precisely, let
\begin{equation}\label{reproducing}
(\Pi _AF)(z) \equiv \int _{\cc d}F(w)e^{(z,w)}\, d\mu (w),
\end{equation}
when $Fe^{R|\cdo |-|\cdo |^2}\in L^1(\cc d)$, for every $R\ge 0$. Here
\begin{gather*}
(z,w) = \sum _{j=1}^dz_j\overline{w_j},\quad \text{when} \quad
z=(z_1,\dots ,z_d) \in \cc d
\\[1ex]
\text{and} \quad w=(w_1,\dots ,w_d)\in \cc d,
\end{gather*}
is the scalar product of $z\in \cc d$ and $w\in \cc d$.
Then it is proved in \cite{B1,B2} that $\Pi _AF =F$ when
$F\in A^2(\cc d)$.

\medspace

In \cite{B1} it is also proved that
\begin{equation}\label{BargmannHermite}
\mathfrak V_dh_\alpha  = e_\alpha ,\quad \text{where}\quad
e_\alpha (z)\equiv \frac {z^\alpha}{\sqrt {\alpha !}},\quad z\in \cc d .
\end{equation}
In particular, the Bargmann transform maps the orthonormal basis
$\{ h_\alpha \}_{\alpha \in \nn d}$ in $L^2(\rr d)$ bijectively into the
orthonormal basis $\{ e_\alpha \}_{\alpha \in \nn d}$ of monomials
in $A^2(\cc d)$. Furthermore, if $f,g\in L^2(\rr d)$ and $F,G\in A^2(\cc d)$
are given by
\begin{equation}\label{FourSeries}
\begin{alignedat}{2}
f&=\sum _{\alpha \in \nn d} a_\alpha h_\alpha
,&\quad
g&=\sum _{\alpha \in \nn d} b_\alpha h_\alpha
\\[1ex]
F&=\sum _{\alpha \in \nn d} a_\alpha e_\alpha
,&\quad
G&=\sum _{\alpha \in \nn d} b_\alpha e_\alpha
\end{alignedat}
\end{equation}
then $F=\mathfrak V_df$, $G=\mathfrak V_dg$ and
\begin{equation}\label{Scalarproducts}
(f,g)_{L^2}=(F,G)_{A^2}=\sum _{\alpha \in \nn d}a_\alpha \overline {b_\alpha}
\end{equation}
Here and in what follows, $(\cdo ,\cdo )_{L^2(\rr d)}$ and
$(\cdo ,\cdo )_{A^2(\cc d)}$ denote the scalar products in
$L^2(\rr d)$ and $A^2(\cc d)$, respectively.

\medspace

Next we recall the link between the Bargmann transform
and the short-time Fourier transform
$f\mapsto V_\phi f$, when $\phi$ is given by \eqref{phidef}.
Let $S$ be the dilation operator given by
\begin{equation}\label{Sdef}
(SF)(x,\xi ) = F(2^{-\frac 12}x,-2^{-\frac 12}\xi ),
\end{equation}
when $F\in L^1_{loc}(\rr {2d})$. Then it
follows by straight-forward computations that
\begin{multline}\label{bargstft1}
(\mathfrak{V} _d f)(z)  =  (\mathfrak{V} _df)(x+\im \xi ) 
=  (2\pi )^{\frac d2}e^{\frac 12(|x|^2+|\xi|^2)}e^{-i\scal x\xi}V_\phi f(2^{\frac 12}x,-2^{\frac 12}\xi )
\\[1ex]
=(2\pi )^{\frac d2}e^{\frac 12(|x|^2+|\xi|^2)}e^{-i\scal x\xi}(S^{-1}(V_\phi f))(x,\xi ),
\end{multline}
or equivalently,
\begin{multline}\label{bargstft2}
V_\phi f(x,\xi )  =  
(2\pi )^{-\frac d2} e^{-\frac 14(|x|^2+|\xi |^2)}e^{-\im \scal x \xi /2}(\mathfrak{V} _df)
(2^{-\frac 12}x,-2^{-\frac 12}\xi).
\\[1ex]
=(2\pi )^{-\frac d2}e^{-i\scal x\xi /2}S(e^{-\frac {|\cdo |^2}2}(\mathfrak{V} _df))(x,\xi ).
\end{multline}
We observe that \eqref{bargstft1} and \eqref{bargstft2}
can be formulated into
\begin{equation*}%\label{BargStftlink}
\mathfrak V_d = U_{\mathfrak V}\circ V_\phi ,\quad \text{and}\quad
U_{\mathfrak V}^{-1} \circ \mathfrak V_d =  V_\phi ,
\end{equation*}
where $U_{\mathfrak V}$ is the linear, continuous and bijective operator on
$\mathscr D'(\rr {2d})\simeq \mathscr D'(\cc d)$, given by
\begin{equation}\label{UVdef}
(U_{\mathfrak V}F)(x,\xi ) = (2\pi )^{\frac d2} e^{\frac 12(|x|^2+|\xi |^2)}e^{-i\scal x\xi}
F(2^{\frac 12}x,-2^{\frac 12}\xi ) .
\end{equation}

\par

\begin{defn}\label{thespaces}
Let $\omega$ be a weight on $\rr {2d}$, $\mascB $ be a mixed quasi-norm
space on $\rr {2d}\simeq \cc d$, and let $r>0$ be such that $r\le \nu _1(\mascB )$.
\begin{enumerate}
\item The space $B(\omega ,\mascB )$ is the modified
weighted
$\mascB $-space which consists of all $F\in L^r_{loc}(\rr {2d})=
L^r_{loc}(\cc {d})$ such that
$$
\nm F{B(\omega ,\mascB )}\equiv \nm {(U_{\mathfrak V}^{-1}F)\omega }{\mascB }<\infty .
$$
Here $U_{\mathfrak V}$ is given by \eqref{UVdef};

\vrum

\item The space $A(\omega ,\mascB )$ consists of all $F\in A(\cc
d)\cap B(\omega ,\mascB )$ with topology inherited
from $B(\omega ,\mascB )$.
\end{enumerate}
\end{defn}

\par

We note that the spaces in Definition \ref{thespaces} are normed
spaces when $\nu _1(\mascB )\ge 1$.

\par

For convenience we set $\nm F{B(\omega ,\mascB )}=\infty$, when
$F\notin B(\omega ,\mascB )$ is measurable, and
$\nm F{A(\omega ,\mascB )}=\infty$, when $F\in A(\cc d)\setminus B(\omega ,\mascB )$.
We also set
$$
A^{p,q}_{(\omega )}(\cc d) = A(\omega ,L^{p,q}(\cc d) )
\quad \text{and}\quad
A^{p}_{(\omega )} = A^{p,p}_{(\omega )}.
$$

\par

The following result follows from Theorems 3.4 and 5.1 in
\cite{To11}, and justify the definition of the spaces in Definition
\ref{thespaces}.
\par

\begin{prop}\label{mainstep1prop}
Let $\omega \in \mascP  _{Q}^0(\rr {2d})$, $\mascB $ be a mixed
quasi-norm space on $\rr {2d}$ and let $\phi$ be as in
\eqref{phidef}. Then the following is true:
\begin{enumerate}
\item the map $\mathfrak V _d$ is an isometric bijection from
$M(\omega ,\mascB )$ to $A(\omega ,\mascB )$;

\vrum

\item the map $\Pi _A$ is a continuos projection from
$B(\omega ,\mascB )$ to $A(\omega ,\mascB )$.
\end{enumerate}
\end{prop}

\par

The case $\omega =1$ and $\mascB =L^2$ was proved already in
\cite {B1}. In Section \ref{sec3} we extend Proposition
\ref{mainstep1prop}{\,}(1) in such way that the only
assumption $\omega \in \mascP  _{Q}^0(\rr {2d})$ is 
that it should be a weight on $\rr {2d}$.

\par

Let $D_r(z_0)$ be the polydisc
$$
\sets {z=(z_1,\dots ,z_d)\in \cc d}{|z_j-z_{0,j}|<r_j,\ j=1,\dots ,d},
$$
with respect to
$$
z_0 =(z_{0,1},\dots ,z_{0,d}) \in \cc d,
\quad \text{and}\quad
r=(r_1,\dots ,r_d)\in \mathbf [0,\infty )^d,
$$
and let $A_d \{ z _0\}$ be the set of all functions which are analytic at $z_0$.
Then

$$
A(\cc d) = \bigcap _{r\in [0,\infty )^d} A(D_r(z)),\qquad
A_d \{ z_0 \}  =\bigcup _{r\in [0,\infty )^d} A(D_r(z_0)),%\qquad z\in \cc d.
$$

\par

The following result is a straight-forward consequence of Theorem 3.2 in \cite{To11}.

\par

\begin{prop}\label{AnSpacesEmb}
Let $p_1,p_2\in (0,\infty ]$, $s,t\ge \frac 12$, $r,r_1,r_2\in \mathbf R$
be such that $r_2<r_1$,
\begin{equation}\label{varthetarstDef}
\vartheta _{r,s,t}(z) = e^{r(|x|^{\frac 1t}+|\xi |^{\frac 1s})}, \quad
z=x+i\xi ,\ x,\xi \in \rr d,
\end{equation}
%%
%be as in \eqref{MstDef},
and let $\sigma _r(z)=\eabs z^r$.
Then
\begin{gather*}
A^{p_1}_{(\vartheta  _{r_1,s,t})}(\cc d)
\subseteq
A^{p_2}_{(\vartheta  _{r_2,s,t})}(\cc d),
\\[1ex]
A^{p_1}_{(\vartheta _{r,s,t}\sigma _N)}(\cc d)
\subseteq
A^{p_2}_{(\vartheta _{r,s,t})}(\cc d)
\subseteq
A^{p_1}_{(\vartheta _{r,s,t}\sigma _{-N})}(\cc d),
\\[1ex]
\nm F{A^{p_2}_{(\vartheta _{r_2,s,t})}}\lesssim \nm
F{A^{p_1}_{(\vartheta _{r_1,s,t})}}
\intertext{and}
\nm F{A^{p_1}_{(\vartheta _{r_2,s,t}\sigma _{-N})}}\lesssim \nm
F{A^{p_2}_{(\vartheta _{r_1,s,t})}}
\lesssim \nm F{A^{p_1}_{(\vartheta _{r_2,s,t}\sigma _{N})}},\qquad F\in A(\cc d),
\end{gather*}
provided $N\ge 0$ is chosen large enough.
\end{prop}

\par

\begin{proof}
Let $\Omega $ be the set of all weights $z\mapsto
\omega _{r_2,s,t}(z)\eabs z^h$, $h\in \mathbf R$.
Then $\Omega $ is an admissible family of weights.
%in the sense of \cite[Definition 1.4]{To11}.
By Theorem 3.2 in \cite{To11} we get
$$
\nm F{A^{p_2}_{(\omega _{s,t,r_2})}}\lesssim \nm
F{A^{p_1}_{(\omega _{s,t,r_2}\sigma _r)}},\qquad F\in A(\cc d),
$$
provided $r\ge 0$ is chosen large enough. Since
$\omega _{s,t,r_2}\sigma _N\lesssim \omega _{s,t,r_1}$, we get
$$
\nm F{A^{p_1}_{(\omega _{s,t,r_2}\sigma _r)}}\lesssim \nm F{A^{p_1}
_{(\omega _{s,t,r_1})}},\qquad F\in A(\cc d),
$$
and the first inclusion and the first inequality in the assertion
follow from these estimates. The other parts follow by similar arguments
and are left for the reader.
\end{proof}

\par

\begin{rem}\label{RemAnalSpaceTop}
Later on we consider consider spaces of entire functions of the forms
\begin{align}
\textstyle{\underset{r>0}\bigcup A^\infty _{(\vartheta _r)}(\cc d)}
&= \sets {F\in A(\cc d)}{|F(z)|\lesssim \omega _r(z)\ \text{for some}\ r>0}
\label{Eq:GenIndAnFuncSet}
\intertext{and}
\textstyle{\underset{r>0}\bigcap A^\infty _{(\vartheta _r)}(\cc d)}
&= \sets {F\in A(\cc d)}{|F(z)|\lesssim \omega _r(z)\ \text{for every}\ r>0}
\label{Eq:GenProjAnFuncSet}
\end{align}
for a suitable family $\{ \omega _r\} _{r>0}$ of weights on $\cc d$. Here
$\vartheta _r(z)=e^{-\frac {|z|^2}2}\omega _r(z)$. 
We let the topologies of the sets in \eqref{Eq:GenIndAnFuncSet}
and \eqref{Eq:GenProjAnFuncSet}
be the inductive limit and projective limit topologies of
$A^\infty _{(\vartheta _r)}(\cc d)$, respectively.
\end{rem}

\par

%%%%%%%%%%%%%%%%%%%%%%%%%%%%%%%%
\section{Hermite and power series expansions, and the Bargmann
transform}\label{sec2}
%%%%%%%%%%%%%%%%%%%%%%%%%%%%%%%%

\par

In this section we consider topological vector spaces of
Hermite series expansions and link them to topological vector spaces of
power series expansions through the Bargmann transform.
In the case when a Hilbert space of power series expansions
is equal to $A^2_{(\omega )}(\cc d)$ for suitable $\omega$, then
we deduce explicit formulas between involved weights for the set of
Hermite series expansions and the weight $\omega$.

\par

%In this section we consider Hilbert spaces, $\maclH ^2_{[\vartheta ]}(\rr d)$, of
%Hermite series expansions, parameterized by the weight function $\vartheta$,
%and link such spaces to Hilbert spaces, $\bsycalA ^2_{[\vartheta ]}(\cc d)$, of
%power series expansions, enabled by \eqref{BargmannHermite}.
%
%In the case the power series expansions in
%$\bsycalA ^2_{[\vartheta ]}(\cc d)$ converges absolutely everywhere, then
% 
% \par
%
%In a general situation, the Bargmann transform can be considered as
%a map from the set of formal Hermite series expansions on $\rr d$
%to the set of power series expansions on $\cc d$.

\subsection{Topological vector spaces of Hermite series and
power series expansions}\label{subsec2.1}

\par

The spaces of series expansions depend on parameters
in the extended sets
$$
{\textstyle{\mathbf R_\flat = \mathbf R_+ \bigcup \sets { \flat _\sigma}{\sigma >0} }}
\quad \text{and}\quad
{\textstyle{\overline {\mathbf R_\flat} = \mathbf R_\flat \bigcup \{ 0 \} }},
$$
of positive real numbers. For conveniency we also set $\flat _\infty =\frac 12$.
Beside the usual ordering in $\mathbf R$, the elements $\flat _\sigma$
in $\mathbf R_\flat$ and $\overline {\mathbf R_\flat}$ are ordered by
the relations $x_1<\flat _{\sigma _1}<\flat _{\sigma _2}<x_2$, when
$\sigma _1<\sigma _2$, and $x_1<\frac 12$ and $x_2\ge \frac 12$ are real.

\par

%Moreover, beside the usual multiplication rule in $\mathbf R$, the products
%$\flat _\sigma\cdot x = x\cdot \flat _\sigma$ and $\flat _{\sigma _1}\cdot
%\flat _{\sigma _2}$ are given by $\frac x2$ and $\frac 14$, respectively,
%when $x$ is real.

%
%It is convenient to consider
%the larger sets $\mathbf R_\flat$ and $\overline {\mathbf R_\flat}$ of
%the set of all positive real numbers, $\mathbf R_+$.
%
%\par
%
%\begin{defn}
%The sets $\mathbf R_\flat$ and $\overline {\mathbf R_\flat}$ are given by
%$$
%{\textstyle{\mathbf R_\flat = \mathbf R_+ \bigcup \sets { \flat _\sigma}{\sigma >0} }}
%\quad \text{and}\quad
%{\textstyle{\overline {\mathbf R_\flat} = \mathbf R_\flat \bigcup \{ 0 \} }}.
%$$
%Furthermore, $\flat _\infty =\frac 12$.
%
%\par
%
%Moreover,
%\begin{enumerate}
%\item beside the usual ordering in $\mathbf R$, the elements $\flat _\sigma$
%in $\mathbf R_\flat$ and $\overline {\mathbf R_\flat}$ are ordered by
%the relations $x_1<\flat _{\sigma _1}<\flat _{\sigma _2}<x_2$, when
%$\sigma _1<\sigma _2$, $x_1<\frac 12$ and $x_2\ge \frac 12$ are real;
%
%\vrum
%
%\item beside the usual multiplication rule in $\mathbf R$, the products
%$\flat _\sigma\cdot x = x\cdot \flat _\sigma$ and $\flat _{\sigma _1}\cdot
%\flat _{\sigma _2}$ are
%given by $\frac x2$ and $\frac 14$, respectively, when $x$ is real.
%\end{enumerate}
%\end{defn}

\par

\begin{defn}\label{DefSeqSpaces}
Let $p\in (0,\infty ]$, $r,t\in \mathbf R_+$, $s\in \mathbf R_\flat$, $\vartheta$
be a weight on $\nn d$, and let
\begin{equation}\label{varthetarsDef}
\vartheta _{r,s}(\alpha )\equiv
\begin{cases}
e^{r|\alpha |^{\frac 1{2s}}}, & \text{when}\quad s\in \mathbf R_+,
\\[1ex]
r^{|\alpha |}(\alpha !)^{\frac 1{2\sigma }}, & \text{when}\quad s = \flat _\sigma ,
\quad \qquad \alpha \in \nn d.
\end{cases}
\end{equation}
Then
\begin{enumerate}
\item $\ell _0' (\nn d)$ is the set of all sequences $\{c_\alpha \} _{\alpha \in \nn d}
\subseteq \mathbf C$ on $\nn d$;

\vrum

\item $\ell _{0,0}(\nn d)\equiv \{ 0\}$, and $\ell _0(\nn d)$ is the set of all sequences
$\{c_\alpha \} _{\alpha \in \nn d}\subseteq \mathbf C$ such that $c_\alpha \neq 0$
for at most finite numbers of $\alpha$;

\vrum

\item $\ell ^p_{[\vartheta ]}(\nn d)$ is the quasi-Banach space which consists of
all sequences $\{ c_\alpha \} _{\alpha \in \nn d} \subseteq \mathbf C$
such that
$$
\nm {\{ c_\alpha \} _{\alpha \in \nn d} }{\ell ^p_{[\vartheta ]}}\equiv
\nm {\{ c_\alpha \vartheta (\alpha )\} _{\alpha \in \nn d} }{\ell ^p}
$$
is finite;

\vrum

\item $\ell _{0,s}(\nn d)\equiv \underset {r>0}\bigcap \ell ^p_{[\vartheta _{r,s}]}(\nn d)$
and $\ell _s(\nn d)\equiv \underset {r>0}\bigcup \ell ^p_{[\vartheta _{r,s}]}(\nn d)$, with
projective respective inductive limit topologies of $\ell ^p_{[\vartheta _{r,s}]}(\nn d)$
with respect to $r>0$;

\vrum

\item $\ell _{0,s}'(\nn d)\equiv \underset {r>0}\bigcup \ell ^p_{[1/\vartheta _{r,s}]}(\nn d)$
and $\ell _s'(\nn d)\equiv \underset {r>0}\bigcap \ell ^p_{[1/\vartheta _{r,s}]}(\nn d)$, with
inductive respective projective limit topologies of $\ell ^p_{[1/\vartheta _{r,s}]}(\nn d)$
with respect to $r>0$.
%
%\vrum
%
%\item $\bsycalA _0'(\cc d)\equiv \mathbf C[[z_1,\dots ,z_d]]$ is
%the set of formal power series
%%%
%\begin{equation}\label{Powerseries}
%F(z)=\sum _{\alpha \in \nn d}c_\alpha e_\alpha (z) \quad \{c_\alpha \}
%_{\alpha \in \nn d}\in \ell _0' (\nn d),
%\end{equation}
%%%
%where $e_\alpha$ is given by \eqref{BargmannHermite};
%
%\vrum
%
%\item $\bsycalA _0(\cc d)$ is the set of all analytic polynomials
%on $\cc d$.
\end{enumerate}
\end{defn}

\par

%The main reason to consider the flat spaces $\maclH _\flat$, $\maclH _{0,\flat}$
%and their duals is that they possess convenient mapping properties under the
%Bargmann transform (see Section \ref{sec3}).
%
%\par

Let $p\in (0,\infty ]$, and let $\Omega _N$ be the set of all
$\alpha \in \nn d$ such that $|\alpha |\le N$. Then the
topology of $\ell _0(\nn d)$ is defined by the inductive
limit topology of the sets
$$
\Sets {\{ c_\alpha \} _{\alpha \in \nn d} \in \ell _0'(\nn d)}{c_\alpha =0\
\text{when}\ \alpha \neq \Omega _N}
$$
with respect to $N\ge 0$, and whose topology
is given through the quasi-norms
\begin{equation}\label{SemiNormsEllSpaces}
\{ c_\alpha \} _{\alpha \in \nn d}\mapsto \nm {\{ c_\alpha \}
_{|\alpha |\le N} }{\ell ^p(\Omega _N)},
\end{equation}
Since any two such quasi-norms on a finite-dimensional vector space are equivalent,
it follows that these topologies are independent of $p$.
Furthermore, $\ell _0' (\nn d)$ is a Fr{\'e}chet space and independent of $p$
when the topology is defined by the quasi-semi-norms
\eqref{SemiNormsEllSpaces}.

\par

Next we introduce spaces of formal Hermite series expansions
\begin{equation}\label{Hermiteseries}
f=\sum _{\alpha \in \nn d}c_\alpha h_\alpha ,\quad \{c_\alpha \}
_{\alpha \in \nn d}\in \ell _0' (\nn d),
\end{equation}
and spaces of formal power series expansions
\begin{equation}\label{Powerseries}
F=\sum _{\alpha \in \nn d}c_\alpha e_\alpha , \quad \{c_\alpha \}
_{\alpha \in \nn d}\in \ell _0' (\nn d),
\end{equation}
%%of functions and distributions which correspond to
which correspond to
\begin{equation}\label{ellSpaces}
\ell _{0,s}(\nn d),\quad \ell _s(\nn d),\quad \ell ^p_{[\vartheta ]}(\nn d),
\quad \ell _s'(\nn d)\quad \text{and}\quad \ell _{0,s}'(\nn d).
\end{equation}
Here $e_\alpha$ are given by \eqref{BargmannHermite}. For that
reason we consider the mappings
\begin{alignat}{2}
T_1 &: &\,  
\{ c_\alpha \} _{\alpha \in \nn d} &\mapsto \sum _{\alpha \in \nn d}
c_\alpha h_\alpha \label{T1Map}
\intertext{and}
%is bijective from $\ell _0(\nn d)$ to $\maclH _0(\rr d)$ and
%is bijective from $\ell _0' (\nn d)$ to $\maclH _0'(\rr d)$, and the map}
%
T_2 &: &\,  
\{ c_\alpha \} _{\alpha \in \nn d} &\mapsto \sum _{\alpha \in \nn d}
c_\alpha e_\alpha ,\label{T2Map}
\end{alignat}
between sequences and series expansions.

\par

\begin{defn}\label{DefclHclASpaces}
Let $p\in (0,\infty ]$, $s\in \overline {\mathbf R_\flat}$, $\vartheta$
be a weight on $\nn d$, and let $e_\alpha$ be given by
\eqref{BargmannHermite}.
\begin{itemize}
\item the images of $T_1$ in \eqref{T1Map} of the spaces in
\eqref{ellSpaces} are denoted by
\begin{equation}\label{clHSpaces}
\maclH _{0,s}(\rr d),\quad \maclH _s(\rr d),\quad \maclH ^p_{[\vartheta ]}(\rr d),
\quad \maclH _s'(\rr d)\quad \text{and}\quad \maclH _{0,s}'(\rr d),
\end{equation}
respectively. Furthermore, the topologies of the spaces in \eqref{clHSpaces}
are inherited from corresponding spaces in \eqref{ellSpaces}.

\vrum

\item the images of $T_2$ in \eqref{T2Map} of the spaces in
\eqref{ellSpaces} are denoted by
\begin{equation}\label{clASpaces}
\bsycalA _{0,s}(\cc d),\quad \bsycalA _{s}(\cc d),\quad \bsycalA ^p_{\! [\vartheta ]}(\cc d),
\quad \bsycalA _{s}'(\cc d)\quad \text{and}\quad \bsycalA _{0,s}'(\cc d),
\end{equation}
respectively. Furthermore, the topologies of the spaces in \eqref{clASpaces}
are inherited from corresponding spaces in \eqref{ellSpaces}.

\vrum

\item the quasi-norms $\nm f{\maclH ^p_{[\vartheta ]}}$ and
$\nm F{\bsycalA ^p_{[\vartheta ]}}$ of $f\in \maclH _0'(\rr d)$ and
$F\in \bsycalA _0'(\cc d)$, respectively, 
are given by $\nm {\{ c_\alpha \} _{\alpha \in \nn d}}{\ell ^p_{[\vartheta ]}}$,
when $f$ and $F$ are given by \eqref{Hermiteseries} and \eqref{Powerseries},
respectively.
\end{itemize}
\end{defn}

\par

%For $r\in \mathbf R$ and $s>0$ we also set
%$$
%\ell ^p_{r,s} = \ell ^p_{[\vartheta _{r,s}]},
%\quad
%\maclH ^p_{r,s} = \maclH ^p_{[\vartheta _{r,s}]}
%\quad \text{and}\quad
%\bsycalA ^p_{r,s} = \bsycalA ^p_{[\vartheta _{r,s}]},
%$$
%when
%%%
%\begin{equation}\label{varthetarsDef}
%\vartheta _{r,s}\equiv e^{r|\cdo  |^{\frac 1{2s}}},
%\end{equation}
%%%
%since such weights appear in several situation (cf. Definition \ref{DefSeqSpaces}).
%
%\par

For any $f\in \maclH _0'(\rr d)$ with expansion \eqref{Hermiteseries}, the coefficients
$c_\alpha$ are still called the Hermite coefficients for $f$ and are usually denoted
by $c_\alpha (f)$. Hence \eqref{Hermiteseries} takes the form
\begin{equation}\label{fHermite}
f=\sum _\alpha c_\alpha (f) h_\alpha ,\qquad c_\alpha (f) = (f,h_\alpha )_{L^2}.
\end{equation}

\par

Evidently, in \eqref{ellSpaces}, \eqref{clHSpaces}
and \eqref{clASpaces}, the largest spaces are
$$
\ell _0'(\nn d),\quad \maclH _0'(\rr d)
\quad \text{and}\quad
\bsycalA _0'(\cc d),
$$
respectively, and the smallest non-trivial spaces are
$$
\ell _0(\nn d),\quad \maclH _0(\rr d)
\quad \text{and}\quad
\bsycalA _0(\cc d),
$$
respectively. By the definitions it also follows that the following holds true.

\par

\begin{prop}\label{prop:BijSpaces}
Let $s\in \overline {\mathbf R_\flat}$, $p\in (0,\infty ]$ and let $\vartheta$
be a weight on $\nn d$. Then the following is true:
\begin{enumerate}
\item the map $T_1$ in \eqref{T1Map} is a homeomorphism from
$\ell _{0,s}(\nn d)$ to $\maclH _{0,s}(\rr d)$, from
$\ell _{s}(\nn d)$ to $\maclH _{s}(\rr d)$, from
$\ell _{s}'(\nn d)$ to $\maclH _{s}'(\rr d)$, and from
$\ell _{0,s}'(\nn d)$ to $\maclH _{0,s}'(\rr d)$;

\vrum

\item the map $T_2$ in \eqref{T2Map} is a homeomorphism from
$\ell _{0,s}(\nn d)$ to $\bsycalA _{0,s}(\cc d)$, from
$\ell _{s}(\nn d)$ to $\bsycalA _{s}(\cc d)$, from
$\ell _{s}'(\nn d)$ to $\bsycalA _{s}'(\cc d)$, and from
$\ell _{0,s}'(\nn d)$ to $\bsycalA _{0,s}'(\cc d)$;

\vrum

\item the mappings $T_1$ and $T_2$ in \eqref{T1Map} and \eqref{T2Map}
are isometric bijections from $\ell ^p_{[\vartheta ]}(\nn d)$ to
$\maclH ^p_{[\vartheta ]}(\rr d)$, and from $\ell ^p_{[\vartheta ]}(\nn d)$ to
$\bsycalA ^p_{[\vartheta ]}(\cc d)$, respectively.
\end{enumerate}
\end{prop}

\par

In the sequel we let the conjugate exponent $p'$ of $p\in (0,\infty ]$
be defined by
$$
p'=
\begin{cases}
\ \ 1, & p=\infty
\\[1ex]
\displaystyle{\frac p{p-1}}, & 1<p<\infty
\\[1ex]
\ \ \infty , & 0<p\le 1 .
\end{cases}
$$

\par

\begin{rem}
Let $s\in \overline {\mathbf R_\flat}$ and $p\in (0,\infty ]$ be the same as in
Definitions \ref{DefSeqSpaces} and \ref{DefclHclASpaces}. By
straight-forward computations it follows that the following is true:
\begin{itemize}
\item $\ell _{0,s}(\nn d)$,
$\ell _{s}(\nn d)$, $\ell _{0,s}'(\nn d)$, $\ell _s'(\nn d)$ and their topologies
are independent of $p\in (0,\infty ]$;

\vrum

\item the duals of $\ell _0(\nn d)$, $\ell _{0,s}(\nn d)$ and $\ell _{s}(\nn d)$
are equal to $\ell _0'(\nn d)$, $\ell _{0,s}'(\nn d)$ and $\ell _s'(\nn d)$,
respectively (also in topological sense), through unique extensions of the
$\ell ^2$-form on $\ell _0(\nn d)\times \ell _0(\nn d)$;

\vrum

\item if in addition $p\in [1,\infty )$, then the dual of
$\ell ^p_{[\vartheta ]}(\nn d)$ can be identified by $\ell ^{p'}_{[1/\vartheta ]}(\nn d)$
through a unique extension the $\ell ^2$-form on $\ell _0(\nn d) \times \ell _0(\nn d)$.
\end{itemize}
By Proposition \ref{prop:BijSpaces}, the same holds true if the spaces in
\eqref{ellSpaces} and the $\ell ^2$-form are replaced by the spaces in
\eqref{clHSpaces} and the $L^2$-form, respectively, or replaced by the
spaces in \eqref{clASpaces} and the $A^2$-form, respectively.

\par

In particular, if $f$ and $F$ are given by \eqref{Hermiteseries}
and \eqref{Powerseries}, and
$$
f_0=\sum _{\alpha \in \nn d}c_{0,\alpha}h_\alpha \in \maclH _0(\rr d)
\quad \text{and}\quad 
F_0=\sum _{\alpha \in \nn d}c_{0,\alpha}e_\alpha \in \bsycalA _0(\cc d),
$$
for some $\{ c_{0,\alpha }\} _{\alpha \in \nn d}\in \ell _0(\nn d)$, then
$$
(\{ c_{0,\alpha }\} _{\alpha \in \nn d} , \{ c_{\alpha }\} _{\alpha \in \nn d}) _{\ell ^2}
= (f_0,f)_{L^2}= (F_0,F)_{A^2} = \sum _{\alpha \in \nn d} c_{0,\alpha }
\overline{c_{\alpha }} .
$$
\end{rem}

\par

The Bargmann
transform $\mathfrak V_df$ of $f\in \maclH _0'(\rr d)$ is defined as the
right-hand side of \eqref{Powerseries} when $f$ is given by \eqref{Hermiteseries}.
That is,
$$
\mathfrak V_df = \sum _{\alpha \in \nn d}c_\alpha e_\alpha \in \bsycalA _0'(\cc d) 
\quad \text{when}\quad
f = \sum _{\alpha \in \nn d}c_\alpha h_\alpha \in \maclH _0'(\rr d)
$$
%It follows that $\mathfrak V_d$
%is a homeomorphism from $\maclH _0'(\rr d)$ to $\bsycalA _0'(\cc d)$,
%which restricts to a homeomorphism from $\maclH _0(\rr d)$ to
%$\bsycalA _0(\cc d)$.
%
%\par
%
In particular, Proposition \ref{prop:BijSpaces} shows that the mappings
$T_1$, $T_2$ and $\mathfrak V_d$ induce homeomorphisms,
still denoted by $T_1$, $T_2$ and $\mathfrak V_d$, respectively, between
Gelfand-tripples of suitable spaces in \eqref{DefclHclASpaces},
\eqref{clHSpaces} and \eqref{clASpaces}. More precisely,
the mappings
\begin{equation}\label{G-tripple1}
\begin{alignedat}{2}
\big ( \ell _s(\nn d), \ell ^2(\nn d),\ell _s'(\nn d) \big ) \quad & \overset
{\overset{\scriptstyle{T_1}}{}} {\rightarrow } \quad &
&\big ( \maclH _s(\rr d), L^2(\rr d),\maclH _s'(\rr d) \big )
\\[1ex]
\scriptstyle{T_2}\! &  \searrow \quad & & \phantom{\maclH _s(\rr d), L^2}
\downarrow \, \scriptstyle{\mathfrak V_d}
\\[1ex]
& & &\big ( \bsycalA _{s}(\cc d), A^2(\cc d),\bsycalA _{s}'(\cc d) \big )
\end{alignedat}
\end{equation}
are homeomorphisms for every $s\in \overline{\mathbf R_\flat}$.
%where $T_1$ and $T_2$ are the mappings in \eqref{T1Map} and \eqref{T2Map}.
%%
\begin{equation}\label{G-tripple2}
\begin{alignedat}{2}
\big ( \ell _{0,s}(\nn d), \ell ^2(\nn d),\ell _{0,s}'(\nn d) \big ) \quad & \overset
{\overset{\scriptstyle{T_1}}{}} {\rightarrow } \quad &
&\big ( \maclH _{0,s}(\rr d), L^2(\rr d),\maclH _{0,s}'(\rr d) \big )
\\[1ex]
\scriptstyle{T_2}\! &  \searrow \quad & & \phantom{\maclH _{0,s}(\rr d), L^2}
\downarrow \, \scriptstyle{\mathfrak V_d}
\\[1ex]
& & &\big ( \bsycalA _{0,s}(\cc d), A^2(\cc d),\bsycalA _{0,s}'(\cc d) \big )
\end{alignedat}
\end{equation}
are homeomorphisms for every $s\in \mathbf R_\flat$, and
\begin{equation}\label{G-tripple3}
\begin{alignedat}{2}
\big ( \ell _{[\vartheta]}^p(\nn d), \ell ^2(\nn d),\ell _{[1/\vartheta ]}^{p'}(\nn d) \big ) \quad & \overset
{\overset{\scriptstyle{T_1}}{}} {\rightarrow } \quad &
&\big ( \maclH _{[\vartheta ]}^p(\rr d), L^2(\rr d),\maclH _{[1/\vartheta ]}^{p'}(\rr d) \big )
\\[1ex]
\scriptstyle{T_2}\! &  \searrow \quad & & \phantom{\maclH _{[\vartheta ]}^p(\rr d), L^2}
\downarrow \, \scriptstyle{\mathfrak V_d}
\\[1ex]
& & &\big ( \bsycalA _{[\vartheta ]}^p(\cc d), A^2(\cc d),\bsycalA _{[1/\vartheta ]}^{p'}(\cc d) \big )
\end{alignedat}
\end{equation}
are isometric bijections for every weight $\vartheta$ on $\nn d$ and $p\in (0,\infty ]$. Note that
the latter triples do not necessarily need to be Gelfand-triples. On the other hand,
if more restrictive $p\in [1,\infty )$ and $\vartheta >c$ for some constant $c>0$,
then the triples in \eqref{G-tripple3} are Gelfand-triples.

\par

%%% %
\subsection{Identification properties between $\bsycalA _{[\vartheta ]}^2(\cc d)$
and $A_{(\omega )}(\cc d)$.}
%%% %

\par

Any formal power series \eqref{Powerseries} may be identified with
the function
$$
F(z) =\sum _{\alpha \in \nn d}c_\alpha \frac {z^\alpha }{\sqrt {\alpha !}},
$$
provided the series on the right-hand side converges in a neighborhood
of origin. Then $F(z)\in A_d\{ 0\}$. More precisely, the map
\begin{equation}\label{Eq:FormalPowerToAnalFunc}
\begin{aligned}
\sum _{\alpha \in \nn d}c_\alpha e_\alpha
\ &\mapsto \ 
\left (
z\mapsto
\sum _{\alpha \in \nn d}c_\alpha \frac {z^\alpha }{\sqrt {\alpha !}}
\right )
\\[1ex]
\maclH _0(\rr d) &\to A(\cc d)
\end{aligned}
\end{equation}
is injective and continuous. Since
two different power series give rise to two different functions in
$A_d\{ 0\}$, and that every function in $A_d\{ 0\}$ is equal to
its power series near origin, it follows that the map in
\eqref{Eq:FormalPowerToAnalFunc} extends to a bijective and continuous
map from the set of power series which are convergent near origin to $A _d\{ 0\}$.
From now on we often identify power series which converges near origin
with corresponding functions in $A_d\{ 0\}$.

\par

Next we
show that \eqref{Eq:FormalPowerToAnalFunc} extends to a bijective and isometric
map from $\bsycalA _{[\vartheta ]}^2(\cc d)$ to
$A^2_{(\omega )}(\cc d)$ and its norm, when $\vartheta$ and $\omega$
are given by
\begin{align}
\vartheta (\alpha )
&=
\left (  \frac {1}{\alpha !} \int _{\mathbf R_+^d} \omega _0(r)^2r^\alpha
\, dr  \right )^{\frac 12}\label{omegavarthetaRel}
\intertext{and}
\omega (z)
&=
e^{\frac {|z|^2}2}\omega _0(|z_1|^2,\dots ,|z_d|^2),\label{omega0omegaRel}
\end{align}
for some suitable weight $\omega _0$ on $\mathbf R_+^d$. Consequently,
the Bargmann transform is bijective and isometric from
$\maclH _{[\vartheta ]}^2(\rr d)$ to $A^2_{(\omega )}(\cc d)$ for such
choices of $\vartheta$ and $\omega$. See
also \cite[Theorem (4.1)]{JaEi} for related results in one dimension.

\par

\begin{thm}\label{cA2sCharExt}
Let $e_\alpha$ be as in \eqref{BargmannHermite}, $\alpha \in \nn d$,
and let $\omega _0$ be a positive measurable function on $\mathbf R_+^d$.
Also let $\vartheta$ and $\omega$ be weights on $\nn d$ and $\cc d$,
respectively, related to each others by \eqref{omegavarthetaRel} and
\eqref{omega0omegaRel}, and such that
\begin{equation}\label{varthetaCond}
\frac {r^{|\alpha |}}{(\alpha !)^{\frac 12}}\lesssim \vartheta (\alpha ),\quad
\alpha \in \nn d,
\end{equation}
holds for every $r>0$.
Then the map \eqref{Eq:FormalPowerToAnalFunc} extends uniquely to a
bijective and isometric map from $\bsycalA ^2_{[\vartheta ]}(\cc d)$ to
$A^2_{(\omega )}(\cc d)$. In particular, the following is true:
\begin{enumerate}
\item $A^2_{(\omega )}(\cc d)$ is a Hilbert space with
$\displaystyle{\left \{ \frac {e_\alpha}{\vartheta (\alpha )} \right \}
_{\alpha \in \nn d}}$ is an orthonormal basis;

\vrum

\item $\mathfrak V_d$ from $\maclH _0(\rr d)$ to $A(\cc d)$, extends
uniquely to a bijective and isometric map from $\maclH _{[\vartheta ]}^2(\rr d)$
to $A^2_{(\omega )}(\cc d)$, and
\begin{equation}\label{FepExpPre}
\nm {\mathfrak V_df}{A^2_{(\omega )}} = \left ( \sum _\alpha |c_\alpha (f)
\vartheta (\alpha )|^2 \right ) ^{\frac 12},
\end{equation}
when $c_\alpha (f)$ is given by \eqref{fHermite}.
\end{enumerate}
\end{thm}

\par

\begin{proof}
Let $F\in \bsycalA ^2_{[\vartheta ]}(\cc d)$. By \eqref{varthetaCond} it follows that
the power series expansion of $F$ is absolutely convergent for every $z\in \cc d$.
Hence, $F\in A(\cc d)$, which implies that
$\bsycalA ^2_{[\vartheta ]}(\cc d)\subseteq A(\cc d)$.
When proving that
$A^2_{(\omega )}(\cc d)=\bsycalA ^2_{[\vartheta ]}(\cc d)$, it suffices to
prove (1) and (2), in view of the definition of $\bsycalA _{[\vartheta ]}^2$ norm, since any entire function
is equal to its power series expansion.

\par

%First we consider $(e_\alpha ,e_\beta )_{A ^2_{(\omega )}}$.
We have
\begin{equation*}
(e_\alpha ,e_\beta )_{A ^2_{(\omega )}}
%= \pi ^{-d}\int _{\cc d} e_\alpha (z)\overline{e_\beta (z)}\omega _0(|z_1|^2,\dots ,|z_d|^2)|^2
%\, d\lambda (z)
%\\[1ex]
=
\pi ^{-d}\int _{\cc d}  \frac
{z^\alpha \overline z^\beta}{(\alpha !\beta !)^{\frac 12}}
\omega _0(|z_1|^2,\dots ,|z_d|^2)^2\, d\lambda (z).
\end{equation*}
Let
\begin{align*}
z = (r_1e^{i\fy _1},\dots ,r_de^{i\fy _d}),\quad \text{where}
\quad r &=(r_1,\dots ,r_d)\in [0,\infty )^d,
\\[1ex]
\text{and}\quad  \fy &= (\fy _1,\dots ,\fy _d)\in [0,2\pi )^d.
\end{align*}
Then the last integral becomes
$$
\pi ^{-d}\int _{\mathbf R_+^d} \int _{[0,2\pi)^d}  \frac
{r^{\alpha +\beta } e^{i\scal {\alpha -\beta }\fy}}
{(\alpha !\beta !)^{\frac 12}}
\omega _0(r_1^2,\dots ,r_d^2)^2\, r_1\cdots r_d\, d\fy dr.
$$
By the assumptions it follows that the last integral belongs to
$L^1([0,\infty )^d\times [0,2\pi)^d)$. Hence, by Fubbini's theorem
we get
\begin{multline*}
\pi ^{-d}\int _{\mathbf R_+^d} \int _{[0,2\pi)^d}  \frac
{r^{\alpha +\beta } e^{i\scal {\alpha -\beta }\fy}}
{(\alpha !\beta !)^{\frac 12}}
\omega _0(r_1^2,\dots ,r_d^2)^2\, r_1\cdots r_d\, d\fy dr
\\[1ex]
=\pi ^{-d}\int _{\mathbf R_+^d}  \left ( \int _{[0,2\pi)^d}
e^{i\scal {\alpha -\beta }\fy}\, d\fy  \right )
\frac {r^{\alpha +\beta } }{(\alpha !\beta !)^{\frac 12}}
\omega _0(r_1^2,\dots ,r_d^2)^2\, r_1\cdots r_d\, dr
\\[1ex]
=
2^{d}\delta _{\alpha ,\beta}\int _{\mathbf R_+^d} 
\frac {r^{2\alpha} }{\alpha !}
\omega _0(r_1^2,\dots ,r_d^2)\, r_1\cdots r_d\, dr
\\[1ex]
=
\delta _{\alpha ,\beta} \frac {1}{\alpha !} \int _{\mathbf R_+^d}
r^{\alpha} \omega _0(r)^2\, dr =\delta _{\alpha ,\beta} \vartheta (\alpha )^2.
\end{multline*}

\par

Hence, $\displaystyle{\left \{ \frac {e_\alpha}{\vartheta (\alpha )} \right \}
_{\alpha \in \nn d}}$ is an orthonormal sequence in
$A^2_{(\omega )}(\cc d)$, and
\begin{equation}\label{ParsIneq}
\sum _\alpha |c_\alpha \vartheta (\alpha )|^2 \le \nm F{A^2_{(\omega )}},
\end{equation}
when $F$ is given by \eqref{Powerseries}.

\par

For general $F\in A^2_{(\omega )}(\cc d)$ with Taylor series
\eqref{Powerseries}, we have
\begin{multline}\label{FExpComp}
\nm F{A^2_{(\omega )}}^2 = \pi ^{-d}\int _{\cc d}
|F(z)\omega _0(|z_1|^2,\dots ,|z_d|^2)|^2\, d\lambda (z)
\\[1ex]
=
\pi ^{-d} \int _{\rr d_+}\left ( \int _{[0,2\pi )^d} \Phi (r,\fy )
\, d\fy \right )\omega _0(r_1^2,\dots ,r_d^2)^2r_1\cdots r_d\, dr,
\end{multline}
where
$$
\Phi (r,\fy ) = 
%\left (  
\sum _{\alpha ,\beta \in \nn d} \frac {c_\alpha \overline {c_\beta}
r^{\alpha +\beta}e^{i\scal {\alpha -\beta}\fy}}{(\alpha !\beta !)^{\frac 12}}.
%\right )
$$
By \eqref{varthetaCond} and \eqref{ParsIneq}, it follows
%Since the left-hand side of \eqref{FExpComp} is finite, it follows
that the radius of convergence of $r_j\mapsto \Phi (r,\fy )$, $j=1,\dots ,d$,
is equal to $\infty$ (with uniform bounds with respect to
$\fy \in [0,2\pi )^d$). This gives
\begin{multline*}
\int _{[0,2\pi )^d} \left (  
\sum _{\alpha ,\beta \in \nn d} \frac {c_\alpha \overline {c_\beta}
r^{\alpha +\beta}e^{i\scal {\alpha -\beta}\fy}}{(\alpha !\beta !)^{\frac 12}}
\right ) \, d\fy
\\[1ex]
=
\sum _{\alpha ,\beta \in \nn d}\int _{[0,2\pi )^d}
\frac {c_\alpha \overline {c_\beta}
r^{\alpha +\beta}e^{i\scal {\alpha -\beta}\fy}}{(\alpha !\beta !)^{\frac 12}}
\, d\fy
=
(2\pi )^d\sum _{\alpha \in \nn d}
\frac {|c_\alpha |^2 r^{2\alpha}}{\alpha !},
\end{multline*}
and by \eqref{FExpComp} we get
\begin{multline*}%\label{FExpComp2}
\nm F{A^2_{(\omega )}}^2
=
2^d \int _{\rr d_+} \left (
\sum _{\alpha \in \nn d}
\frac {|c_\alpha |^2 r^{2\alpha}}{\alpha !}
\right )
\omega _0(r_1^2,\dots ,r_d^2)^2r_1\cdots r_d\, dr
\\[1ex]
=
\int _{\rr d_+} \left (
\sum _{\alpha \in \nn d}
\frac {|c_\alpha |^2 r^{\alpha}}{\alpha !}
\right )
\omega _0(r_1,\dots ,r_d)^2\, dr
\\[1ex]
=
\sum _{\alpha \in \nn d}
\frac {|c_\alpha |^2}{\alpha !}\int _{\rr d_+} r^{\alpha}
\omega _0(r_1,\dots ,r_d)^2\, dr
=
\sum _{\alpha \in \nn d}
|c_\alpha \vartheta (\alpha )|^2.
%
%=
%\nm F{\bsycalA ^2_{[\vartheta ]}}^2 .
\end{multline*}
Hence $F\in \bsycalA ^2_{[\theta ]}(\cc d)$ and $\nm F{A^2_{(\omega )}}
= \nm F{\bsycalA ^2_{[\theta ]}}$. This gives the result.
\end{proof}

\par

\begin{rem}
We note that for any $F\in A(\cc d)$, then
$F\in A^2_{(\omega )}(\cc d)$ for some choice of $\omega$.
This follows by letting
$$
\omega (r) =\Big (  \sup _{|w|=|r|}|F(w)|+1\Big )^{-1}e^{-|r|}.
$$
\end{rem}

\par

A common situation is that $\omega$ in Theorem \ref{cA2sCharExt}
is radial symmetric. In this case, %\eqref{omegavarthetaRel} and
\eqref{omega0omegaRel} takes the forms
\begin{align}
\omega (z)
&=
e^{\frac {|z|^2}2}\omega _0(|z|^2),\tag*{(\ref{omega0omegaRel})$'$}
\intertext{for some weight $\omega _0$ on $\overline{\mathbf
R_+}$, and in the following proposition we show that
$\vartheta (\alpha )$ becomes}
\vartheta (\alpha )
&=
\left (  \frac {1}{(|\alpha |+d-1)!} \int _0^\infty
\omega _0(r)^2r^{|\alpha | +d-1}\, dr  \right )^{\frac 12}.
\tag*{(\ref{omegavarthetaRel})$'$}
\end{align}
In particular, $\vartheta (\alpha)$
is also radial symmetric with respect to $\alpha$.

\par

\begin{prop}\label{varthetaCondRadSym}
Let $\omega _0$ be a weight on $\overline{\mathbf R_+}$, and let $\vartheta$ and
$\omega$ be weights on $\nn d$ and $\cc d$, respectively, related to each others
by \eqref{omegavarthetaRel}$'$ and \eqref{omega0omegaRel}$'$,
and such that \eqref{varthetaCond} holds for every $R>0$.
%
%the same as in Theorem \ref{cA2sCharExt}, and that in addition
%\eqref{omegaRotInv} holds for some weight $\omega _0$ on
%$\mathbf R_+\cup \{ 0\}$.
Then the map \eqref{Eq:FormalPowerToAnalFunc} extends uniquely
to a bijective and isometric map from $\bsycalA ^2_{[\vartheta ]}(\cc d)$
to $A^2_{(\omega )}(\cc d)$.
%, and
%%%
%\begin{equation}
%\end{equation}
%%%
\end{prop}

\par

\begin{proof}
Let
$$
\Delta _j = \sets {(t_1,\dots ,t_{d-j})\in [0,1]^{d-j}}{t_1+\cdots
+t_{d-j} \le 1},\quad j=1,\dots ,d-1.
$$
We take $r\in \mathbf R_+$ and $(t_1,\dots ,t_{d-1})\in \Delta _1$
as new variables of integration in 
\eqref{omegavarthetaRel}, where
$$
r_j = rt_j,\ \text{when}\ j\neq d
\qquad \text{and}\qquad
r_d = r\left (1-\sum _{j=1}^{d-1}t_j\right ).
$$
By straight-forward computations it follows that the Jacobian is
equal to $r^{d-1}$. Hence, \eqref{omegavarthetaRel} gives that
$
\vartheta (\alpha )^2 = I\cdot J,
$
where
\begin{align}
I &= \frac {1}{\alpha !} \int _{\mathbf R_+} \omega (r)^2r^{|\alpha |+d-1}\, dr
\label{IntegralIDef}
\intertext{and}
J &= \int _{\Delta _{1}} \left (\prod _{j=1}^{d-1} t_j^{\alpha _j}\right )
\left ( 1-\sum _{j=1}^{d-1}t_j   \right )^{\alpha _d}\, dt
\end{align}

\par

We set
\begin{align*}
J_k &= c_k
\idotsint _{\Delta _{k}} \left (\prod _{j=1}^{d-k} t_j^{\alpha _j}\right )
\left ( 1-\sum _{j=1}^{d-k}t_j   \right )^{\beta _k}\, dt_1\cdots dt_{d-k}
\intertext{where}
\beta _k &= \left ( \sum _{j=d-k+1}^d\alpha _{j}\right ) +k-1
\quad \text{and}\quad
c_k = (\beta _k!)^{-1}\left (\prod _{j=d-k+1}^d\alpha _{j}!\right ).
\end{align*}
We have $J=J_1$, and claim that $J=J_k$ for $k=1,\dots ,d-1$,
or equivalently, $J_k=J_{k+1}$, $k=1,\dots ,d-2$.

\par

In fact, let
\begin{align*}
y_k &= (t_1,\dots ,t_{d-k-1})\in \Delta _{k+1},\quad
\gamma _k = (\alpha _1,\dots ,\alpha _{d-k-1})
\intertext{and}
\fy (y_k) &= 1-\sum _{j=1}^{d-k-1}t_j.
\end{align*}
Then
\begin{align}
J_k &= c_k\int _{\Delta _{k+1}}y_k^{\gamma _k}h(y_k)\, dy_k,
\label{Jkidentity}
\intertext{where}
h(y_k) &= \int _0^{\fy (y_k)}u^{\alpha _{d-k}}(\fy (y_k)-u)^{\beta _k}\, du.
\notag
\end{align}

\par

By integrations by parts we get
\begin{multline*}
h(y_k) = \int _0^{\fy (y_k)}u^{\alpha _{d-k}}(\fy (y_k)-u)^{\beta _k}\, du
\\[1ex]
= \frac {\alpha _{d-k}!}{(\beta _k+1)\cdots (\beta _k+\alpha _{d-k})}
\int _0^{\fy (y_k)}(\fy (y_k)-u)^{\beta _k+\alpha _{d-k}}\, du
\\[1ex]
= \alpha _{d-k}!\beta _k!(\beta _k+\alpha _{d-k}+1)!)^{-1}
\fy (y_k)^{\beta _k+\alpha _{d-k}+1}
= \alpha _{d-k}!\beta _k!(\beta _{k+1}!)^{-1}\fy (y_k)^{\beta _{k+1}},
\end{multline*}
and inserting this into \eqref{Jkidentity} gives
$$
J_k = c_k\alpha _{d-k}!\beta _k!(\beta _{k+1}!)^{-1} \int _{\Delta _{k+1}}
y^{\gamma _k}\fy (y_k)^{\beta _{k+1}}\, dy_k = J_{k+1}.
$$
Hence, by integration by parts we obtain
\begin{multline*}
J=J_{d-1} = c_{d-1}\int _0^1t_1^{\alpha _1}(1-t_1)^{\beta _{d-1}}\, dt_1
\\[1ex]
=c_{d-1}\alpha _1! \beta _{d-1}!((\beta _{d-1}+\alpha _1)!)^{-1}
\int _0^1 (1-t_1)^{\beta _{d-1}+\alpha _1}\, dt_1
\\[1ex]
= c_{d-1}\alpha _1! \beta _{d-1}!((\beta _{d-1}+\alpha _1+1)!)^{-1}
=
\frac {\alpha !}{\beta _d!}.
\end{multline*}

\par

The result now follows by from the latter equalities, \eqref{IntegralIDef} and
the fact that $\vartheta (\alpha )^2=I\cdot J$.
\end{proof}

\par

A simple but interesting case appears when $\omega =\omega _h$ in
Proposition \ref{varthetaCondRadSym} is given by
\begin{equation}\label{defomegas}
\omega _h (z) = e^{\frac 12(1-2h)|z|^2}.
\end{equation}

\par

\begin{cor}\label{cA2sChar}\label{CorvarthetaCondRadSym}
Let $F\in A(\cc d)$ be given by \eqref{Powerseries}, $\omega _h$
be given by \eqref{defomegas}, and let $\vartheta _r(\alpha )=e^{r|\alpha |}$
when $r\in \mathbf R$ and $\alpha \in \nn d$. % , and let $\maclH _r^2
% =\maclH _{[\vartheta _r]}^2$. %, when $r>0$. Then
%%
\begin{equation}\label{FepExp}
\nm F{A^2_{(\omega _h)}} = e^{rd}
\left ( \sum _{\alpha} \left | c_\alpha  e^{r|\alpha |} \right | ^2\right )^{\frac 12},
\quad
r=-\frac {\log (2h)}2.
\end{equation}

\par

Furthermore, the map $\mathfrak V_d$ from $\maclH _0(\rr d)$ to $A(\cc d)$
extends uniquely to a homeomorphism from $\maclH _{[\vartheta _r]}^2(\rr d)$ to
$A^2_{(\omega _h)}(\cc d)$, and
\begin{equation}\label{FepExp2}
\nm {\mathfrak V_d f}{A^2_{(\omega _h)}} = e^{rd} \nm f{\maclH _{[\vartheta _r]}^2},
\quad r=-\frac {\log (2h)}2.
\end{equation}
\end{cor}

\par

\begin{proof}
Let $F\in A(\cc d)$, $\omega _0(r) = e^{-h(t_1+\cdots +t_d)}$ and let $\vartheta$
be as in Theorem \ref{cA2sCharExt}. Then
%Proposition \ref{varthetaCondRadSym} gives
%%
\begin{multline*}
\vartheta (\alpha )^2 = \frac {1}{\alpha !} \int _{\rr d_+}\omega _0(t)^2t^\alpha \, dt
%\\[1ex]
= \frac {1}{\alpha !} \int _{\rr d_+} e^{-2h(t_1+\cdots +t_d)}t^\alpha \, dt
\\[1ex]
=
%\left ( \frac {1}{2h} \right )^d
(2h)^{-|\alpha |-d}\frac 1{\alpha !}
\int _{\rr d_+} e^{-(t_1+\cdots +t_d)}r^\alpha \, dt
%\\[1ex]
=
%\left ( \frac {1}{2h} \right )^d
(2h)^{-|\alpha |-d}  %= e^{2rd} e^{2r|\alpha |}
= e^{2rd}\vartheta _r(\alpha )^2.
\end{multline*}
The result now follows from Theorem \ref{cA2sCharExt} or Proposition
\ref{varthetaCondRadSym}.
\end{proof}

\par

%%%%%%%%%%%%%%%%%%%%%%%%%%%%%%%%
\section{Characterizations of $\maclH _{\flat _\sigma} (\rr d)$ and
$\maclH _{0,\flat _\sigma}(\rr d)$, and their duals}\label{sec3}
%%%%%%%%%%%%%%%%%%%%%%%%%%%%%%%%

\par

In this section we consider
$\maclH _{\flat _\sigma} (\rr d)$, $\maclH _{0,\flat _\sigma}(\rr d)$,
$\maclH _{\flat _\sigma} '(\rr d)$ and $\maclH _{0,\flat _\sigma}'(\rr d)$, and their
images under the Bargmann transform.
We show that the Barmann transform is continuous and bijective
from $\maclH _{\flat _1} '(\rr d)$ to $A(\cc d)$, and from $\maclH
_{0,\flat _1} '(\rr d)$ to $A_d\{ 0\}$.
%, the set of functions which are analytic near origin in
%$\cc d$.
Furthermore, we prove that the images of $\maclH _{\flat
_\sigma} '(\rr d)$ and $\maclH _{0,\flat _\sigma}'(\rr d)$ for $\sigma >1$,
and $\maclH _{\flat _\sigma} (\rr d)$ and $\maclH _{0,\flat _\sigma}(\rr d)$ for any
$\sigma >0$ are sets of entire functions on $\cc d$, obeying suitable boundedness
conditions of exponential types. (Cf. Theorems
\ref{flatSpacesChar} and \ref{flatSpacesCharDual}.)

\par

We use these properties to extend the definition of modulation spaces
to permit any weight function, and show that these spaces are quasi-Banach
spaces, and thereby complete.

\par

In the last part of the section we analyze the test function space $\maclS _C(\rr d)$,
introduced by Gr{\"o}chenig in \cite{Gc2}, and show that this
space agrees with $\maclH _{0,\flat _1}(\rr d)$.

\medspace

The following results identify the spaces in \eqref{clASpaces} in the
case $s=\flat _\sigma$ with convenient topological spaces of analytic
functions. The first one deals with the images of test function spaces, and the
second one with corresponding distribution spaces. Here recall that
$\flat _\infty =\frac 12$, and set
\begin{equation}\label{Eq:kappaDef}
\kappa _1(\sigma )
=
\begin{cases}
\frac {2\sigma}{\sigma +1}, & 0\le \sigma <\infty
\\[1ex]
2, & \sigma =\infty   
\end{cases}
\quad \text{and}\quad
\kappa _2(\sigma )
=
\begin{cases}
\frac {2\sigma}{\sigma -1}, & 1< \sigma <\infty
\\[1ex]
2, & \sigma =\infty   .
\end{cases}
\end{equation}

\par

\begin{thm}\label{flatSpacesChar}
Let $\kappa _1$ be given by
\eqref{Eq:kappaDef}. Then the following is true: 
\begin{enumerate}
\item if $\sigma \in (0,\infty ]$, then the map
\eqref{Eq:FormalPowerToAnalFunc} extends uniquely
to a homeo\-morphism from $\bsycalA _{0,\flat _\sigma} (\cc d)$ to
\begin{equation}
\label{AflatIdentities1}
\sets {F\in A(\cc d)}{|F(z)|\lesssim
e^{r|z|^{\kappa _1(\sigma )}} \ \text{for every}\ r >0} \text ;
\end{equation}

\vrum

\item if $\sigma \in (0,\infty )$, then the map \eqref{Eq:FormalPowerToAnalFunc}
extends uniquely to a homeo\-morphism from $\bsycalA _{\flat _\sigma} (\cc d)$ to
\begin{equation}
\label{AflatIdentities2}
\sets {F\in A(\cc d)}{|F(z)|\lesssim
e^{r|z|^{\kappa _1(\sigma )}} \ \text{for some}\ r >0} .
\end{equation}
\end{enumerate}
%
%%%
%\begin{align}
%\bsycalA _{0,\flat _\sigma} (\cc d) &= \sets {F\in A(\cc d)}{|F(z)|\lesssim
%e^{r|z|^{\frac {2\sigma }{\sigma +1}}} \ \text{for every}\ r >0},\label{AflatIdentities1}
%\\[1ex]
%\bsycalA _{\flat _\sigma} (\cc d) &= \sets {F\in A(\cc d)}{|F(z)|\lesssim
%e^{r|z|^{\frac {2\sigma }{\sigma +1}}} \ \text{for some}\ r >0},\label{AflatIdentities2}
%\\[1ex]
%\bsycalA _{\flat _1} '(\cc d) &= A(\cc d)
%\quad \text{and}\quad
%\bsycalA _{0,\flat _1}'(\cc d) = A_d\{ 0\} .\label{AflatIdentities3}
%\intertext{If in addition $\sigma >1$, then}
%\bsycalA _{\flat _\sigma}' (\cc d) &= \sets {F\in A(\cc d)}{|F(z)|\lesssim
%e^{r|z|^{\frac {2\sigma }{\sigma -1}}} \ \text{for every}\ r >0},\label{AflatIdentities4}
%\intertext{and}
%\bsycalA _{0,\flat _\sigma}' (\cc d) &= \sets {F\in A(\cc d)}{|F(z)|\lesssim
%e^{r|z|^{\frac {2\sigma }{\sigma -1}}} \ \text{for some}\ r >0},\label{AflatIdentities5}
%\end{align}
%%%
\end{thm}

\par

\begin{thm}\label{flatSpacesCharDual}
Let $\kappa _2$ be given by
\eqref{Eq:kappaDef}. Then the following is true: 
\begin{enumerate}
\item the map
\eqref{Eq:FormalPowerToAnalFunc} extends
uniquely to a homeo\-morphism from $\bsycalA _{\flat _1}
'(\cc d)$ to $A(\cc d)$, and from $\bsycalA _{0,\flat _1}
'(\cc d)$ to $A_d\{ 0\}$;

\vrum

\item if $\sigma \in (1,\infty )$, then the map
\eqref{Eq:FormalPowerToAnalFunc} extends
uniquely to a homeo\-morphism from $\bsycalA _{\flat _\sigma}
'(\cc d)$ to
\begin{equation}
\label{AflatIdentities4}
\sets {F\in A(\cc d)}{|F(z)|\lesssim
e^{r|z|^{\kappa _2(\sigma )}} \ \text{for every}\ r >0} \text ;
\end{equation}

\vrum

\item if $\sigma \in (1,\infty ]$, then the map
\eqref{Eq:FormalPowerToAnalFunc} extends uniquely
to a homeo\-morphism from $\bsycalA _{0,\flat _\sigma} '(\cc d)$ to
\begin{equation}
\label{AflatIdentities5}
\sets {F\in A(\cc d)}{|F(z)|\lesssim
e^{r|z|^{\kappa _2(\sigma )}} \ \text{for some}\ r >0}.
\end{equation}
\end{enumerate}
\end{thm}

\par

We recall that Remark \ref{RemAnalSpaceTop} explains the topologies for
the sets in \eqref{AflatIdentities1}--\eqref{AflatIdentities5}.

\par

%\begin{rem}
%By using the convention
%$$
%\frac {2\sigma}{\sigma -1}=\frac {2\sigma}{\sigma +1}=\frac 12
%$$
%when $\sigma =\infty$, (1) in Theorem \ref{flatSpacesChar} and (3) in
%Theorem \ref{flatSpacesCharDual} are the same as the following:
%\begin{itemize}
%\item if $\sigma \in (0,\infty ]$, then $\mathfrak V_d$ 
%\end{itemize}
%\end{rem}
%
%\par

We need some preparations for the proof and start to make some
remarks on the topologies of
$A(\cc d)$ and $A_d\{ 0\}$. Let $r_0=(r_{0,1},\dots ,r_{0,d})\in [0,\infty )^d$ and
$F$ be defined in the open
neighbourhood $\Omega \subseteq D_{r_0}(0)\subseteq \cc d$ of
origin. Then $F$ is extendable to an analytic function in $A(D_r(0))$, if and only
if \eqref{Powerseries} holds when $z\in \Omega$
for some $\{ c_\alpha \} _{\alpha \in \nn d}$ which satisfies
\begin{equation}\label{Aseminorm}
\sup _\alpha \left | \frac {c_\alpha r^{\alpha }}
{(\alpha !)^{\frac 12}} \right | <\infty ,
\end{equation}
for every $r\in [0,\infty )^d$ such that $r<r_0$. Here
$(r_1,\dots ,r_d)<(r_{0,1},\dots ,r_{0,d})$ means that
$r_j<r_{0,j}$ for every $j=1,\dots ,d$.

\par

The map which takes the power series in \eqref{Powerseries} into
the left-hand side of \eqref{Aseminorm} defines a semi-norm on $A(D_{r_0}(0))$,
for every $r<r_0$. It follows that $A(D_{r_0}(0))$ is a Fr{\'e}chet space under
the topology defined by these semi-norms. Now we let the topologies in $A(\cc d)$ and
$A_d\{ 0\}$ be defined as the projective limit and inductive limit topology,
respectively, of $A(D_{r}(0))$, $r>0$.

\par

In the following propositions we let
\begin{equation}\label{omegaRspec}
\omega _{r,\sigma}(x,\xi ) = e^{\frac 14(|x|^2+|\xi |^2)}
e^{-r(|x|^{\kappa _1(\sigma )}+|\xi |^{\kappa _1(\sigma )})},
\end{equation}

\par

\begin{prop}\label{calJChar}
Let $p\in (0,\infty ]$, $\sigma \in (0,\infty )$, $\kappa _1$ be as in
\eqref{Eq:kappaDef}, $\mascB$ be a mixed quasi-norm space, $\phi$
be as in \eqref{phidef}, $f\in \maclS _{\frac 12}'(\rr d)$, and let $F=\mathfrak V_df$.
Also let $\omega _{r,\sigma}$ be as in \eqref{omegaRspec}.
Then the following conditions are equivalent:
\begin{enumerate}
\item $f\in \maclH _{\flat _\sigma} (\rr d)$;

\vrum

\item $\nm {F\cdot e^{-r|\cdo |^{\kappa _1(\sigma )}}}{L^p}<\infty$ for some $r>0$;

\vrum

\item $f\in M(\omega _{r,\sigma},\mascB )$ for some $r>0$.
\end{enumerate}
\end{prop}

\par

\begin{prop}\label{calJ0Char}
Let $p\in (0,\infty ]$, $\sigma \in (0,\infty ]$, $\kappa _1$ be as in
\eqref{Eq:kappaDef}, $\mascB$ be a mixed quasi-norm space, $\phi$
be as in \eqref{phidef}, $f\in \maclS _{\frac 12}'(\rr d)$, and let $F=\mathfrak V_df$.
Also let $\omega _{r,\sigma}$ be as in \eqref{omegaRspec}.
Then the following conditions are equivalent:
\begin{enumerate}
\item $f\in \maclH _{0,\flat _\sigma}(\rr d)$;

\vrum

\item $\nm {F\cdot e^{-r|\cdo |^{\kappa _1(\sigma )}}}{L^p}<\infty$
for every $r>0$;

\vrum

\item $f\in M(\omega _{r,\sigma},\mascB )$ for every $r>0$.
\end{enumerate}
\end{prop}

\par

%We only prove Proposition \ref{calJChar}. Proposition \ref{calJ0Char} follows by
%similar arguments and is left for the reader.
%
%\par

\begin{proof}[Proof of Propositions \ref{calJChar} and \ref{calJ0Char}]
If $\mascB =L^p$, then (2) is equivalent to (3), in view of the definitions. Since
$\{ \omega _{R,\sigma} \} _{R>0}$ is an admissible family of weights, Theorem 3.4 in
\cite{To11} shows that (2) and (3) are equivalent for arbitrary $\mascB$.
We need to prove that (2) is equivalent to (1), and then we first assume that (1)
Proposition \ref{calJChar} holds.

\par

Let $C$ and $R$ be such that
$$
|c_\alpha (f) |\le  CR^{|\alpha |}(\alpha !)^{-\frac 1{2\sigma}},
$$
and let $r=|z|$ and $\rho =\frac 1{\kappa _1(\sigma )}$. Then
\begin{multline*}
|F(z)| = \left | \sum _\alpha c_\alpha (f) \frac {z^\alpha }{(\alpha !)^{\frac 12}} \right |
\le
C \sum _\alpha \frac {(R_1r)^{|\alpha |}}{(\alpha !)^\rho}
\\[1ex]
\lesssim
\sup _{k\ge 0}\frac {(R_2r)^k}{(k!)^{\rho}}
\lesssim
\sup _{u\ge 1}\frac {(R_3r)^u}{u^{\rho u}},
\end{multline*}
for some $R_j$, $j\ge 1$.

\par

We need to maximize
$$
g(u) = \ln \left ( \frac {r_1^u}{u^{\rho u}}  \right ) = u\ln r_1 -\rho u \ln u,
\quad u\ge 1,
$$
where $r_1=R_3r$. By differentiation it follows that the maximum is attained for
$$
u=u_0=\frac {r_1^{\frac 1\rho}}e,
$$
giving that
\begin{equation*}
|F(z)|
\lesssim
e^{g(u_0)} = e^{\rho u_0} = e^{c|z|^{\kappa _1(\sigma )}},
\end{equation*}
for some constant $c>0$. This gives (2) in Proposition \ref{calJChar}.

\par

From these arguments it also follows that that there is a constant $C>0$
such that $c$ can be chosen in such way that
$$
C^{-1}\le Rc\le C,
$$
which shows that Proposition \ref{calJ0Char} (1) implies
Proposition \ref{calJ0Char} (2), also in the case $\sigma =\infty$.
%By letting
%$\rho =\frac 12$ when $\sigma =\infty$, the previous arguments show
%that Proposition \ref{calJ0Char} (1) implies
%Proposition \ref{calJ0Char} (2) also in this case.

\par

Assume instead that (2) in Proposition \ref{calJChar} holds, let $c_\alpha$
be chosen such that \eqref{Powerseries} holds, $k=|\alpha |+d-1$ and let
$\vartheta (\alpha )$ be given by \eqref{omegavarthetaRel}$'$. Then
$$
\vartheta (\alpha )^2 =\frac 1{k!}\int _0^\infty e^{-2cr^{\frac 1{2\rho}}}r^k\, dr.
$$
By taking $u=2cr^{\frac 1{2\rho}}$ as new variable of integration we obtain
\begin{equation}\label{Eq:varthetaGammaEst}
\vartheta (\alpha )^2
=
\frac {2\rho} {(2c)^{2\rho (k+1}) k!}\int _0^\infty e^{-u}u^{2\rho (k+1)-1}\, du
\asymp \frac {\Gamma (2\rho (k+1))}{(2c)^{2\rho k}k!}.
\end{equation}

\par

By Stirling's formula we get
\begin{multline*}
\vartheta (\alpha )^2
\asymp
\theta _1^k \frac {\displaystyle{\left ( \frac {2\rho (k+1)-1}e \right ) ^{2\rho (k+1)-1}}}
{\displaystyle{ \left (  \frac ke  \right )^k}}
\gtrsim \theta _2^k k^{(2\rho -1)k}
\\[1ex]
= \theta _2^k k^{\frac k\sigma}\gtrsim \theta _3^k(k!)^{\frac 1\sigma}
\gtrsim \theta _4^{|\alpha |}(\alpha !)^{\frac 1\sigma},
\end{multline*}
for some constants $\theta _j$, $j=1,\dots ,4$.

\par

If $\theta =\theta _4^{\frac 12}$, then Theorem \ref{cA2sCharExt} now
gives
\begin{multline}\label{FJAnalComp}
\sum _\alpha |c_\alpha  \theta ^{|\alpha |}(\alpha !)^{\frac 1{2\sigma}}|^2
\lesssim
\sum _\alpha |c_\alpha  \vartheta (\alpha )|^2 \asymp
\nm {F\cdot e^{-c|\cdo |^{\frac 1\rho}}}{L^2} <\infty ,
\end{multline}
provided $c>$ is chosen large enough. This gives (1) in Proposition
\ref{calJChar}.

\par

Evidently, the constants $\theta _j$ here above can be chosen in
such way that
$$
C^{-1}\le \theta _j c\le C,
$$
for some constant $C>0$ which is independent of $c$, which shows that
(2) in Proposition \ref{calJ0Char} implies (1) in Proposition \ref{calJ0Char}
when $\sigma <\infty$. By similar arguments, or using Corollary
\ref{CorvarthetaCondRadSym}, it follows that the latter implication
also holds in the case $\sigma =\infty$.
\end{proof}

\par

For the duals $\maclH _{\flat _\sigma} '(\rr d)$ and $\maclH _{0,\flat _\sigma} '(\rr d)$
we have now the following.

\par

\begin{prop}\label{calJCharDual}
Let $p\in [1,\infty ]$, $\sigma \in (1,\infty )$, $\kappa _2$ be as in
\eqref{Eq:kappaDef}, $\phi$
be as in \eqref{phidef}, $f\in \maclH _0'(\rr d)$, and let $F=\mathfrak V_df$.
Then the following conditions are equivalent:
\begin{enumerate}
\item $f\in \maclH _{\flat _\sigma} '(\rr d)$;

\vrum

\item $\nm {F\cdot e^{-r|\cdo |^{\kappa _2(\sigma )}}}{L^p}<\infty$ for every $r>0$.
\end{enumerate}
\end{prop}

\par

\begin{prop}\label{calJ0CharDual}
Let $p\in [1,\infty ]$, $\sigma \in (1,\infty ]$, $\kappa _2$ be as in
\eqref{Eq:kappaDef}, $\phi$
be as in \eqref{phidef}, $f\in \maclH _0'(\rr d)$, and let $F=\mathfrak V_df$.
Then the following conditions are equivalent:
\begin{enumerate}
\item $f\in \maclH _{0,\flat _\sigma} '(\rr d)$;

\vrum

\item $\nm {F\cdot e^{-r|\cdo |^{\kappa _2(\sigma )}}}{L^p}<\infty$ for some $r>0$.
\end{enumerate}
\end{prop}

\par

For the proofs we need the following lemma.

\par

\begin{lemma}\label{GenExpEst}
Let $R>0$, $0<\rho \le \frac 12$, $F\in A(\cc d)$ and let $p\in [1,\infty ]$. Then
there are positive constants $c_1$, $c_2$ and $C$ which only depends on
$\rho$ and $p$ such that
\begin{equation}\label{AnalLpNormEsts}
C^{-1}\nm {Fe^{-c_1R|\cdo |^{\frac 1\rho}}}{L^p}
\le \nm {Fe^{-R|\cdo |^{\frac 1\rho}}}{L^\infty}
\le
C\nm {Fe^{-c_2R|\cdo |^{\frac 1\rho}}}{L^p}.
\end{equation}
\end{lemma}

\par

\begin{proof}
The first inequality in \eqref{AnalLpNormEsts} is a straight-forward consequence
of H{\"o}lder's inequality.

\par

By H{\"o}lder's inequality again, it suffices to show that the second inequality
holds for $p=1$. Let $z\in \cc d$ be fixed and let
$$
\Omega _z = \sets {w\in \cc d}{2^{-1}\le |w_j-z_j|\le 1,\ j=1,\dots ,d}.
$$
By the mean-value property of analytic functions we obtain
\begin{multline*}
|F(z)e^{-R|z|^{\frac 1\rho}}|
=
\left ( \frac {3\pi}4 \right )^d \left | \int _{\Omega _z}F(w)
e^{-R|z|^{\frac 1\rho}}\, d\lambda (w)  \right |
\\[1ex]
\lesssim
\int _{\Omega _z}|F(w)
e^{-Rc|w|^{\frac 1\rho}}|\, d\lambda (w)
\le
\nm {F\cdot 
e^{-Rc|\cdo |^{\frac 1\rho}}}{L^1(\cc d)}.
\end{multline*}
By taking supremum over all $z\in \cc d$, the second
inequality in \eqref{AnalLpNormEsts} follows.
\end{proof}

\par

Propositions \ref{calJCharDual} and \ref{calJ0CharDual} follow
by letting $\rho =\frac 1{\kappa _2(\sigma )}$ instead of
$\rho =\frac 1{\kappa _1(\sigma )}$ in the proof of Propositions
\ref{calJChar} and \ref{calJ0Char}, and using
Lemma \ref{GenExpEst}. The details are left for the reader.

\par

\begin{proof}[Proof of Theorems \ref{flatSpacesChar} and \ref{flatSpacesCharDual}]
First we prove the first part of Theorem \ref{flatSpacesCharDual} (1), which follows
if we prove that $\mathfrak V_d$ is homeomorphic from $\maclH _{\flat _1}'(\rr d)$
to $A(\cc d)$. We have that
\eqref{Aseminorm} is true for every $r>0$, if and only if
\begin{equation}\label{AlternativeTopSum}
\sum _\alpha  \frac {|c_\alpha |^2R^{|\alpha |}}{\alpha !} <\infty ,
\end{equation}
for every $R>0$, and that the topology on $A(\cc d)$, induced by the
latter sum, remains the same. Hence, if we let $\mathfrak V_df$ be
the power series expansion in \eqref{Powerseries} when $f\in
\maclH _{\flat _1} '(\rr d)$ is given by \eqref{fHermite}, it follows that
$\mathfrak V_d$ is continuous and linear from
$\maclH _{\flat _1} '(\rr d)$ to $A(\cc d)$.

\par

Furthermore, since the Taylor series for an analytic
function is unique, it follows that $\mathfrak V_d$ is injective,
and since every entire function is equal to its Taylor series, it 
also follows that $\mathfrak V_d$ is surjective. This gives the first part
of Theorem \ref{flatSpacesCharDual} (1), and the second part follows
by similar arguments and is left for the reader.

\par

By (1)--(3) in Propositions \ref{calJChar} it follows that
$\bsycalA _{\flat _\sigma} (\cc d)$ is contained in
\eqref{AflatIdentities2}. On the other hand, if $F$ belongs to the
set in \eqref{AflatIdentities2}, which is
a subset of $A(\cc d)$, then $F=\mathfrak V_df$ for some
$f\in \maclH _{\flat _1}'(\rr d)$. By \eqref{FJAnalComp} it now
follows that $f\in \maclH _{\flat _\sigma} (\rr d)$, and hence
$F= \mathfrak V_df\in \bsycalA _{\flat _\sigma} (\cc d)$.
This gives Theorem \ref{flatSpacesChar} (2).

\par

By using Propositions \ref{calJ0Char}, \ref{calJCharDual} and
\ref{calJ0CharDual} instead of Proposition \ref{calJChar}, the
remaining parts follow in similar ways. The details are left for the reader.
\end{proof}

\par

%%%%
\subsection{A broad family of modulation spaces}
%%%%

\par

Next we show that certain parts of the modulation space theory in
\cite{F1,FGW,Gc2,To11} can be extended, by letting the modulation
spaces be defined in the framework of $\maclH _{\flat _1}'(\rr d)$. More
precisely, for $\phi$ as in Definition \ref{bfspaces2}, and $\omega$ and
$\mascB$ as in Definition \ref{thespaces}, we let the
modulation space $M(\omega ,\mascB )$ be the set of $f\in \maclH _{\flat _1} '(\rr d)$
such that \eqref{modnorm2} holds. This extends the notion of modulation spaces
in Definition \ref{bfspaces2}. The following result shows that \cite[Theorem 3.4]{To11} holds
in a broader context.

\par

\begin{thm}\label{ABComplete}
Let $\omega$ and $\mascB$ be as in Definition \ref{thespaces}. Then
the following is true:
\begin{enumerate}
\item $A(\omega ,\mascB )$, $B(\omega ,\mascB )$ and $M(\omega ,\mascB )$
are quasi-Banach spaces. If in addition $\nu _1(\mascB )\ge 1$, then
$A(\omega ,\mascB )$, $B(\omega ,\mascB )$ and $M(\omega ,\mascB )$
are Banach spaces;

\vrum

\item the Bargmann transform is isometric and bijective
from $M(\omega ,\mascB)$ to $A(\omega ,\mascB)$.
\end{enumerate}
\end{thm}

\par

\begin{proof}
The statements in (1) concerning $B(\omega ,\mascB )$ are obvious. For the other
spaces, we first
prove that $A(\omega ,\mascB )$ is complete and thereafter show that (2) holds.
The completeness of $A(\omega ,\mascB )$ is then carried over to
$M(\omega ,\mascB )$.

\par

Let $F_j$, $j=1,2,\dots$ be a Cauchy sequence in $A(\omega ,\mascB )$ and set
$G_{j,k}=F_j-F_k$. Then $\nm {F_j-F}{B(\omega ,\mascB )}\to 0$ for some
$F\in B(\omega ,\mascB )$, as $j$ tends to $\infty$. We have to prove that
$F\in A(\cc d)$, and it suffices to prove that $F_j\to F$ locally uniformly.

\par

We have that $\log |G_{j,k}(z)|$ is pluri-subharmonic for every $j.k\ge 1$. By
using this fact and Jensen's inequality we get
\begin{multline*}
\log |G_{j,k}(z)|^p \le (2\pi )^{-d}\int _{[0,2\pi )^d} \log |G_{j,k}(z_1+r_1e^{i\fy _1} ,
\dots ,z_d+r_de^{i\fy _d})|^p \, d\fy
\\[1ex]
\le \log \left ((2\pi )^{-d}\int _{[0,2\pi )^d} |G_{j,k}(z_1+r_1e^{i\fy _1} ,
\dots ,z_d+r_de^{i\fy _d})|^p \, d\fy \right ).
\end{multline*}
for every $r_1,\dots ,r_d\ge 0$, where $p=\nu _1(\mascB )$. Hence,
$$
|F_j(z)-F_k(z)|^p \le (2\pi )^{-d}\int _{[0,2\pi )^d} |G_{j,k}(z_1+r_1e^{i\fy _1} ,
\dots ,z_d+r_de^{i\fy _d})|^p \, d\fy .
$$
By multiplying the previous inequality with $r_1\cdots r_d$
and integrating over a small rectangle $[0,\ep ]^d$
with respect to $(r_1,\dots ,r_d)$, we get
$$
|F_j(z)-F_k(z)| \le C_\ep \nm {F_j-F_k}{L^p(\Omega _\ep)},
$$
where $C_\ep >0$ only depends on $\ep >0$. Here $\Omega _\ep$
is the polydisc 
$$
\sets {w\in \cc d}{|w_j-z_j|<\ep,\ j=1,\dots ,d}.
$$

\par

Hence, if $\ep _1<\ep _2$, %and $\chi$ is the characteristic function of
%$\Omega _{\ep _2}$,
then
$$
\nm {F_j-F_k}{L^\infty (\Omega _{\ep _1})} \lesssim
\nm {F_j-F_k}{L^p(\Omega _{\ep _2})} \lesssim
\nm {F_j-F_k}{B(\omega ,\mascB )},
$$
where the last inequality follows from H{\"o}lder's inequality. Consequently,
$F_j$, $j=1,2,\dots$, is a Cauchy sequence in $L^\infty (\Omega _{\ep _1})$,
and it follows that $F_j\to F_0$ in $L^\infty (\Omega _{\ep _1})$ for some
$F_0\in A(\Omega _{\ep _1})$.  Since
$F_k\to F$ in $B(\omega ,\mascB )$ as $k$ tends to $\infty$, it follows,
by passing to subsequences of $F_j$ if necessary, that $F=F_0$ a.{\,}e.,
and the analyticity of $F$ follows. This proves that $A(\omega ,\mascB )$
is complete.

\par

By the definitions it follows that the Bargmann transform from
$\maclH _{\flat _1} '(\rr d)$ to $A(\cc d)$ restricts to an isometric and injective
map from $M(\omega ,\mascB)$ to $A(\omega ,\mascB)$.
On the other hand, if $F\in A(\omega ,\mascB )$, then $F=\mathfrak V_df$
for some $f\in \maclH _{\flat _1} '(\rr d)$, in view of Theorem \ref{flatSpacesChar}.
By the definitions it follows that $f$ has the same (finite)
quasi-norm in $M(\omega ,\mascB )$ as $\nm F{A(\omega ,\mascB )}$.
Hence $f\in M(\omega ,\mascB )$, and (2) follows. This gives the result.
\end{proof}

\par

%%%%
\subsection{Comparisons between Pilipovi{\'c} spaces and
a test function space introduced by Gr{\"o}chenig}
%%%%

\par

We finish the section by performing comparisons between
the spaces $\maclH _{\flat _1} (\rr d)$, $\maclH _{0,{\flat _1}}(\rr d)$ and the space
$\maclS _C(\rr d)$, introduced by Gr{\"o}chenig in \cite{Gc2}.
We recall that $\maclS _C(\rr d)$ consists of all
$f\in L^2(\rr d)$ such that $f=V_\phi ^*F$, for some $F$ belonging
to $L^\infty (\rr d)\cap \mascE '(\rr d)$. %, the set of all elements in $L^\infty (\rr d)$
%with compact supports.
Here $\phi (x)=\pi ^{-\frac d4}e^{-\frac {|x|^2}2}$, as
usual. In particular, $f\in \maclS _C(\rr d)$, if and only if
$$
f(x) = (2\pi )^{-\frac d2}\iint _{\rr {2d}} F(y,\eta )e^{-\frac 12|x-y|^2}
e^{i\scal x\eta}\, dyd\eta ,
$$
for some $L^\infty (\rr d)\cap \mascE '(\rr d)$.

\par

\begin{lemma}\label{GrochSpaceBargm}
Let $F\in L^\infty (\cc d)$. Then the Bargmann transform of $f=V_\phi ^*F$
is given by $\Pi _AF_0$, where
$$
F_0(x,\xi ) = (2\pi ^3)^{\frac d4}F(\sqrt 2 x,-\sqrt 2 \xi )e^{\frac 12(|x|^2+|\xi |^2)}
e^{-i\scal x\xi}.
$$

\par

Moreover, the image of $\maclS _C(\rr d)$ under the Bargmann
transform is given by
$$
\sets {\Pi _AF}{F\in L^\infty (\cc d)\bigcap \mascE '(\cc d)}.
$$
\end{lemma}

\par

Here recall that the map $\Pi _A$ is defined
by \eqref{reproducing}.

\par

\begin{proof}
The result follows by straight-forward computation of the Bargmann
transform of elements given by $V_\phi ^*F$ when $F\in L^\infty (\rr d)$
or $F\in L^\infty (\rr d) \cap \mascE '(\rr d)$. The details are left
for the reader.
\end{proof}

\par

The next result links
\begin{align}
\Omega _{0,R} &= \sets {\Pi _AF}{F\in L^\infty (\cc d),\ \supp F\subseteq B_{R}(0)},
\label{Omega0Def}
\\[1ex]
\Omega _{1,R} &= \sets {\Pi _AF}{F\in L^\infty (\cc d),\ \supp F\subseteq D_{d,R}}
\label{Omega1Def}
\intertext{and}
\Omega _{2,R} &= \Sets {F =\sum _{\alpha \in \nn d}c_\alpha e_\alpha}
{\sup _{\alpha \in \nn d}\left (|c_\alpha |R^{-|\alpha |}\right )<\infty}
\label{Omega2Def}
\end{align}
to each others, where $D_{d,R}$ is the polydisc
$$
\sets{(z_1,\dots ,z_d)\in \cc d}{|z_j|\le R,\ j=1,\dots ,d} .
$$

\par

\begin{thm}\label{Thm:GrochFlatSpacIdent}
Let $\Omega _{j,R}$, $j=0,1,2$, be given by
\eqref{Omega0Def}--\eqref{Omega2Def} when $R>0$, and let
$R_0,R\in \mathbf R$ be such that $0<R_0<R$. Then
\begin{equation}\label{Omega12RRelations}
\Omega _{0,R}\subseteq \Omega _{2,R}
\quad \text{and}\quad
\Omega _{2,R_0}\subseteq \Omega _{1,R}.
\end{equation}
In particular,
$$
\maclS _C(\rr d)=\maclH _{\flat _1} (\rr d)
\quad \text{and}\quad
\maclH _{0,{\flat _1}}(\rr d)\subsetneq \maclS _C(\rr d).
$$
\end{thm}

\par

We need some preparations for the proof, and begin with the following lemma.

\par

\begin{lemma}\label{Lemma:HermiteGroch}
Let $\chi$ be the characteristic function of the polydisc
$$
\sets {z=(z_1,\dots ,z_d)\in \cc d}{|z_j|\le 1,\ j=1,\dots ,d},
$$
and let
$$
F_\alpha (z)= \sqrt {\alpha !}\left ( \prod _{j=1}^d (\alpha _j+1)\right )
z^\alpha e^{|z|^2}\chi (z),\quad z\in \cc d,\ \alpha \in \nn d .
$$
Then
\begin{equation}\label{Eq:PiAFalpha}
\Pi _AF_\alpha =e_\alpha
\end{equation}
\end{lemma}

\par

\begin{proof}
Since
$$
\Pi _A F_\alpha = (\Pi _AF_{\alpha _1})\otimes \cdots \otimes (\Pi _A F_{\alpha _d})
\quad \text{and}\quad
e_\alpha = e_{\alpha _1}\otimes \cdots \otimes e_{\alpha _d},
$$
when $\alpha =(\alpha _1,\dots ,\alpha _d)\in \nn d$, it suffices to prove the result for $d=1$.
Furthermore, since both sides in \eqref{Eq:PiAFalpha} are entire, it suffices to prove
\begin{equation}\tag*{(\ref{Eq:PiAFalpha})$'$}
\Pi _AF_\alpha (e^{i\fy })=e_\alpha (e^{it}) = \sqrt {e^{i\alpha \fy }}{\sqrt {\alpha !}},\quad
\fy \in \mathbf R,
\end{equation}
by analytic continuation.

\par

By taking polar coordinates when integrating we get
\begin{multline*}
\Pi _AF_\alpha (e^{i\fy }) = \pi ^{-1}\sqrt {\alpha !}(\alpha +1) \int _{|w|\le 1}
w^\alpha e^{e^{i\fy }\overline w}\, d\lambda (w)
\\[1ex]
%=\pi ^{-1}\sqrt {\alpha !}}(\alpha +1) \int _0^1 \int _0^{2\pi}r^{\alpha +1}e^{i\alpha \theta}
%e^{re^{i(\fy -\theta)}}\, d\theta dr
= \sqrt {\alpha !}(\alpha +1) \int _0^1 r^{\alpha +1}I_\alpha (r,\fy )\, dr,
\end{multline*}
where
\begin{multline*}
I_\alpha (r,\fy ) = \pi ^{-1}\int _0^{2\pi}e^{i\alpha \theta}e^{re^{\fy -\theta }}\, d\theta
\\[1ex]
=\pi ^{-1}e^{i\alpha \fy}\int _0^{2\pi}e^{-i\alpha \theta}e^{re^{i\theta}}\, d\theta
\\[1ex]
= 2e^{i\alpha \fy}\left ( \frac 1{2\pi i}\int _\gamma \frac {e^{rz}}{z^{\alpha +1}}\, dz \right ),
\end{multline*}
where $\gamma$ is the unit circle in counterclockwise direction. Hence,
$$
\frac 1{2\pi i}\int _\gamma \frac {e^{rz}}{z^{\alpha +1}}\, dz
= \frac 1{\alpha !} \frac {d^\alpha }{dz^{\alpha}}(e^{rz})\Big \vert _{z=0}= \frac {r^\alpha}{\alpha !},
$$
by Cauchy's integral formula.

\par

This gives
$$
I_\alpha (r,\fy ) = \frac {2e^{i\alpha \fy}r^{\alpha}}{\alpha !}
$$
and
$$
\Pi _AF_\alpha (e^{i\fy }) = \frac {2(\alpha +1)e^{i\alpha \fy}}{\sqrt {\alpha !}}
\int _0^1r^{2\alpha +1}\, dr = \frac {e^{\alpha \fy}}{\sqrt {\alpha !}} =e_\alpha (e^{i\fy}),
$$
and the result follows.
\end{proof}

\par

\begin{proof}[Proof of Theorem \ref{Thm:GrochFlatSpacIdent}]
By Proposition \ref{calJChar}, the result follows if we prove
\eqref{Omega12RRelations}. Assume that $F\in \Omega _{0,R}$. Then
$$
|F(w)e^{(z,w)}|\lesssim e^{R|z|}.
$$
This gives
$$
|\mathfrak V_df(z)| \lesssim e^{R|z|}\int _{\cc d}\, d\mu (w)\asymp e^{R|z|},
$$
which implies $\Omega _{0,R}\subseteq \Omega _{2,R}$.

\par

It remains to prove $\Omega _{2,R_0}\subseteq \Omega _{1,R}$. By
$$
(\Pi _AF)(Rz) =(\Pi _A F_R)(z),
\quad \text{when}\quad
F_R(z) = R^{-2d}e^{(1-\frac 1{R^2})|w|^2}F\Big ( \frac zR \Big ),
$$
we may reduce ourself to the case when $R=1$ and $R_0<1$.

\par

Let $F_\alpha$ be as in the proof of Lemma \ref{Lemma:HermiteGroch}.
By the latter lemma and Weierstrass theorem it suffices to prove that
$$
\sum _{\alpha \in \nn d} |c_\alpha | \nm {F_\alpha}{L^\infty}<\infty ,
$$
since $\supp (F_\alpha )\subseteq D_{d,1}$.

\par

We have
$$
\sum _{\alpha \in \nn d} |c_\alpha | \nm {F_\alpha}{L^\infty}
\lesssim 
\sum _{\alpha \in \nn d} \left (\prod _{j=1}^d(\alpha _j+1)\right ) R_1^{|\alpha |}<\infty ,
$$
and the result follows.
\end{proof}

\par

\begin{example}
Let $F(z) =e^{|z|^2}\chi (z)$, where $\chi$ is the
characteristic function of the set $[0.1]^{2d}$, and let $f$
be chosen such that $\mathfrak V_df=\Pi _AF$. Then $f\in
\maclS _C(\rr d)$ since $F\in L^\infty (\rr d)\cap \mascE '(\rr d)$.

\par

On the other hand, by
straight-forward computations one obtains
$$
\mathfrak V_df(z)=\Pi _AF(z)
=
\pi ^{-d}\left ( \prod _{j=1}^d \frac {e^{z_j}-1}{z_j} \right ) 
\left ( \prod _{j=1}^d \frac {1-e^{-iz_j}}{iz_j} \right ).
$$
This implies that
$$
z\mapsto |\mathfrak V_df(z)|e^{-\ep |z|}
$$
is unbounded when $\ep >0$ is small enough. Consequently,
$f\notin \maclH _{0,{\flat _1}}(\rr d)$.
\end{example}

\par

%%%%%%%%%%%%%%%%%%%%%%%%%%%%%%%%%%
\section{Characterizations of $\maclH _{s} (\rr d)$ and
$\maclH _{0,s}(\rr d)$, and their duals, for $s\in \mathbf R_+$}\label{sec4}
%%%%%%%%%%%%%%%%%%%%%%%%%%%%%%%%%%

\par

In this section we show that $\bsySig _s $, $\bsycalS _{\! s}$ and their
duals are equal to $\maclH _{0,s}$, $\maclH _s$ and their duals, respectively.
We also describe their images under the Bargmann transform as convenient
spaces of entire functions.

\par

First we need certain invariance properties concerning the norm
condition \eqref{GFHarmCond}. More precisely,
the following result links the conditions
\begin{align}
\label{GFHarmCond1}
\sup _{N\ge 0} \frac {\nm{H^Nf}{L^{p_0}}}{h^N(N!)^{2s}} &<\infty ,
%\\[1ex]
\intertext{and}
\label{GFHarmCond2}
\sup _{N\ge N_0} \frac {\nm {H^Nf}{M^{p,q}_{(\omega )}}}{h^N(N!)^{2s}} &<\infty .
\end{align}
to each others and shows in particular that the $L^\infty$ norm in
\eqref{GFHarmCond} can be replaced by other types of Lebesgue or modulation
space  quasi-norms.

\par

\begin{prop}\label{NormEquiv}
Let $p_0\in [1,\infty ]$, $p,q\in (0,\infty ]$, $N_0\ge 0$ be an integer, $s\ge 0$
and let $\omega \in \mascP (\rr {2d})$. Then the following conditions
are equivalent:
\begin{enumerate}
\item \eqref{GFHarmCond1} holds for some $h>0$ (for every $h>0$);

\vrum

\item  \eqref{GFHarmCond2} holds for some $h>0$ (for every $h>0$).
\end{enumerate}
\end{prop}

\par

\begin{proof}
First we prove that \eqref{GFHarmCond2} is independent of $N_0\ge 0$
when $p,q\ge 1$. Evidently, if \eqref{GFHarmCond2} is true for $N_0=0$,
then it is true also for $N_0>0$. On the other hand, the map
\begin{equation}\label{HarmonicOscModMap}
H^N \, :\, M^{p,q}_{(v_N\omega )}(\rr d)\to M^{p,q}_{(\omega )}(\rr d),
\qquad v_N(x,\xi )=(1+|x|^2+|\xi |^2)^N,
\end{equation}
and its inverse are continuous and bijective (cf. e.{\,}g.
\cite[Theorem 3.10]{SiTo}). Hence, if $0\le N\le N_0$,
$N_1=N_0-N\ge 0$ and \eqref{GFHarmCond2} holds for some
$N_0\ge 0$, then
$$
\nm {H^Nf}{M^{p,q}_{(\omega )}} \lesssim \nm {H^{N_0}f}{M^{p,q}_{(\omega /v_{N_1})}}
\lesssim \nm {H^{N_0}f}{M^{p,q}_{(\omega )}}<\infty ,
$$
and \eqref{GFHarmCond2} holds for $N_0=0$. This implies that
\eqref{GFHarmCond2} is independent of $N_0\ge 0$ when $p,q\ge 1$.

\par

Next we prove that (2) is independent of the
choice of $\omega \in \mascP (\rr {2d})$.
For every $\omega _1,\omega _2\in \mascP (\rr {2d})$, we may find an integer
$N_0\ge 0$ such that
$$
\frac 1{v_{N_0}}\lesssim \omega _1, \omega _2\lesssim v_{N_0},
$$
and then
\begin{equation}\label{NormArrayCond}
\nm f{M^{p,q}_{(1/v_{N_0})}}\lesssim \nm f{M^{p,q}_{(\omega _1)}},
\nm f{M^{p,q}_{(\omega _2)}}
\lesssim \nm f{M^{p,q}_{(v_{N_0})}}.
\end{equation}
Hence the stated invariance follows if we prove that
\eqref{GFHarmCond2} holds for $\omega =v_{N_0}$, if it is true for
$\omega =1/v_{N_0}$.

\par

Therefore, assume that \eqref{GFHarmCond2} holds for $\omega =1/v_{N_0}$.
If $N\ge 2N_0$, then the bijectivity of \eqref{HarmonicOscModMap} gives
\begin{multline*}
\frac {\nm {H^Nf}{M^{p,q}_{(v_{N_0})}}}{h^N(N!)^{2s}}
\lesssim
\frac {\nm {H^{N+2N_0}f}{M^{p,q}_{(1/v_{N_0})}}}{h^N(N!)^{2s}}
\\[1ex]
=
h^{2N_0}{{N+2N_0}\choose {2N_0}}^{2s}((2N_0)!)^{2s}
\frac {\nm {H^{N+2N_0}f}{M^{p,q}_{(1/v_{N_0})}}}{h^{N+2N_0}((N+2N_0)!)^{2s}}
\\[1ex]
\asymp
{{N+2N_0}\choose {2N_0}}^{2s}
\frac {\nm {H^{N+2N_0}f}{M^{p,q}_{(1/v_{N_0})}}}{h^{N+2N_0}((N+2N_0)!)^{2s}}
\lesssim
\frac {\nm {H^{N+2N_0}f}{M^{p,q}_{(1/v_{N_0})}}}{h_1^{N+2N_0}((N+2N_0)!)^{2s}},
\end{multline*}
where $h_1=\frac h{4^s}$. This gives the invariance of
(2) with respect to $\omega$ in the case $p,q\ge 1$. For general
$p,q>0$, the invariance of \eqref{GFHarmCond2} with respect to $\omega$,
$p$ and $q$ is now a consequence of the embeddings
$$
M^\infty _{(v_N\omega )}(\rr d)\subseteq M^{p,q} _{(\omega )}(\rr d)
\subseteq M^\infty _{(\omega )}(\rr d),\qquad N>\frac dp.
$$

\par

The equivalence between (1) and (2) follows from
these invariance properties and the continuous
embeddings
$$
M^{p_0,q_1}\subseteq L^{p_0}\subseteq M^{p_0,q_2},
\qquad
q_1=\min (p_0,p_0'),\quad q_2=\max (p_0,p_0'),
$$
which can be found in e.{\,}g. \cite{To8}.
\end{proof}

\par

%Next we deduce that if $s=\frac 12$, then the images of
%$\maclH _{0,s} (\rr d)$, $\maclH _s (\rr d)$ and their duals
%under the Bargmann transform can be characterized
%by suitable unions and intersections of spaces of the forms
%$A^2_{(\omega _h)}(\cc d)$, when
%
%For such weights, Theorem \ref{cA2sCharExt} and
%Proposition \ref{varthetaCondRadSym} take the following form.

\par

The next three theorems characterize spaces in \eqref{clHSpaces}
in terms of Pilipovi{\'c} spaces and spaces of entire functions. Here we let
\begin{equation}\label{Eq:omegajrsDef}
\omega _{1,r,s}(z)
=
\begin{cases}
e^{r(\log \eabs z)^{\frac 1{1-2s}}}, & s<\frac 12
\\[1ex]
e^{\frac {|z|^2}2-r|z|^{\frac 1s}}, & s\ge \frac 12,
\end{cases}
\quad \text{and}\quad
\omega _{2,r,s}(z) = e^{\frac {|z|^2}2+r|z|^{\frac 1s}},
%=
%\begin{cases}
%e^{r(\log \eabs z)^{\frac 1{1-2s}}}, & s<\frac 12
%\\[1ex]
%e^{\frac {|z|^2}2-r|z|^{\frac 1s}}, & s\ge \frac 12,
%\end{cases}
\end{equation}

\par

\begin{thm}\label{PilHsIdentities}
Let $s\ge 0$. Then
\begin{equation}\label{Eq:PilHsIdentities}
\begin{alignedat}{2}
\bsySig _s(\rr d) &= \maclH _{0,s}(\rr d),&\qquad
\bsycalS _{\! s}(\rr d) &= \maclH _{s}(\rr d),
\\[1ex]
\bsycalS _{\! s}'(\rr d) &= \maclH _{s}'(\rr d),&\qquad
\bsySig _s'(\rr d) &= \maclH _{0,s}'(\rr d).
\end{alignedat}
\end{equation}
\end{thm}

\par

\begin{thm}\label{PilSpacesChar}
Let $s\ge 0$ and let $\omega _{1,r,s}$ be as in \eqref{Eq:omegajrsDef}.
Then the following is true:
\begin{enumerate}
\item the map \eqref{Eq:FormalPowerToAnalFunc} extends uniquely
to a homeo\-morphism from $\bsycalA _{s} (\cc d)$ to
\begin{equation}\label{AsIdentities2}
\sets {F\in A(\cc d)}{|F(z)|\lesssim
\omega _{1,r,s} \ \text{for some}\ r >0}\text ;
\end{equation}

\vrum

\item if in addition $s\notin \{0,\frac 12\}$, then the map
\eqref{Eq:FormalPowerToAnalFunc} extends uniquely
to a homeo\-morphism from $\bsycalA _{0,s} (\cc d)$ to
\begin{equation}\label{AsIdentities1}
\sets {F\in A(\cc d)}{|F(z)|\lesssim
\omega _{1,r,s} \ \text{for every}\ r >0}.
\end{equation}
\end{enumerate}
\end{thm}

\par

\begin{thm}\label{PilSpacesCharDual}
Let $s\ge \frac 12$ and let $\omega _{1,r,s}$ be as in \eqref{Eq:omegajrsDef}.
Then the following is true:
\begin{enumerate}
\item the map \eqref{Eq:FormalPowerToAnalFunc} extends uniquely
to a homeo\-morphism from $\bsycalA _{s}' (\cc d)$ to
\begin{equation}\label{AsIdentities7}
\sets {F\in A(\cc d)}{|F(z)|\lesssim
\omega _{2,r,s} \ \text{for every}\ r >0}\text ;
\end{equation}

\vrum

\item if in addition $s\neq \frac 12$, then the map
\eqref{Eq:FormalPowerToAnalFunc} extends uniquely
to a homeo\-morphism from $\bsycalA _{0,s}' (\cc d)$ to
\begin{equation}\label{AsIdentities8}
\sets {F\in A(\cc d)}{|F(z)|\lesssim
\omega _{2,r,s} \ \text{for some}\ r >0}\text ;
\end{equation}
\end{enumerate}
\end{thm}

\par

Again we recall that Remark \ref{RemAnalSpaceTop} explains
the topologies of the spaces in \eqref{AsIdentities2}--\eqref{AsIdentities8}.

\par

\begin{rem}
Evidently, since the Bargmann transform is bijective between the spaces
in \eqref{clHSpaces} and \eqref{clASpaces}, Theorems \ref{flatSpacesChar},
\ref{flatSpacesCharDual}, \ref{PilSpacesChar},
and \ref{PilSpacesCharDual} remains true if the spaces
in \eqref{clASpaces} and the map \eqref{Eq:FormalPowerToAnalFunc}
are replaced by corresponding spaces in \eqref{clHSpaces} and the
Bargmann transform, respectively.
\end{rem}

\par

\par

\begin{rem}\label{Rem:LinkSTFT}
Because of the strong links between the Bargmann transform
and the short-time Fourier transform when the window
function is given by $\phi$ in \eqref{phidef}, Theorem
\ref{PilSpacesChar} implies that the conditions
on the right-hand sides of \eqref{AflatIdentities1}--\eqref{AflatIdentities5}
and \eqref{AsIdentities2}--\eqref{AsIdentities7}
carry over to analogous conditions on the short-time Fourier transforms.
For example, (1) in Theorem \ref{flatSpacesChar}
shows that $f\in \maclH _{0,\frac 12}(\rr d)$
is equivalent to
$$
|(V_\phi f)(x,\xi )|\lesssim e^{-\frac 14(1-\ep )(|x|^2+|\xi |^2)}\quad
\text{for every $\ep >0$.}
$$
In Section \ref{sec6} we present a more comprehensive list of relations between
elements in function spaces and estimates of their short-time Fourier transform
(see Proposition \ref{STFTPilSpChar}).
\end{rem}

\par

Theorem \ref{PilSpacesChar} in the case $0<s<\frac 12$ follows by suitable
applications of Theorem \ref{cA2sCharExt}, and is presented in \cite{FeGaTo2}.
Moreover, proofs of extended versions of Theorems \ref{PilSpacesChar}
and \ref{PilSpacesCharDual} in the case $s\ge \frac 12$ are given in
Section \ref{sec5}, using the facts that $\bsySig _s(\rr d) =\Sigma _s(\rr d)$
when $s>\frac 12$ and $\bsycalS _{\! s}(\rr d) = \maclS _s(\rr d)$ when
$s\ge \frac 12$.  (See Theorem \ref{BargGSMapProp} and its proof.)

\par

At this moment it therefore remains to prove Theorem \ref{PilHsIdentities}, and then
the following lemma takes care of the case $s>0$.

\par

\begin{lemma}\label{Equiv1And3}
Let $C,s,h>0$, $f\in \maclH _0'(\rr d)$, and let
$c_\alpha (f)$ be the same as in \eqref{fHermite}. Then the
following is true:
\begin{enumerate}
\item if
\begin{equation}\label{HNCond}
\nm {H^Nf}{L^2}\le Ch^N(N!)^{2s}
\end{equation}
for every integer
$N\ge 0$, then
\begin{equation}\label{calphaExpCond}
|c_\alpha (f) | \le 4^sCe^{-s|\alpha |^{\frac 1{2s}}/h^{\frac 1{2s}}}
%\qquad e^{-\frac s{h^{\frac 1{2s}}}|\alpha |^{\frac 1{2s}}}
\text ;
\end{equation}

\vrum

\item if % in addition $h\le R$ for some $R>0$ and
\begin{equation}\label{calphaExpCond2}
|c_\alpha (f) | \le Ce^{-\frac 1h|\alpha |^{\frac 1{2s}}},
\end{equation}
then
\begin{equation}\label{HNCond2}
\nm {H^Nf}{L^2}\le C_1 (1+2sh^{2sd}\Gamma (2sd)) (3(4sh)^{2s})^N(N!)^{2s},
\end{equation}
for some constant $C_1$ which only depends on $C$ and $d$.
\end{enumerate}
\end{lemma}

\par

\begin{proof}
Assume that \eqref{HNCond} holds for every $N\ge 0$.
Since $H^Nh_\alpha =(2|\alpha |+d)^Nh_\alpha$ and
$\{ h_\alpha \} _{\alpha \in \nn d}$ is an
orthonormal basis for $L^2$, we get
\begin{equation*}
(2|\alpha|+d)^N | c_\alpha (f) | \le \nm {H^Nf}{L^2}\le Ch^N(N!)^{2s}.
\end{equation*}
Hence,
$$
|c_\alpha (f) |\le \frac {Ch^N(N!)^{2s}}{(2|\alpha|+d)^N},
$$
which implies
\begin{equation}\label{EqEstimate1}
%\begin{cases}
%|c_0(f)| \le Ch^N(N!)^{2s}
%\\[1ex]
%|c_\alpha (f) | \le \displaystyle{\frac {Ch^N(N!)^{2s}}{|\alpha | ^N}},
%\quad |\alpha |\ge 1.
%\end{cases}
|c_0(f)| \le Ch^N(N!)^{2s}
\quad \text{and}\quad
|c_\alpha (f) | \le \displaystyle{\frac {Ch^N(N!)^{2s}}{|\alpha | ^N}},
\quad |\alpha |\ge 1.
\end{equation}
In particular, \eqref{calphaExpCond} holds when $\alpha =0$.

\par

We may therefore assume that $|\alpha |\ge 1$. By \eqref{EqEstimate1}
we get
$$
|c_\alpha (f) |^{\frac 1{2s}}\cdot \frac {\displaystyle{ 
\left ( \frac {|\alpha |^{\frac 1{2s}}}{2h^{\frac 1{2s}}} \right )^N }} 
{N!} \le C^{\frac 1{2s}}\frac 1{2^N},
$$
and by summing over all $N\ge 0$ it follows that
$$
|c_\alpha (f) |^{\frac 1{2s}}e^{|\alpha |^{\frac 1{2s}} /(2h^{\frac 1{2s}})}\le 2C^{\frac 1{2s}},
$$
which is the same as \eqref{calphaExpCond}. This gives (1).

\par

Next assume that \eqref{calphaExpCond2} holds. Then
Parseval's formula gives %and the fact that $h\le R$ give
\begin{multline*}
\nm {H^Nf}{L^2}\lesssim \left ( \sum _\alpha (2|\alpha |+d)^{2N}
e^{-\frac 2h|\alpha |^{\frac 1{2s}} }    \right )^{\frac 12}
\\[1ex]
\lesssim 1+\left ( \sum _{|\alpha |\ge 1} 9^N|\alpha |^{2N}
e^{-\frac 2h |\alpha |^{\frac 1{2s}}}    \right )^{\frac 12}
\lesssim
1+ 3^NR\sup _{t\ge 1}g(t),
\end{multline*}
where
$$
R= \sum _{k=1}^\infty k^{d-1}e^{-k^{\frac 1{2s}}/h}
\quad \text{and}\quad
g(t) = t^Ne^{-t^{\frac 1{2s}}/(2h) }.
$$

\par

We need to estimate $R$ and the last supremum. We have
\begin{multline*}
R^2 \lesssim 1+\int _1^\infty r^{d-1}e^{-r^{\frac 1{2s}}/h}\, dr = 1+2sh^{2sd}\int _{1/(2h)}^\infty
e^{-r}r^{2sd-1}\, dr
\\[1ex]
\le 1+2sh^{2sd}\Gamma (2sd).
\end{multline*}

\par

Furthermore, by
straight-forward computations, it follows that
$g(t)$ for $t\ge 1$ attains its maximum when $t=(4shN)^{2s}$.
Hence, by Stirling's formula we get
\begin{multline*}
\sup _{t\ge 1}g(t) = g((4shN)^{2s}) = ((4sh)^{2s})^N\left ( \frac {N^N}{e^N} \right )^{2s}
\\[1ex]
\lesssim
((4sh)^{2s})^N(2\pi N)^{-s}(N!)^{2s}
\lesssim
%((4sh)^{2s})^N(2^{2s})^N(N!)^{2s}
%= 
((4sh)^{2s})^N(N!)^{2s}.
\end{multline*}
The result now follows by combining these estimates.
\end{proof}

\par

\begin{proof}[Proof of Theorem \ref{PilHsIdentities}]
It remains to prove \eqref{Eq:PilHsIdentities} in the case $s=0$.

\par

Assume that $f\in \bsycalS _{\! 0}(\rr d)$. By Proposition \ref{NormEquiv} it follows that
$\nm {H^Nf}{L^2}\le Ch^N$ for some $h>0$ and $C>0$ which are independent
of $N\ge 0$. If $c_\alpha (f)$ is the same as in \eqref{fHermite}, then 
$$
\sum _{\alpha \in \nn d} |(2|\alpha |+d)^Nc_\alpha (f)|^2 = \nm {H^Nf}{L^2}^2\le C^2h^{2N},
$$
giving that
$$
|c_\alpha (f)| \le Ch^N|(2|\alpha |+d)^{-N}.
$$

\par

If $\alpha$ is chosen such that $h<2|\alpha |+d$, then it follows by
letting $N$ tends to infinity at the last estimate that $c_\alpha (f)=0$.
This shows that $f\in \maclH _0(\rr d)$ and hence $\bsycalS _{\! 0}(\rr d)
\subseteq \maclH _0(\rr d)$.

\par

On the other hand, for any Hermite function $h_\alpha$ we have
$$
\nm {H^Nh_\alpha }{L^2}= (2|\alpha |+d)^N\nm {h_\alpha}{L^2}\le h^N
$$
when $2|\alpha |+d\le h$. Hence any Hermite function belongs to
$\bsycalS _{\! 0}(\rr d)$, and it follows that $\maclH _0(\rr d)
\subseteq \bsycalS _{\! 0}(\rr d)$. This gives \eqref{Eq:PilHsIdentities}.
%
%\par
%
%In order to prove (2), let $r\in \mathbf R$, $h>0$, and let $\omega _h$
%be the same as in the proof of Proposition \ref{cA2sChar}. Then
%$$
%\bsySig (\rr d) = \bigcap _{r>0}\maclH _{[\vartheta _r]}^2(\rr d)
%\quad \text{and}\quad
%\bsySig (\rr d) = \bigcup _{r>0}\maclH _{[1/\vartheta _r]}^2(\rr d).
%$$
%By Proposition \ref{cA2sChar}, it follows that the images of $\bsySig (\rr d)$ and 
%$\bsySig '(\rr d)$ under the Bargmann transform are given by
%$$
%\bigcap _{0<2h<1}A _{(\omega _h)}^2(\cc d)
%\quad \text{and}\quad
%\bigcup _{2h>1}A _{(\omega _h)}^2(\cc d),
%$$
%respectively, which is the same as \eqref{AsIdentities3} and \eqref{AsIdentities8}.
%This gives the result.
\end{proof}

\par

By a straight-forward combination of the \eqref{PilHsIdentities},
Theorems \ref{calJChar}, \ref{calJ0Char}, \ref{PilSpacesChar} and
\ref{PilSpacesCharDual}, and general embedding relations for
Gelfand-Shilov spaces, we get the following result.

\par

\begin{thm}
Let $s_1,s_2\in \mathbf R_\flat $ be such that $s_1<s_2$. Then the embeddings
in \eqref{InclSpaces} and \eqref{InclSpaces2}
hold true, and are continuous and dense. Furthermore, the embeddings
$$
\bsySig _{s_2}'(\rr d) \subseteq \bsycalS _{\! s_1}'(\rr d)\subseteq 
\bsySig _{s_1}'(\rr d) \subseteq \bsycalS _{\! 0}'(\rr d)
$$
hold true and are continuous.
\end{thm}

\par

\subsection{Some consequences}\label{subsec4.1}

\par

Next we show some consequences of the previous results.
As an immediate consequence of Theorem \ref{PilSpacesChar} we have
\begin{equation}\label{Eq:RelPilSpH2theta}
\begin{gathered}
\bsySig _s(\rr d) = \bigcap _{r>0}\maclH _{[\vartheta _r]}^2(\rr d)
\quad \text{and}\quad
\bsycalS _{\! s}(\rr d) = \bigcup _{r>0}\maclH _{[\vartheta _r]}^2(\rr d),
\\[1ex]
\text{when}
\quad
\vartheta _r(\alpha ) = e^{r|\alpha |^{\frac 1{2s}}},
\quad s>0
\quad \text{and}\quad
p\in (0,\infty ]
\end{gathered}
\end{equation}
(see also \cite{Pil1,Pil2}).

\par

We may also apply Theorem \ref{PilSpacesChar} and its proof to characterize
$\bsySig (\rr d)$ and $\bsySig '(\rr d)$ in terms of modulation spaces. In fact, by
Proposition \ref{cA2sChar} and Theorem \ref{PilSpacesChar} it follows that
$\bsySig (\rr d)$ and $\bsySig '(\rr d)$ are the projective and inductive limit,
respectively, of
$$
M^2_{(\omega _h )}(\rr d),\quad \omega _h (x,\xi )
= e^{\frac 14(1-2h )(|x|^2 + |\xi |^2)}
$$
with respect to $h >0$. In particular,
\begin{equation}\label{PilModEmb}
\begin{gathered}
\bsySig (\rr d) = \bigcap _{h >0}M^2_{(\omega _h )}(\rr d)
\quad \text{and}\quad
\bsySig '(\rr d) = \bigcup _{h >0}M^2_{(\omega _h )}(\rr d),
\\[1ex]
\omega _h (x,\xi ) = e^{\frac 14(1-2h)(|x|^2 + |\xi |^2)}.
\end{gathered}
\end{equation}

\par

Let $\mascB$ be an invariant quasi-norm space and let
$\sigma _r(x,\xi )=\eabs x^r\eabs \xi ^r$. Since
$$
\Omega =\sets {\omega _\ep \cdot \sigma _r}{r\in \mathbf R}
$$
is an admissible family of weights,
a combination of \cite[Theorem 3.2]{To11}, Theorem \ref{ABComplete} and
\eqref{PilModEmb} shows that \eqref{PilModEmb} can be extended into
\begin{equation}\tag*{(\ref{PilModEmb})$'$}
\begin{gathered}
\bsySig (\rr d) = \bigcap _{h >0}M(\omega _h ,\mascB )
\quad \text{and}\quad
\bsySig '(\rr d) = \bigcup _{h >0}M(\omega _h ,\mascB ),
\\[1ex]
\omega _h (x,\xi ) = e^{\frac 14(1- 2h)(|x|^2 + |\xi |^2)}.
\end{gathered}
\end{equation}
Furthermore, since for any weight $\omega \in \mascP _Q(\rr {2d})$ we have
$\omega _h \lesssim \omega$, for some $\ep >0$, \eqref{PilModEmb}$'$
gives
$$
\bsySig '(\rr d) = \bigcup _{\omega \in \mascP _Q(\rr {2d})}M(\omega ,\mascB ).
$$

\par

An other consequence of Theorem \ref{PilSpacesChar} is the following.

\par

\begin{prop}
Let $F(z)=e^{\gamma z^2}$, where $\gamma \in \mathbf C$. Then
$F=\mathfrak V_df$ for some $f\in \bsySig '(\rr d)$.
\end{prop}

\par

%For the next result, see also \cite{Pil1,Pil2}.

%\par
%
%\begin{cor}\label{PilDistSpacesHermCoeff}
%Let $s>0$ and $f\in \maclH _0'(\rr d)$ be given by \eqref{fHermite}.
%Then the following is true:
%\begin{enumerate}
%\item $f\in \bsySig _s(\rr d)$, if and only if $|c_\alpha (f)|\lesssim
%e^{-h|\alpha |^{\frac s2}}$ for every $h>0$;
%
%\vrum
%
%\item $f\in \bsycalS _{\! s}(\rr d)$, if and only if $|c_\alpha (f)|\lesssim
%e^{-h|\alpha |^{\frac s2}}$ for some $h>0$;
%
%\vrum
%
%\item $f\in \bsycalS _{\! s}'(\rr d)$, if and only if $|c_\alpha (f)|\lesssim
%e^{h|\alpha |^{\frac s2}}$ for every $h>0$;
%
%\vrum
%
%\item $f\in \bsySig _s'(\rr d)$, if and only if $|c_\alpha (f)|\lesssim
%e^{h|\alpha |^{\frac s2}}$ for some $h>0$.
%\end{enumerate}
%\end{cor}

\par

%%%%%%%%%%%%%%%%%%%%%%%%%%%%%
\section{Mapping properties of Gelfand-Shilov spaces
and their distribution spaces, under the
Bargmann transform}\label{sec5}
%%%%%%%%%%%%%%%%%%%%%%%%%%%%%

\par

In this section we discuss the image of the Bargmann transform on
Gelfand-Shilov spaces and their distribution spaces.
%A part of the analysis is based on dual properties of
%these spaces.
%In the end we apply our results to deduce
%continuity properties of Short-time Fourier transform with Gaussians as
%window functions.
We also use the results to show that the Gelfand-Shilov
spaces of functions or distributions can be obtained by appropriate
unions or intersections of certain modulation spaces, introduced in
\cite{To11}.

\par

With inspiration of the links between the spaces in \eqref{clHSpaces},
\eqref{clASpaces}, the Pilipovi{\'c} spaces and Gelfand-Shilov spaces,  we
make the following definition.

\par

\begin{defn}\label{DefGSBargmannLimits}
Let $s_1,t_1\ge \frac 12$ and $s_2,t_2> \frac 12$, and let
\begin{equation}\label{MstDef}
\omega _{r,s,t}(z) = e^{\frac {|z|^2}2+r(|x|^{\frac 1t}+|\xi |^{\frac 1s})}, \quad
z=x+i\xi ,\ x,\xi \in \rr d,
\end{equation}
when $s,t\ge \frac 12$. Then
$\maclA _{t_1}^{s_1}(\cc d)$, $\maclA _{0,t_2}^{s_2}(\cc d)$,
$(\maclA _{0,t_2}^{s_2})'(\cc d)$ and $(\maclA _{t_1}^{s_1})'(\cc d)$
are given by
\begin{equation}\label{AtsChar}
\begin{aligned}
\maclA _{t_1}^{s_1}(\cc d) &= \sets {F\in A(\cc d)}{|F(z)|\lesssim
\omega _{r,s_1,t_1}(z) \ \text{for some}\ r<0}%\label{AtsChar}
\\[1ex]
\maclA _{0,t_2}^{s_2}(\cc d) &= \sets {F\in A(\cc d)}{|F(z)|\lesssim
\omega _{r,s_2,t_2}(z)\ \text{for every}\ r<0}%\label{A0tsChar}
\\[1ex]
(\maclA _{0,t_2}^{s_2})'(\cc d) &= \sets {F\in A(\cc d)}{|F(z)|\lesssim
\omega _{r,s_2,t_2}(z) \ \text{for some}\ r >0}%\label{A0tsPrimChar}
\\[1ex]
(\maclA _{t_1}^{s_1})'(\cc d) &= \sets {F\in A(\cc d)}{|F(z)|\lesssim
\omega _{r,s_1,t_1}(z) \ \text{for every}\ r >0},%\label{AtsPrimChar}
\end{aligned}
\end{equation}
%%
%%%
%\begin{alignat}{5}
%&\maclA _{t_1}^{s_1}(\cc d), & \quad &\maclA _{0,t_2}^{s_2}(\cc d),
%& \quad &(\maclA _{0,t_2}^{s_2})'(\cc d) &
%\quad &\text{and} &\quad
%&(\maclA _{t_1}^{s_1})'(\cc d) \label{ImGSSpaces}
%\intertext{are the images of}
%&\maclS _{t_1}^{s_1}(\cc d),&\quad &\Sigma _{t_2}^{s_2}(\cc d),
%&\quad &(\Sigma _{t_2}^{s_2})'(\cc d) &
%\quad &\text{and} &\quad
%&(\maclS _{t_1}^{s_1})'(\cc d),\label{GSSpaces}
%\end{alignat}
%%%
%respectively, and the topologies of the spaces in \eqref{ImGSSpaces}
%are inherited from the corresponding spaces in \eqref{GSSpaces}.
\end{defn}

\par

We also set $\maclA _s=\maclA _s^s$ and $\maclA _{0,s}=\maclA _{0,s}^s$,
and remark that $\maclA _s = \bsycalA _s$ when $s\ge \frac 12$ and
$\maclA _{0,s} = \bsycalA _{0,s}$ when $s>\frac 12$,
since $\bsycalS _{\! s}(\rr d)=\maclS _s(\rr d)$ when $s\ge \frac 12$ and
$\bsySig _s(\rr d)=\Sigma _s(\rr d)$ when $s>\frac 12$.

\par

We note that if $p\in (0,\infty ]$, then
\begin{equation}\tag*{(\ref{AtsChar})$'$}
\begin{alignedat}{2}
\maclA _{t_1}^{s_1}(\cc d) &= \bigcup _{r>0} A^p_{(\vartheta _{r,s_1,t_1})}(\cc d),
& \quad
\maclA _{0,t_2}^{s_2}(\cc d) &= \bigcap _{r>0} A^p_{(\vartheta _{r,s_2,t_2})}(\cc d),
\\[1ex]
(\maclA _{0,t_2}^{s_2})'(\cc d) &= \bigcup _{r>0} A^p_{(1/\vartheta _{r,s_2,t_2})}(\cc d)
& \quad \text{and}\quad
(\maclA _{t_1}^{s_1})'(\cc d) &= \bigcap _{r>0} A^p_{(1/\vartheta _{r,s_1,t_1})}(\cc d),
\end{alignedat}
\end{equation}
in view of \cite[Theorem 3.2]{To11}. Here $\vartheta _{r,s,t}$ is given by
\eqref{varthetarstDef}.
We let the topologies of $\maclA _{t}^{s}(\cc d)$ and $\maclA _{0,t}^{s}(\cc d)$
be the inductive and projective limit topology, respectively, of
$A^p_{(\vartheta _{r,s,t})}(\cc d)$, $r>0$, and the topologies of
$(\maclA _{0,t}^{s})'(\cc d)$ and $(\maclA _{t}^{s})'(\cc d)$
be the inductive and projective limit topology, respectively, 
of $A^p_{(1/\vartheta _{r,s,t})}(\cc d)$, $r>0$.

\par

\begin{thm}\label{BargGSMapProp}
Let $s_1,t_1\ge \frac 12$, $s_2,t_2> \frac 12$, and let $\omega _{r,s,t}$ be
as in \eqref{MstDef}. Then the Bargmann transform from $\maclH _0(\rr d)$
to $A'_{1/2}(\cc d)$ extends uniquely to homeomorphisms from
\begin{alignat}{5}
&\maclS _{t_1}^{s_1}(\rr d),& \quad &\Sigma _{t_2}^{s_2}(\rr d),&
\quad &(\Sigma _{t_2}^{s_2})'(\rr d)& \quad &\text{and} &
\quad &(\maclS _{t_1}^{s_1})'(\rr d)\label{Eq:ThmSigmaSSpaces}
\intertext{to}
&\maclA _{t_1}^{s_1}(\rr d), & \quad &\maclA _{0,t_2}^{s_2}(\rr d), &
\quad & (\maclA _{0,t_2}^{s_2})'(\rr d) & \quad
&\text{and} & \quad &(\maclA _{t_1}^{s_1})'(\rr d),\label{Eq:ThmAnalSpaces}
\end{alignat}
respectively.
%Furthermore,
%$\maclS _{t_1}^{s_1}(\rr d)$, $\Sigma _{t_2}^{s_2}(\rr d)$,
%$(\Sigma _{t_2}^{s_2})'(\rr d)$ and $(\maclS _{t_1}^{s_1})'(\rr d)$ to
%$\maclA _{t_1}^{s_1}(\rr d)$, $\maclA _{0,t_2}^{s_2}(\rr d)$, $(\maclA _{0,t_2}^{s_2})'(\rr d)$
%and $(\maclA _{t_1}^{s_1})'(\rr d)$,respectively. Furthermore,
%\eqref{AtsChar} holds, also in topological sense.
\end{thm}

\par

\begin{proof}
Let $\phi$ be given by \eqref{phidef}. By Propositions \ref{stftGelfand2}
and \ref{stftGelfand2dist}, and \eqref{bargstft1} it follows that the Bargmann
transform maps the spaces in \eqref{Eq:ThmSigmaSSpaces}
into corresponding spaces in \eqref{Eq:ThmAnalSpaces}. Furthermore,
these mappings are homeomorphisms from the spaces in
\eqref{Eq:ThmSigmaSSpaces} to corresponding images
under the relative topologies induced from the spaces in
\eqref{Eq:ThmAnalSpaces}. We need to prove that
these images agree with the spaces in \eqref{Eq:ThmAnalSpaces}.

\par

Therefore, let $F\in (\maclA _{t_1}^{s_1})'(\cc d)$. Since
$(\maclA _{t_1}^{s_1})'(\cc d)\subseteq \maclA _{1/2}'(\cc d)$, Theorems
\ref{PilHsIdentities} and \ref{PilSpacesCharDual}, and \eqref{Eq:PilGSidentities}
shows that $F=\mathfrak V_df$ for some $f\in \maclS _{1/2}'(\rr d)$,
and since $\phi \in \maclS _t^s(\rr d)$ for every $s,t\ge \frac 12$,
\eqref{bargstft2} shows that \eqref{stftexpest2} holds with $s_1$ and
$t_1$ in place of $s$ and $t$, respectively, for every $r>0$. Hence,
$f\in (\maclS _{t_1}^{s_1})'(\rr d)$, by Proposition \ref{stftGelfand2dist}.

\par

In the same way, it is proved that the images of the other spaces in
\eqref{Eq:ThmSigmaSSpaces} agree with corresponding spaces in
\eqref{Eq:ThmAnalSpaces}, and the result follows.
\end{proof}

\par

\begin{rem}
We remark here that already in \cite{B2}, complete characterizations of the
$\mascS (\rr d)$ and $\mascS '(\rr d)$ in terms of the images under the Bargmann
transform, are obtained. In fact, it is here proved that
$\mathfrak V _d$ is bijective from $\mascS (\rr d)$ to
\begin{align*}
\maclA _\infty (\cc d) &= \sets {F\in A(\cc d)}{|F(z)|\lesssim e^{\frac {|z|^2}2}
\eabs z^{-N}\ \text{for every}\ N\ge 0},
\intertext{and from $\mascS '(\rr d)$ to}
\maclA _\infty (\cc d) &= \sets {F\in A(\cc d)}{|F(z)|\lesssim e^{\frac {|z|^2}2}
\eabs z^{N}\ \text{for some}\ N\ge 0}.
\end{align*}
\end{rem}

\par

\begin{rem}\label{extensionRemTo11}
Especially the cases when $s=s_1=\frac 12$ or $t=t_1=\frac 12$ in Theorem
\ref{BargGSMapProp} seem to
be new, and are often not taken into account in e.{\,}g. \cite{To11}.
In fact, in \cite{To11} it is usually assumed that the involved weights
should belong to $\mascP _G^0$ or the larger class
$\mascP _Q^0$.
%Here recall that $\omega \in
%\mascP _Q^0(\rr d)$ whenever
%%%
%\begin{equation}\label{PQcondClassic}
%e^{-r |x|^2}\lesssim \omega (x)\lesssim e^{r |x|^2},\quad
%\text{for every $r >0$},
%\end{equation}
%%
%which is one of the basic conditions in the definitions of the
%classes $\mascP _G^0(\rr d)$ and $\mascP _Q^0(\rr d)$.

\par

By relaxing the condition
$$
e^{-r |x|^2}\lesssim \omega (x)\lesssim e^{r |x|^2},\quad
\text{for every $r >0$},
$$
into
%%%
%\begin{equation}\label{PQcond}
%\begin{aligned}
%e^{-r |x|^2} &\lesssim \omega (x),\quad
%\text{for every $r >0$},
%\\[1ex]
%\omega (x) &\lesssim e^{r |x|^2},\quad
%\text{for some $r >0$},
%\end{aligned}
%\end{equation}
%%%
%%
\begin{equation}\label{PQcond}
e^{-r |x|^2} \lesssim \omega (x)\lesssim e^{r_0|x|^2},
\quad \text{for every $r >0$ and some $r_0>0$},
\end{equation}
in the definitions of $\mascP _Q^0(\rr d)$ and its subclasses,
it follows that the results in \cite{To11}, except Theorem
4.7 and Lemma 4.11, still hold in these more general
situations.
\end{rem}

\par

We finish the section by %giving some remarks on modulation spaces and Gelfand-Shilov spaces.
the following extension of Theorem 3.9 in \cite{To11}, and which follows
from Theorems 3.2 and 3.4 in \cite{To11}, and Theorem Proposition \ref{stftGelfand2dist}. Here
recall for $\phi$ as in \eqref{phidef}, $\omega$ a weight on $\rr {2d}$
and $\mascB $ a mixed quasi-norm space on $\rr {2d}$,
$M (\omega ,\mascB )$ is the set of all $f\in \maclS _{1/2}'(\rr d)$
such that $\nm f{M(\omega ,\mascB )} \equiv \nm {V_\phi f\cdot \omega}{\mascB}$
is finite.

\par

\begin{prop}\label{GSandMod}
Let $\mascB $ be a mixed quasi-norm space on $\rr {2d}$, and
set
$$
\omega _r(x,\xi )\equiv e^{r(|x|^{\frac 1t}+|\xi |^{\frac 1s})},\quad r\in \mathbf R.
$$
Then the following is true:
\begin{enumerate}
\item if $s,t\ge \frac 12$ then %and $\phi \in \maclS _t^s(\rr d)\setminus 0$, then
$$
\bigcup _{r>0}M (\omega _r,\mascB ) = \maclS _t^s(\rr d)
\quad \text{and}\quad
\bigcap _{r<0}M (\omega _r,\mascB ) = (\maclS _t^s)'(\rr d)
\text ;
$$

\vrum

\item if $s,t> \frac 12$, then %and $\phi \in \Sigma _t^s (\rr d)\setminus 0$, then
$$
\bigcap _{r>0}M (\omega _r,\mascB ) = \Sigma _t^s(\rr d)
\quad \text{and}\quad
\bigcup _{r<0}M (\omega _r,\mascB ) = (\Sigma _t^s)'(\rr d).
$$
\end{enumerate}
\end{prop}

\par

%%%%%%%%%%%%%%%%%%%%%%%%%%%%%%%%
\section{Some consequences and further remarks}\label{sec6}
%%%%%%%%%%%%%%%%%%%%%%%%%%%%%%%%

\par

In this section we obtain further properties of the spaces in previous sections.
First we deduce invariance properties under fractional Fourier transforms. Thereafter
we obtain invariance properties under actions of linear partial differential operators
with polynomial coefficients. Then we show some ideas on how to define
spaces like $\bsySig _t^s(\rr d)$, $\bsycalS _{\! t}^s(\rr d)$ and their duals, and in
the end we summarize the boundedness properties of the spaces in tables.

\par

%%% %
\subsection{Some remarks on fractional Fourier
transforms}\label{subsec7.1}
%%% %

\par

Next we use %present some consequences of
the previous results to deduce some invariance properties for
fractional Fourier transforms of element in certain modulation
spaces.

\par

If $f\in L^2(\rr d)$, then it is proved already in \cite{B1} that
$$
(\mathfrak V_d(\mascF f))(z) = (\mathfrak V_df)(-iz) = (\mathfrak V_df)(e^{-i \frac \pi 2}z).
$$
The same arguments show that the latter equality still holds after the assumption
$f\in L^2(\rr d)$ is relaxed into $f\in \maclH _0'(\rr d)$.
For this reason, the \emph{fractional Fourier transform}, $\mascF _rf$ of $f$
of order $r\in \mathbf R$, is defined by the formula
$$
(\mathfrak V_d(\mascF _rf))(z) = (\mathfrak V_df)(e^{-ir\frac \pi 2}z),\qquad z\in \cc d
$$
(cf. Section 3b in \cite{B1}, or \cite{ZaGa} and the references therein).
More generally, the
\emph{partial fractional Fourier transform}, $\mascF _{\mabfr}f$ of $f$ of order
$\mabfr = (r_1,\dots ,r_d)\in \rr d$ is defined by the formula
$$
(\mathfrak V_d(\mascF _{\mabfr} f))(z) = (\mathfrak V_df)
(e^{-ir_1 \frac \pi 2}z_1,\dots ,e^{-ir_d \frac \pi 2}z_d),\qquad z\in \cc d.
$$
or equivalently,
\begin{equation}\label{PartFracFT}
\begin{gathered}
\mascF _{\mabfr} \equiv \mathfrak V_d^{-1} \circ T_{\mabfr} \circ \mathfrak V_d,
\\[1ex]
\text{where}\quad
(T_{\mabfr}F)(z) \equiv  F(e^{-ir_1\frac \pi 2}z_1,\dots ,e^{-ir_d\frac \pi 2}z_d)
\end{gathered}
\end{equation}
when $F\in \bsycalA _0'(\cc d)$.

\par

Certain parts of (1) and (2) in the following proposition are well-known (cf. \cite{B1,B2}).
In order to be self-contained we here present a proof.

\par

\begin{prop}\label{propFracFT}
Let $s\in \overline{\mathbf R _\flat}$, $\mabfr\in \rr d$, $p\in (0,\infty ]$ and let
$\omega$ be a weight on $\rr {2d}$ such that
$$
\omega (x,\xi ) = \omega _0(|z_1|,\dots ,|z_d|),\quad z_j=x_j+i\xi _j\in
\mathbf C,\ j=1,\dots, d, 
$$
for some $\omega _0$. Then the following is true:
\begin{enumerate}
\item $\mascF _{\mabfr}$ is continuous and bijective on $\maclH
_s (\rr d)$, $\maclH _{0,s}(\rr d)$ and $\mascS (\rr d)$, and
their duals;

\vrum

\item $\mascF _{\mabfr}h_\alpha = e^{-i\scal {\mabfr}\alpha \frac \pi 2}h_\alpha$ for every
$\alpha \in \nn d$, and
$\mascF _{\mabfr} \circ H = H\circ \mascF _{\mabfr}$ on $\bsycalS
_{\! 0}'(\rr d)$;

\vrum

\item $\mascF _{\mabfr}$ is isometric and bijective on
$M^p_{(\omega )}(\rr d)$. Moreover,
if $N$ is an integer and $f\in \maclH _{\flat _1} '(\rr d)$, then
$$
\nm {H^N(\mascF _{\mabfr}f)}{M^p_{(\omega )}}
= \nm {H^Nf}{M^p_{(\omega )}}.
$$
\end{enumerate}
\end{prop}

\par

\begin{proof}
By \eqref{BargmannHermite} and \eqref{PartFracFT} we get
\begin{equation*}
(\mathfrak V_d(\mascF _{\mabfr}h_\alpha ))(z)= \frac {e^{-i\scal \alpha
{\mabfr}\frac \pi 2}z^\alpha } {(\alpha !)^{\frac 12}}
= e^{-i\scal \alpha {\mabfr}\frac \pi 2}(\mathfrak V_dh_\alpha )(z),
\end{equation*}
which gives $\mascF _{\mabfr}h_\alpha = e^{-i\scal {\mabfr}\alpha \frac \pi 2}
h_\alpha$ (see also \cite{B1}).

\par

The assertion (1) now follows from the latter equality,
Theorems \ref{calJChar}, \ref{calJ0Char} and \ref{PilSpacesChar},
Theorem \ref{PilSpacesChar} and duality.

\par

Let $T_{\mabfr}$ be as in \eqref{PartFracFT}, $f\in \maclH _{\flat _1} '(\rr d)$,
and let $F=\mathfrak V_df$. By applying
the Bargmann transform on $H(\mascF _{\mabfr}f)$ we get
\begin{multline*}
\mathfrak V_d(H(\mascF _{\mabfr}f))(z)
= \sum _{j=1}^d (2z_j\partial _j(T_{\mabfr}F)(z)+d(T_{\mabfr}F)(z))
\\[1ex]
= \sum _{j=1}^d (  T_{\mabfr}(2z_j\partial _jF)(z)+d(T_{\mabfr}F)(z) )
\\[1ex]
= T_{\mabfr} \left (\left (\sum _{j=1}^d (2z_j\partial _j+d)\right )F\right )(z)
= \mathfrak V_d(\mascF _{\mabfr}(Hf))(z),
\end{multline*}
and (2) follows.

\par

We only prove (3) in the case $p<\infty$. The case $p=\infty$ follows
by similar arguments and is left for the reader.

\par

By identifying $(x,\xi )\in \rr {2d}$ with $x+i\xi \in \cc d$, it follows
from the assumptions that $(T_\mabfr ^{-1}\omega )=\omega$. This gives
\begin{multline*}
\nm {\mascF _{\mabfr}f}{M^p_{(\omega )}}^p
= \nm {\mathfrak V_d(\mascF _{\mabfr}f)}{A^p_{(\omega )}}^p
\\[1ex]
=  2^d(2\pi )^{-p\frac d2}\int _{\cc d} |(T_{\mabfr}F)(z)e^{-\frac {|z|^2}2}
\omega (\sqrt 2\overline z)|
^p\, d\lambda (z)
\\[1ex]
=  2^d(2\pi )^{-p\frac d2}\int _{\cc d} |F(z)e^{-\frac {|z|^2}2}
\omega (\sqrt 2\overline z)|^p\, d\lambda (z)
\\[1ex]
= \nm {\mathfrak V_d(f)}{A^p_{(\omega )}}^p
= \nm f{M^p_{(\omega )}}^p,
\end{multline*}
and (3) follows. The proof is complete.
\end{proof}

\par

%%% %
\subsection{Distribution theory properties}\label{subsec7.2}
%%% %

\par

Next we discuss distribution theory properties of Pilipovi{\'c} spaces and their
duals, and start with the following.

\par

\begin{prop}\label{PropMultByPol}
Let $a(x,\xi )$ be a polynomial on $\rr {2d}$ and let $s\in \overline{\mathbf R_\flat}$.
Then the following is true:
\begin{enumerate}
\item the operator $a(x,D)$ on $\mascS (\rr d)$ is uniquely extendable
to a continuous operator on $\maclH _0'(\rr d)$;

\vrum

\item the operator $a(x,D)$ on $\maclH _0'(\rr d)$ restricts to a continuous
map on any of the spaces $\maclH _{0,s}(\rr d)$ and $\maclH _s(\rr d)$,
and their duals.
\end{enumerate}
\end{prop}

\par

\begin{proof}
Assume that $f$ belongs to
\begin{equation}\label{HspacesHere}
\maclH _{0,s}(\rr d)\quad \text{or}\quad \maclH _{s} (\rr d),
\end{equation}
and has the Hermite expansion \eqref{fHermite}. If $a(x,D)$ is the
creation or annihilation operator with respect to the coordinate
$x_j$, then it follows by straight-forward computations that
$$
c_\alpha (a(x,D)f) =c_\beta (f), 
$$
where
$$
\beta _k=
\begin{cases}
\alpha _j\pm 1,\quad k=j
\\[1ex]
\alpha _k\phantom{\pm 1},\quad k\neq j
\end{cases}
$$
for some choice of plus or minus. By straight-forward computations
it follows that if
$$
|c_\alpha (f)|\lesssim e^{-r|\alpha |^{\frac 1{2s}}}
\quad \text{or}\quad
|c_\alpha (f)|\lesssim r^{|\alpha |}(\alpha !)^{-\frac 12},
$$
for some or for every $r>0$, then $c_\alpha (x,D)f$ obeys
the same type of estimate as $c_\alpha (f)$ does.
Hence, the creation and annihilation operators are continuous on
any of the spaces in \eqref{HspacesHere}.

\par

The result now follows in general from the facts that any operator
$a(x,D)$ with polynomial symbol is a superposition of the operators
$f\mapsto x_j\cdot f$ and $f\mapsto D_jf$, $j=1,\dots ,d$, and that
any of the latter operators are linear combinations of creation and
annihilation operators
\end{proof}

\par

\begin{cor}
Let $s\in \overline{\mathbf R _\flat}$. Then the spaces in
\eqref{AflatIdentities1}--\eqref{AflatIdentities5} and
\eqref{AsIdentities2}--\eqref{AsIdentities8}
are invariant under multiplication of analytic polynomials and under
complex differentiations.
%If in addition $s\neq 0$, then
%the same holds true for the space \eqref{AsIdentities1}.
\end{cor}

\par

\begin{proof}
The result follows from Proposition \ref{PropMultByPol} and the fact
that the creation and annihilation operators with respect to the coordinate
$x_j$ carry over to multiplication and differentiation, respectively, with  
respect to the coordinate $z_j$, by the Bargmann transform.
\end{proof}

\par

Proposition \ref{PropMultByPol} shows that several partial differential operators
possess suitable mapping properties on Pilipovi{\'c} spaces.
On the other hand, the following consequence of Theorem \ref{PilSpacesChar}
shows that important properties, valid for other
test function spaces and their distribution spaces, are violated.
Here a test-function space $V$
and its dual $V'$ are called test-function modules if the multiplication
$(f,g)\mapsto f\cdot g$ is continuous from $V\times V$ to $V$,
and thereby extendable to continuous mappings from $V'\times V$ to
$V$ and from $V\times V'$ to $V$.

\par

\begin{prop}\label{CorLackInvariance}
Let $f$ be a Gauss function and let $s\in \overline{\mathbf R _\flat}$.
Then the following is true:
\begin{enumerate}
\item $f \in \maclH _{0,\frac 12} (\rr d)$, if and only if
$f(x)=Ce^{-\frac {|x|^2}2+L(x)}$ for some (complex-valued)
linear form $L$ on $\rr d$ and constant $C$;

\vrum

\item $\maclH _{0,s}(\rr d)$ and $\maclH _{0,s}'(\rr d)$ are neither
invariant under dilations nor test-function modules when $0<s\le \frac 12$;

\vrum

\item $\maclH _s(\rr d)$ and $\maclH _s'(\rr d)$ are neither
invariant under dilations nor test-function modules when $0\le s< \frac 12$.
\end{enumerate}
\end{prop}

\par

\begin{proof}
It suffices to prove (1). Let $\phi$ be as in \eqref{phidef}, and set
$$
\phi _{A,L} (x) = e^{-\frac 12\scal {Ax}x +L(x)},
$$
when $A$ is a symmetric $d\times d$ matrix
with positive real part, and $L$ is a linear form on $\rr d$
with values in $\mathbf C$. 
If $f=\phi _{I,L}$ for some (complex-valued)
linear form $L$ on $\rr d$, then
$$
|(V_\phi f)(x,\xi )| = C_0 e^{-\frac 14(|x-x_0|^2+|\xi -\xi _0|^2)},
$$
for some fixed $x_0,\xi _0\in \rr d$ and constant $C_0$. Since
the right-hand side can be estimated by $C_\ep
e^{-\frac 14(1-\ep )(|x|^2+|\xi |^2)}$, for every $\ep >0$, it follows that
$f \in \bsySig (\rr d)$ in this case, in view of Remark \ref{Rem:LinkSTFT}.

\par

Next assume that $f$ is a Gaussian and belongs to
$\bsySig (\rr d)$. Then $f=C\phi _{A,L}$,
for some constant $C$, symmetric $d\times d$ matrix
$A$ with positive real part, and linear form $L$ on $\rr d$
with values in $\mathbf C$. Since $\bsySig$ is invariant under
the Fourier transform, it follows that
$$
\widehat f = C_1\phi _{A^{-1},L_0}
$$
belongs to $\bsySig (\rr d)$ for some choice of linear form $L_0$. 

\par

Let
\begin{equation}\label{AjDef}
A_1 = A(I+A)^{-1}
\quad \text{and}\quad
A_2 = (I+A)^{-1}.
\end{equation}
By straight-forward computations we get
\begin{align*}
(V_{\phi}f) (x,0) &= C_1\phi _{A,L} *\phi (x) = C_2 \phi _{A_1,L_1}(x)
\intertext{and}
(V_{\phi}\widehat f) (\xi ,0) &= C_3\phi _{A^{-1},L_0} *\phi (\xi )
= C_4 \phi _{A_2,L_2}(\xi ),
\end{align*}
for some constants $C_1,\dots ,C_4$ and linear forms
$L_1$ and $L_2$. Since $f,\widehat f\in \bsySig (\rr d)$, Theorem
\ref{PilSpacesChar}, Remark \ref{Rem:LinkSTFT} and the latter
formulas imply that
$$
\operatorname{Re} (\scal {A_1x}x)\ge \frac {|x|^2}2
\quad \text{and}\quad
\operatorname{Re} (\scal {A_2x}x)\ge \frac {|x|^2}2.
$$
Hence, if $\lambda _1,\dots ,\lambda _d$ are the eigenvalues
of $A$, it follows from the latter estimates and \eqref{AjDef}
that
$$
\operatorname{Re} \big ( (1+\lambda _j )^{-1} \big ) \ge 2^{-1}
\quad \text{and}\quad
\operatorname{Re} \big ( (1+\lambda _j ^{-1} )^{-1} \big ) \ge 2^{-1},
$$
when $j=1,\dots ,d$, which is true only when $\lambda _j=1$ for every
$j$.

\par

In fact, by letting $s_j$ and $t_j$ be the real and imaginary parts of
$\lambda _j$, the latter inequalities become
$$
s_j^2+t_j^2 \le 1 \le s_j^2-3t_j^2 ,
$$
which implies that $s_j=\pm 1$ and $t_j=0$. Since $s_j>0$, it follows
that $\lambda _j=s_j=1$, and the result follows.
\end{proof}

\par

%%% %
\subsection{A broader class of Pilipovi{\'c} spaces}\label{subsec7.3}
%%% %

\par

The mapping properties explained in Theorem \ref{PilSpacesChar}
can serve as a source of inspiration of defining broad families of function
and distribution spaces, related to Pilipovi{\'c} spaces and their spaces of
distributions.

\par

In what follows, let $\kappa _1$ and $\kappa _2$ be given by
\eqref{Eq:kappaDef}, and let
%For $s$ belonging to $I$ we now set
$$
\vartheta _{r,s}(x)
=
\begin{cases}
\displaystyle{e^{r(\log \eabs x)^{\frac 1{1-2s}}}},&  s<\frac 12,
\\[1ex]
\displaystyle{e^{r|x|^{\kappa _1(\sigma )}}},&  s=\flat _\sigma ,
\\[1ex]
\displaystyle{e^{\frac {r|x|^2}{2(1+r)}}},&  s=\frac 12,
\\[1ex]
\displaystyle{e^{\frac {|x|^2}2-r|x|^{\frac 1s}}},&  s>\frac 12,
\end{cases}
\qquad 
\vartheta _{r,s}^*(x)
=
\begin{cases}
%\displaystyle{e^{r(\log \eabs x)^{\frac 1{1-2s}}}},&\text{when}\  s<\frac 12,
\displaystyle{e^{r|x|^{\kappa _2(\sigma )}}},&  s=\flat _\sigma ,
\\[1ex]
\displaystyle{e^{\frac {(1+r)|x|^2}{2r}}},& s=\frac 12,
\\[1ex]
\displaystyle{e^{\frac {|x|^2}2+r|x|^{\frac 1s}}},& s>\frac 12,
\end{cases}
$$
\begin{alignat*}{3}
\omega _{r,s,t}(z) &= \vartheta _{r,t}(x)\vartheta _{r,s}(\xi ), &
\quad &\text{when}\ s,t\in \overline{\mathbf R_\flat},\ z=x+i\xi ,\  x,\xi \in \rr d,
\intertext{and}
\omega _{r,s,t}^*(z) &= \vartheta _{r,t}^*(x)\vartheta _{r,s}^*(\xi ), &
\quad &\text{when}\ s,t\in \overline{\mathbf R_\flat},\ z=x+i\xi ,\  x,\xi \in \rr d.
\end{alignat*}
Furthermore, let
\begin{alignat*}{2}
\bsycalA _t^s(\cc d) &=
\sets {F\in A(\cc d)}{|F(z)|\lesssim \omega _{r,s,t}(z),\ \text{for some}\ r>0},
& &
\\[1ex]
\bsycalA _{0,t}^s(\cc d) &=
\sets {F\in A(\cc d)}{|F(z)|\lesssim \omega _{r,s,t}(z),\ \text{for every}\ r>0},
& &
\intertext{and let $(\bsycalA _t^s)'(\cc d)$ and $(\bsycalA _{0,t}^s)'(\cc d)$
be the duals of $\bsycalA _t^s(\cc d)$ and $\bsycalA _{0,t}^s(\cc d)$,
respectively, $s,t\in \overline{\mathbf R _\flat}$. Then}
(\bsycalA _t^s)'(\cc d) &=
\sets {F\in A(\cc d)}{|F(z)|\lesssim \omega _{r,s,t}^*(z),\ \text{for every}\ r>0},
& \ s,t&> \flat _1
\intertext{and}
(\bsycalA _{0,t}^s)'(\cc d) &=
\sets {F\in A(\cc d)}{|F(z)|\lesssim \omega _{r,s,t}^*(z),\ \text{for some}\ r>0},
& \ s,t&>\flat _1.
\end{alignat*}

\par

We now let $\bsycalS _{\! t}^s(\rr d)$, $\bsySig _t^s(\rr d)$, $(\bsycalS _{\! t}^s)'(\rr d)$ and
$(\bsySig _t^s)'(\rr d)$ be the counter images of the latter spaces, i.{\,}e.,
\begin{equation}\label{ExtPilSpaces}
\begin{alignedat}{2}
\bsycalS _{\! t}^s(\rr d) &\equiv \mathfrak V_d^{-1}(\bsycalA _t^s(\cc d)),
&\quad
\bsySig _t^s(\rr d) &\equiv \mathfrak V_d^{-1}(\bsycalA _{0,t}^s(\cc d)),
\\[1ex]
(\bsycalS _{\! t}^s)'(\rr d) &\equiv \mathfrak V_d^{-1}((\bsycalA _t^s)'(\cc d)),
&\quad
(\bsySig _t^s)'(\rr d) &\equiv \mathfrak V_d^{-1}((\bsycalA _{0,t}^s)'(\cc d)),
\end{alignedat}
\end{equation}
for suitable $s$ and $t$.

\par

Evidently, from the definitions it follows that
\begin{alignat*}{5}
\bsycalS _{\! s}^s &= \bsycalS _{\! s},&\quad
\bsySig _s^s &= \bsySig _s,&\quad
(\bsycalS _{\! s}^s)' &= \bsycalS _{\! s}' &
\quad &\text{and} &\quad
(\bsySig _s^s)' &= \bsySig _s',
\intertext{for admissible $s\in \overline{\mathbf R_\flat}$, and}
\bsycalS _{\! s}^t &= \maclS _{s}^t,&\quad
\bsySig _s^t &= \Sigma _s^t,&\quad
(\bsycalS _{\! s}^t)' &= (\maclS _{s}^t)' &
\quad &\text{and} &\quad
(\bsySig _s^t)' &= (\Sigma _s^t)',
\end{alignat*}
when the spaces on the right-hand sides are non-trivial. An interesting question
concerns of finding convenient
descriptions of the spaces in \eqref{ExtPilSpaces}. This seems to be non-trivial
from the point of view of Hermite series expansions, in view of
\cite{GraLecPilRod}.

\par

By the link between the Bargmann transform and short-time Fourier transform
\eqref{bargstft1} when $\phi$ is given by \eqref{phidef}, it follows that we may
identify the spaces in \eqref{ExtPilSpaces} by suitable estimates on the
involved elements. More precisely, let $\widetilde \vartheta _{r,s}$
and $\widetilde \vartheta _{r,s}^*$ be the modifications
of $\vartheta _{r,s}$ and $\vartheta _{r,s}^*$, given by
\begin{align}
\label{tildethetaDef}
\widetilde \vartheta _{r,s}(x)
&=
\begin{cases}
\displaystyle{e^{-\frac {|x|^2}4+r(\log \eabs x)^{\frac 1{1-2s}}}},&  s<\frac 12,
\\[1ex]
\displaystyle{e^{-\frac {|x|^2}4+r|x|^{\kappa _1(\sigma )}}},&  s=\flat _\sigma ,
\\[1ex]
\displaystyle{e^{\frac {|x|^2}{4(1+r)}}},&  s=\frac 12,
\\[1ex]
\displaystyle{e^{-r|x|^{\frac 1s}}},&  s>\frac 12,
\end{cases}
\intertext{and}
\label{tildethetaDef2}
\widetilde \vartheta _{r,s}^*(x)
&=
\begin{cases}
\displaystyle{e^{-\frac {|x|^2}4+r|x|^{\kappa _2(\sigma )}}},& 
s=\flat _\sigma ,
\\[1ex]
\displaystyle{e^{r|x|^{\frac 1s}}},& s\ge \frac 12,
\end{cases}
\end{align}
then we get the following.

\par

\begin{prop}\label{STFTPilSpChar}
Let $r>0$, $\kappa _1$ $\kappa _2$ and $\phi$ be as in \eqref{phidef} and
\eqref{Eq:kappaDef}, $\widetilde \vartheta _{r,s}$ for $s\in \overline{\mathbf R_\flat}$ and
$\widetilde \vartheta _{r,s}^*$ for
$s\in {\mathbf R_\flat}$ be given by \eqref{tildethetaDef} and \eqref{tildethetaDef2}, and let
$$
\omega _{r,s,t}(x,\xi )\equiv \widetilde \vartheta _{r,t}(x)\widetilde \vartheta _{r,s}(\xi)
\quad \text{and}\quad
\omega _{r,s,t}^*(x,\xi )\equiv \widetilde \vartheta _{r,t}^*(x)\widetilde \vartheta _{r,s}^*(\xi).
$$
Then the following is true:
\begin{enumerate}
%\item if $s,t\ge 0$, then
%$$
%f\in \bsycalS _{\! t}^s(\rr d)
%\quad \Longleftrightarrow \quad
%|(V_\phi f)(x,\xi )| \lesssim \omega _{r,s,t}(x,\xi )
%$$
%for some $r>0$;
%
\item if $s,t\ge 0$ ($s,t>0$), then $f\in \bsycalS _{\! t}^s(\rr d)$
($f\in \bsySig _{t}^s(\rr d)$), if and only if 
$$
|(V_\phi f)(x,\xi )| \lesssim \omega _{r,s,t}(x,\xi )
$$
for some (for every) $r>0$;

\vrum

%\item if $s,t> 0$, then
%$$
%f\in \bsySig _{t}^s(\rr d)
%\quad \Longleftrightarrow \quad
%|(V_\phi f)(x,\xi )| \lesssim \omega _{r,s,t}(x,\xi )
%$$
%for every $r>0$;
\item if $s,t>\flat _1$, then $f\in (\bsycalS _{\! t}^s)'(\rr d)$
($f\in (\bsySig _{t}^s)'(\rr d)$), if and only if 
$$
|(V_\phi f)(x,\xi )| \lesssim \omega _{r,s,t}^*(x,\xi )
$$
for every (for some) $r>0$.
\end{enumerate}
\end{prop}

\par

%%% %
\subsection{Tables of properties for function and distribution
spaces}\label{subsec7.4}
%%% %

\par

In the following tables we list the boundedness properties of the spaces
considered in previous sections. Here we
let $N_1= \frac N{\log N}$. The first tabular lists
the test function spaces and their properties.
%\footnote{If $\sigma =1$, then $f$ should here fulfill
%$|H^Nf|\lesssim R^{N_1} N_1^{N-N_1}$ (cf. \cite{FeGaTo2})}
\\[4ex]
\begin{savenotes}
{\footnotesize{
{
\begin{tabular}{l||c|c|c|cl}
$\maclH _s(\rr d)$/$\maclH _{0,s}(\rr d)$   & \ $|H^Nf| \lesssim $\  & \ $|c_\alpha (f)|\lesssim$
\  & \ $|\mathfrak V_df(z)|\lesssim$ 
& \textbf{References}
%\\[1ex]
%$s_1<\frac 12$, $s_2>\frac 12$ & & & &
\\[1ex]
\hline
                    & & & & 
\\
$s>\frac 12$  & $R^NN!^{2s}$
&  $e^{-R|\alpha |^{\frac 1{2s}}}$ & 
$e^{\frac {|z|^2}2-R|z|^{\frac 1{s}}}$ & 
Thm. \ref{BargGSMapProp}\footnote{See also \cite{CapRodToft}.}
\\[2ex]
$\maclH _{\frac 12}(\rr d)$ &  $R^NN!$
&  $e^{-R|\alpha |}$ & 
$e^{(1-R) \frac {|z|^2}2}$ & Thm.
\ref{BargGSMapProp}
\\[2ex]
$\maclH _{0,\frac 12}(\rr d)$ & $R^NN!$
&  $e^{-R|\alpha |}$ & 
$e^{R|z|^2}$ & Thm. \ref{PilSpacesChar}
\\[2ex]
$s=\flat _\sigma$, $\sigma <\infty$ &
See\footnote{If $\sigma =1$, then $f$ should here fulfill
$|H^Nf|\lesssim R^{N_1} N_1^{N-N_1}$ (cf. \cite{FeGaTo2}).}
&  ${R^{|\alpha |}}{({\alpha !})^{-\frac 1{2\sigma}}}$ &
$e^{R|z|^{\frac {2\sigma}{\sigma +1}}}$ & 
Thm. \ref{flatSpacesChar}
\\[2ex]
$s<\frac 12$  & $R^NN!^{2s}$
&  $e^{-R|\alpha |^{\frac 1{2s}}}$ & 
$e^{R(\log \eabs z)^{\frac 1{1-2s}}}$ & Thm.
\ref{PilSpacesChar}\footnote{See also \cite{FeGaTo2}.}
\end{tabular}}
\\[2ex]
\begin{center}
{Table 1: The test function spaces $\maclH _s(\rr d)$ and
$\maclH _{0,s}(\rr d)$. For  $\maclH _s$ ($\maclH _{0,s}$), the
above estimates should hold for some (for every) $R>0$.}
\end{center}
}}
\end{savenotes}

\vspace{0.5cm}

\begin{savenotes}
{\footnotesize{
{
\begin{tabular}{l||c|c|cl}
$\maclH _s'(\rr d)$/$\maclH _{0,s}'(\rr d)$   & \ $|c_\alpha (f)|\lesssim$
\  & \ $|\mathfrak V_df(z)|\lesssim$ 
& \textbf{References}
%\\[1ex]
%$s_1<\frac 12$, $s_2>\frac 12$ & & & &
\\[1ex]
\hline
                    & & & & 
\\
$s>\frac 12$  
&  $e^{R|\alpha |^{\frac 1{2s}}}$ & 
$e^{\frac {|z|^2}2+R|z|^{\frac 1{s}}}$ & 
Thm. \ref{BargGSMapProp}$^2$
\\[2ex]
$\maclH _{\frac 12}'(\rr d)$ 
&  $e^{R|\alpha |}$ & 
$e^{(1+R) \frac {|z|^2}2}$ & Thm.
\ref{BargGSMapProp}
\\[2ex]
$\maclH _{0,\frac 12}'(\rr d)$ 
&  $e^{R|\alpha |}$ & 
$e^{R|z|^2}$ & Thm. \ref{flatSpacesCharDual}
\\[2ex]
$s=\flat _\sigma$, $\sigma >1$ 
&  ${R^{|\alpha |}}{({\alpha !})^{\frac 1{2\sigma}}}$ &
$e^{R|z|^{\frac {2\sigma}{\sigma -1}}}$ & 
Thm. \ref{flatSpacesCharDual}
\\[2ex]
$s=\flat _\sigma$, $\sigma \le 1$ 
&  ${R^{|\alpha |}}{({\alpha !})^{\frac 1{2\sigma}}}$ &
\emph{No conditions} & 
Thm. \ref{flatSpacesCharDual}
\\[2ex]
$s<\frac 12$
&  $e^{-R|\alpha |^{\frac 1{2s}}}$ & 
\emph{No conditions} & 
Thm. \ref{PilSpacesCharDual}
\end{tabular}}
\\[2ex]
\begin{center}
{Table 2: The distribution spaces $\maclH _s'(\rr d)$ and
$\maclH _{0,s}'(\rr d)$. For  $\maclH _s'$ ($\maclH _{0,s}'$), the
above estimates should hold for every (for some) $R>0$.}
\end{center}
}}
\end{savenotes}

%\vspaces{0.5cm}
%
%
%The next tabular lists the distribution spaces and their properties.
%\\[4ex]
%{\footnotesize{
%{\begin{tabular}{l||c|c|c|l}
%\textbf{Spaces} &  & \ $|c_\alpha (f)|\lesssim$ \  & \ $|\mathfrak V_df(z)|\lesssim$ 
%& \textbf{Reference}
%\\[1ex]
%\hline
%                    & & & & 
%\\
%$\bsycalS _{\! s}'(\rr d)$, $s>\frac 12$ & $\forall \, R>0$: 
%&  $e^{R|\alpha |^{\frac 1{2s}}}$ & 
%$e^{\frac {|z|^2}2+R|z|^{\frac 1{s}}}$ & \cite{CapRodToft} or
%Thm. \ref{BargGSMapProp}
%\\[2ex]
%$\bsySig _{s}'(\rr d)$, $s>\frac 12$ & $\exists \, R>0$: 
%&  $e^{R|\alpha |^{\frac 1{2s}}}$ & 
%$e^{\frac {|z|^2}2+R|z|^{\frac 1{s}}}$ & \cite{CapRodToft} or
%Thm. \ref{BargGSMapProp}
%\\[2ex]
%$\bsycalS _{\! \frac 12}'(\rr d)$ & $\forall \, R>0$: 
%&  $e^{R|\alpha |}$ & 
%$e^{(1+R) \frac {|z|^2}2}$ & Thm. \ref{BargGSMapProp}
%\\[2ex]
%$\bsySig _{\frac 12}'(\rr d)$ & $\exists \, R>0$: 
%&  $e^{R|\alpha |}$ & 
%$e^{R|z|^2}$ & Thm. \ref{PilSpacesChar}
%\\[2ex]
%$\maclH _\flat '(\rr d)$ & $\forall \, R>0$: 
%&  ${R^{|\alpha |}}{\sqrt {\alpha !\, }}$ &  \emph{No condition} & Thm. \ref{flatSpacesChar}
%\\[2ex]
%$\maclH _{0,\flat}'(\rr d)$ & $\exists \, R>0$: 
%&  ${R^{|\alpha |}}{\sqrt {\alpha !\, }}$ &  \emph{No condition} & Thm.
%\ref{flatSpacesChar}
%\\[2ex]
%$\bsycalS _{\! s}'(\rr d)$, $s<\frac 12$ & $\forall \, R>0$: 
%&  $e^{R|\alpha |^{\frac 1{2s}}}$ & 
%\emph{No condition} & Thm. \ref{PilSpacesChar}
%\\[2ex]
%$\bsySig _{s}'(\rr d)$, $s<\frac 12$ & $\exists \, R>0$: 
%&  $e^{R|\alpha |^{\frac 1{2s}}}$ & 
%\emph{No condition} & Thm. \ref{PilSpacesChar}
%\end{tabular}}
%}}

\par

%\end{savenotes}

\end{document}